\documentclass[a4paper,11pt,english]{article}
\usepackage{babel}
\usepackage{authblk} %If no author is given "IMMEDIATE" is written in the title as a warning
\usepackage[latin1]{inputenc}

\author[1,2]{Henri Berestycki\thanks{hb@ehess.fr}}
\author[3,1]{Luca Rossi\thanks{l.rossi@uniroma1.it}} %{luca.rossi@ehess.fr}
\author[4]{Andrea Tellini\thanks{andrea.tellini@upm.es}}
\affil[1]{\small Centre d'Analyse et de Math\'ematique Sociales,
	EHESS - CNRS, 54 Boulevard Raspail, 75006 Paris, France}
\affil[2]{\small Senior Visiting
	Fellow, HKUST Jockey Club Institute for Advanced Study, Hong Kong
	University of Science and Technology, Hong Kong}
\affil[3]{\small Sapienza Università di Roma\\ Dipartimento di Matematica Guido Castelnuovo\\ P.le Aldo Moro 5\\ 00185 Roma, Italy}
\affil[4]{\small Universidad Polit\'ecnica de Madrid\\ E.T.S.I.D.I.\\ Departamento de Matem\'atica Aplicada a la Ingenier\'ia Industrial\\ Ronda de Valencia 3\\ 28012 Madrid, Spain}

\title{\textbf{Coupled reaction-diffusion equations on adjacent domains}}

\date{\today}

\usepackage{amsmath,amsfonts,amssymb}
\usepackage{yhmath,mathrsfs,yfonts,arcs}
\usepackage{amsthm}
\usepackage{multirow}
\usepackage{graphicx}
\usepackage[percent]{overpic} %To write on figures
\usepackage{paralist}
\usepackage[titletoc,title]{appendix}
\usepackage{cite}
\usepackage{leftidx}
\usepackage{soul}
\usepackage{color}
\usepackage{bbm}
\usepackage{tabu} %for command \tabulinesep = to master space in cells
\usepackage[pagewise]{lineno}
\usepackage{hyperref}
\renewcommand{\theequation}{\arabic{section}.\arabic{equation}}

\theoremstyle{plain}
\newtheorem{theorem}{Theorem}[section]
\newtheorem{proposition}[theorem]{Proposition}
\newtheorem{lemma}[theorem]{Lemma}

\theoremstyle{definition}

\newtheorem{remark}[theorem]{Remark}

\theoremstyle{remark}

\usepackage[margin=1.14in]{geometry}
\DeclareMathOperator{\supp}{supp}

\newcommand\Item[1][]{%
	\ifx\relax#1\relax  \item \else \item[#1] \fi
	\abovedisplayskip=0pt\abovedisplayshortskip=0pt~\vspace*{-\baselineskip}}

\newcommand{\field}[1] {\mathbb{#1}}
\newcommand{\N}{\field{N}}

\newcommand{\R}{\field{R}}
\newcommand{\C}{\field{C}}

\newcommand{\cBRR}{c_{\mathrm{rf}}}

\def\a{\alpha}
\def\b{\beta}
\def\e{\varepsilon}
\def\D{\Delta}
\def\d{\delta}
\def\g{\gamma}
\def\G{\Gamma}
\def\l{\lambda}

\def\n{\nu}

\def\O{\Omega}
\def\p{\partial}
\def\r{\rho}

\def\Si{\Sigma}
\def\s{\sigma}
\def\t{\theta}

\def\ov{\overline}
\def\un{\underline}
\def\ua{\uparrow}
\def\da{\downarrow}

\def\k{\kappa}

\newcommand{\mc}{\mathcal}

\makeatletter
\newcommand{\pushright}[1]{\ifmeasuring@#1\else\omit\hfill$\displaystyle#1$\fi\ignorespaces}
\newcommand{\pushleft}[1]{\ifmeasuring@#1\else\omit$\displaystyle#1$\hfill\fi\ignorespaces}
\makeatother

\begin{document}
\maketitle

\begin{abstract}
We consider a reaction-diffusion system for two densities defined in adjacent domains of $\R^N$. We treat two spatial configurations:  a cylinder and its complement, and two half-spaces. Each one of these densities is governed by a specific reaction-diffusion equation of Fisher--KPP type. The two densities interact by
an exchange  through the separating boundary.

 We study the long-time behavior of this system. We first show or characterize the existence and uniqueness of a positive steady state and that positive solutions converge to it (when it exists). Then, we establish the existence of an asymptotic speed of propagation in the directions along the interface separating the domain. Moreover, we study the qualitative dependence of this speed  with respect to various parameters of the model.
 
 In the case $N=2$, we compare such properties to those studied in \cite{BRR1, BRR2, BRR3, BRR4} for a model with a line representing a road of fast diffusion at the boundary of a half-plane. That case can be viewed as a singular limit of the problem studied here. 

 \end{abstract}



\setcounter{equation}{0}
\setcounter{figure}{0}
\section{Introduction}
\label{1section1}
In this paper we study the following system of two reaction-diffusion equations, one set in the interior of a cylinder $\O:=\R\times B_R\subset\R^N$, $N\geq 2$, and the other one in its complement: 
\begin{equation}
\label{1eq:1.1}
\left\{
	\begin{array}{ll}
		\p_t u-D\D u=g(u) & x\in\O, \quad t>0 \\
		\p_t v-d\D v=f(v) & x\in\R^N\setminus\ov \O, \quad t>0 \\
		D\, \p_n u = \nu v-\mu u & x\in\p\O, \quad t>0\\
		-d\, \p_n v = \mu u-\nu v & x\in\p\O, \quad t>0,
		\end{array}\right.
		\end{equation}
		Here and in the sequel, $B_R$ denotes the ball of radius $R$ centered at the origin of  $\R^{N-1}$, and $\partial_n$ stands for the derivative in the direction of the outward normal to $\O$. Moreover, unless otherwise specified, we will use the following notation for points of $\R^N$ $x=(x_1,y)$, where $y:=(x_2,\dots, x_N)$.

		System~\eqref{1eq:1.1} is complemented with an initial datum $(u_0,v_0)$ that, without further reference, will be always assumed to be non-negative, bounded and uniformly continuous up to the boundary of the respective domains of definition. A typical example, that we will consider in most of our problems, is a continuous pair with compact support. The constant $D>0$ represents the diffusion coefficient of $u(x,t)$, while $v(x,t)$ has a possibly different coefficient $d>0$, again constant.

The third and fourth equations of \eqref{1eq:1.1}  describe an exchange given by a balanced flux through $\p\O$: a constant fraction $\mu>0$ of the density $u$ passes from $\O$ into its complement, while a constant fraction $\nu>0$ of $v$ goes from the complement into $\O$. 
In particular, we emphasize that the two densities are related only through this flux condition, and we do not impose any continuity ($u=v$) on the boundary. 

The reaction term $g$ will be assumed to be of Fisher-KPP type, i.e., a locally Lipschitz  function, differentiable at $0$ and satisfying
\begin{equation}
\label{eq:gKPP}
g(0)=g(1)=0, \qquad 0< g(s)\leq g'(0)s \,\text{ for  $s\in(0,1)$,} \qquad g(s)<0 \,\text{ for  $s>1$},
\end{equation}
and, for some results, we will in addition require the following assumption, known as strong KPP property:
\begin{equation}
\label{eq:gstrongKPP}
s\mapsto \frac{g(s)}{s} \quad \text{ is decreasing for $s>0$.}
\end{equation}

Regarding the reaction term $f$, we will consider two possibilities:
\begin{itemize}
	\item either Fisher-KPP again with a saturation level $S>0$ possibly different from 1, i.e., we assume $f(s)$ to be a locally Lipschitz function, differentiable at $0$, satisfying
	\begin{equation}
	\label{eq:fKPP}
	f(0)=f(S)=0, \qquad 0< f(s)\leq f'(0)s \,\text{ for  $s\in(0,S)$,} \qquad f(s)<0 \,\text{ for  $s>S$}
	\end{equation}
	for some $S>0$, and, in some cases,
	\begin{equation}
	\label{eq:fstrongKPP}
	s\mapsto \frac{f(s)}{s} \quad \text{ is decreasing for $s>0$,}
	\end{equation}
	\item or we will take $f(s)=-\r s$ with $\r>0$, which amounts to put a decay (mortality) term in the complement of $\O$.
\end{itemize}

To sum up, system \eqref{1eq:1.1} describes the spatial and temporal evolution of a population or a density subject to two separate reaction-diffusion equations in two adjacent parts of the environment with transmission (exchange) conditions at the interface of the two domains. Our main goal is to analyze the effect of such an heterogeneity on the asymptotic behavior of the solutions of \eqref{1eq:1.1} when the initial datum is compactly supported.

 The motivation for studying such a system arises from the observation that some biological species or diseases diffuse or reproduce along specific directions faster than in the rest of the habitat. For example, the pine processionary caterpillar is believed to move faster on paths inside European forests (see \cite{R}), wolves in Western Canada preferentially use seismic lines (see \cite{McK}), and the early spread of HIV in the Democratic Republic of Congo was enhanced by transport networks like railways and rivers (see \cite{Faria}).

With the purpose of describing such situations, the series of works \cite{BRR1,BRR2,BRR3,BRR4} has introduced  the following \emph{road-field} model with a density $v(x_1,x_2,t)$ lying in the half-plane $\{(x_1,x_2)\in\R^2:x_2>0\}$ and a density $u(x_1,t)$ lying at its boundary:

\begin{equation}
\label{1eq:BRR2}
\left\{
\begin{array}{ll}
\p_t u-D\p^2_{x_1x_1} u=g(u)+\nu v(x_1,0,t)-\mu u & x_1\in\R, \quad t>0 \\
\p_t v-d\D v=f(v) & x_1\in\R, \quad x_2>0, \quad t>0 \\
-d\, \p_{x_2}v(x_1,0,t) = \mu u-\nu v(x_1,0,t) & x_1\in\R, \quad t>0.
\end{array}\right.
\end{equation}
It was proved in \cite{BRR1} that, under assumptions \eqref{eq:gKPP}, \eqref{eq:gstrongKPP}, \eqref{eq:fKPP} and \eqref{eq:fstrongKPP}, problem~\eqref{1eq:BRR2} admits a unique positive, bounded steady state $(U,V(x_2))$, with $U$ constant, and that there exists a quantity $\cBRR>0$ (the subindex stands for \lq\lq road-field\rq\rq) such that
\begin{enumerate}[(i)]
	\item \label{th:1.1.i} for all $c>\cBRR$, $\lim_{t\to\infty}\sup_{\substack{|x_1|\geq ct \\ x_2\geq 0}}(u,v)=(0,0)$,
	\item \label{th:1.1.ii} 
	for all $0<c<\cBRR$ and $a>0$, $\lim_{t\to\infty}\sup_{\substack{|x_1|\leq ct \\ 0\leq x_2\leq a}}|(u,v)-(U,V)|=0$.
\end{enumerate}
These two properties amount to saying that $\cBRR$ is the \emph{asymptotic speed of propagation} of problem \eqref{1eq:BRR2} in the $x_1$ direction. Moreover, this work \cite{BRR1} further studied
qualitative properties of this speed. More precisely, it showed that, if we set
\begin{equation}
\label{eq:cKPP}
c_f:=2\sqrt{df'(0)},
\end{equation}
then
\begin{equation}
\label{1eq:5.1bis}
\cBRR=c_f \quad \text{if $\frac{D}{d}\leq 2-\frac{g'(0)}{f'(0)}$,} \qquad \text{while} \qquad \cBRR>c_f \quad \text{otherwise,}
\end{equation}
and, in the latter case, the function $D\mapsto \cBRR(D)$ is increasing. Furthermore, the following limit exists and satisfies:
\begin{equation}
	\label{1eq:5.1ter}
\lim_{D\to+\infty}\frac{\cBRR(D)}{\sqrt{D}}\in(0,+\infty)
\end{equation}
(in particular $\lim_{D\to+\infty}\cBRR(D)=+\infty$).

  We recall that the quantity $c_f$ is the asymptotic speed of propagation (in any direction) for the homogeneous Fisher-KPP equation $v_t-d\D v=f(v)$ in $\R^N$, $N\geq 1$ (see \cite{F,KPP,AW}). 
 Consequently, the relations in \eqref{1eq:5.1bis} establish precisely the values of the parameters for which the heterogeneity introduced in \eqref{1eq:BRR2} by the presence of the road enhances the propagation speed. In particular, we see from \eqref{1eq:5.1bis} that 
the enhancement occurs if and only if $D>D_{\mathrm{rf}}$, where
\begin{equation} \label{1eq:Drf}
D_{\mathrm{rf}}:=\max\left\{0,d\left(2-\frac{g'(0)}{f'(0)}\right)\right\}.
\end{equation}
  Moreover,
  \eqref{1eq:5.1ter} shows that the enhancement becomes arbitrarily large when a sufficiently large value of the diffusion $D$ is considered on the road.\\

  In the light of these previous studies, it is natural to investigate systems of the type \eqref{1eq:1.1}. Indeed, 
  with respect to \eqref{1eq:BRR2}, our configuration in \eqref{1eq:1.1} with $N=2$ represents the case in which the road is replaced by a \emph{strip}. Therefore, one of the goals of this work is to study whether the properties of the road-field model studied in \cite{BRR2} still hold with a thick region of fast diffusion, and to determine the behavior for higher spatial dimensions. In particular, we prove the existence of an asymptotic speed of propagation in the direction $x_1$, and we study when an enhancement with respect to the homogeneous case takes place, as well as the behavior of the asymptotic speed of propagation as the diffusion $D$ and the radius of the cylinder vary. Our first result is the following.
  
  \begin{theorem}
  	\label{1th:3.3}
  	Assume \eqref{eq:gKPP}--\eqref{eq:fstrongKPP}. Then, there exists a unique, positive,
	bounded steady state $(U,V)$ of~\eqref{1eq:1.1} such that the solution $(u,v)$ of \eqref{1eq:1.1} starting from any initial datum $(u_0,v_0)$ not identically equal to $(0,0)$ converge to $(U,V)$ as $t\to\infty$, locally uniformly in space, in the closure of the respective domains. In addition, $(U,V)$ does not depend on $x_1$.
  	
  	Moreover, for $2\leq N\leq 5$, if $(u_0,v_0)$ has compact support, there exists $c^*>0$ satisfying:
  	\begin{enumerate}[(i)]
  		\item \label{1th:3.3.i} for all $c>c^*$, 
\begin{linenomath*}
  		\begin{equation*}
  		\lim_{t\to\infty}\sup_{\substack{|x_1|\geq ct \\ |y|\leq R, \; |\tilde y|\geq R} }\left(u(x_1,y,t), v(x_1,\tilde y,t)\right)=(0,0),
  		\end{equation*}
  	\end{linenomath*}
  		\item \label{1th:3.3.ii} for all $0<c<c^*$ and $R_1>R$,
  		\begin{linenomath*}
  		\begin{equation*}
  		\lim_{t\to\infty}\sup_{\substack{|x_1|\leq ct \\ |y|\leq R, \; R\leq |\tilde y|\leq R_1} }\left(\left|u(x_1,y,t)-U(y)\right| + \left|v(x_1,\tilde y,t)-V(\tilde y)\right|\right)=0,
  		\end{equation*}
  	\end{linenomath*}
  	\end{enumerate}	
  Furthermore, $c^*\geq c_f$, and the following qualitative properties hold:
  \begin{enumerate}[(i)]
  	\setcounter{enumi}{2}
  	\item \label{1th:3.3.iii} 
there exists a non-empty open set $\mc E\subseteq(D_{\mathrm{rf}},+\infty)$, with 
$D_{\mathrm{rf}}$ as in~\eqref{1eq:Drf},
depending on $d$, $f'(0)$, $g'(0)$, $\mu$, $\nu$, $R$, $N$, such that
\begin{linenomath*}
  		\begin{equation*}
  	  	c^*>c_f \quad \text{if and only if} \quad D\in\mc E.
  	  	\end{equation*}
    	\end{linenomath*}
  	  	In the cases $N=2,3$, it holds $\mc E=(D_{\mathrm{rf}},+\infty)$ for all the values of the parameters, whereas, for $N=4,5$, such an equality holds true if and only if $\mc E$ is connected and $\frac{g'(0)}{f'(0)}>2$.
  	\item \label{1th:3.3.iv} the function $D\mapsto c^*(D)$, with
  	the other parameters fixed, is such that $\frac{c^*(D)}{\sqrt{D}}$ converges to a positive real number as $D\to+\infty$;
  	\item \label{1th:3.3.v} if $\frac{g'(0)}{f'(0)}>2$, then 
  		\begin{equation}
  	\label{1eq:4.5}
  	\lim_{D\to 0} c^*(D)=g'(0)\sqrt{\frac{d}{g'(0)-f'(0)}}>c_f,
  	\end{equation}
  	otherwise, $\lim_{D\to 0} c^*(D)=c_f$;
  
  	\item \label{1th:3.3.vi} the function $R\mapsto c^*(R)$, with the other parameters fixed, is strictly increasing whenever $c^*>c_f$. Moreover, it satisfies:
  	\begin{equation}
  	\label{0eq:1.10}
  		\lim_{R\da 0} c^*(R)=c_f, \qquad \lim_{R\to+\infty}c^*(R)=c^*_{\infty}:=
  		\begin{cases}
  		c_f & \text{ if $\frac{D}{d}\leq 2-\frac{g'(0)}{f'(0)}$} \\
  		c_g & \text{ if $\frac{d}{D}\leq 2-\frac{f'(0)}{g'(0)}$} \\
  		c_{a} & \text{ otherwise,}
  		\end{cases}
  		\end{equation}
  		where, in analogy with \eqref{eq:cKPP}, $c_g:=2\sqrt{Dg'(0)}$ denotes the Fisher-KPP speed corresponding to the diffusion and reaction inside $\O$, while
  		\begin{equation}
  		\label{0eq:1.11}
  		c_{a}:=\frac{|Df'(0)-dg'(0)|}{\sqrt{(D-d)(f'(0)-g'(0))}}.
  		\end{equation}
  \end{enumerate}	
  \end{theorem}

Let us first briefly comment the properties in this result.
First of all, we point out that the convergence result stated in the first part of the theorem contains a Liouville-type result guaranteeing the uniqueness of positive, bounded steady states of the system.
Next, as a consequence of statement \eqref{1th:3.3.iii}, 
we infer that the spatial dimension has a direct effect on the set of values of the parameters (in particular $D$)
for which the enhancement of the asymptotic speed of propagation takes place, since, for $N=4,5$, such a set can be strictly smaller than the corresponding one for the road-field problem. The limitation $N\leq 5$ in Theorem~\ref{1th:3.3} is due to a technical point in the construction of the sub- and supersolutions that allow us to characterize $c^*$, and it is related to the singular behavior of the modified Bessel function of the second kind, as it will be apparent in Section~\ref{1subsection:c^*_2}. Nonetheless, we will be able to determine in more general situations (see Remark~\ref{1re:4.2}\eqref{1re:4.2iv}) when the asymptotic speed of propagation coincides with $c_f$, and we will prove (see Proposition~\ref{1pr:4.2bis}) that it is the case, for example, for sufficiently large $N$. This means that the cylinder looses its effect for large dimensions, which can be explained since the relative volume of  $\O$ with respect to any cylinder with radius $R'>R$  goes to $0$ as the dimension goes to $+\infty$.
We leave it as an open problem to prove the existence of the asymptotic speed of spreading and its characterization in dimensions higher than 5.

For the model \eqref{1eq:BRR2}, the asymptotic speed of propagation is a nondecreasing function of the diffusivity $D$. Here, we do not have an analogous result. As a matter of fact, we will see that, in some cases, it is non-monotone. In 
 Remark~\ref{1re:4.2}\eqref{1re:4.2ii} we give an explanation of the mechanism behind this surprising phenomenon, which, to our knowledge,  is new for such kind of systems.

Part~\eqref{1th:3.3.vi} of the previous result establishes that, as the radius $R$ of the cylinder converges to $+\infty$, the speed of propagation increases to a value which, as shown in Proposition~\ref{0pr:1.6}, is larger than or equal to the maximum of the two speeds $c_f$ and $c_g$, which are the ones associated with the equations outside and inside the cylinder, respectively. It can actually be \lq\lq anomalously" strictly larger\footnote{See the bibliographical discussion at the end of this section for the use of such a term in the context of propagation phenomena}. The limit as $R\to 0$ shows instead that the cylinder has no effect on the propagation, thus there is no direct relation between \eqref{1eq:1.1} and the road-field system \eqref{1eq:BRR2}. Nonetheless, we prove that, up to performing a singular rescaling in the exchange coefficient $\mu$, 
in the case $N=2$ we  recover the road-field speed of propagation $\cBRR$ as $R\da 0$. The precise result is the following.

\begin{theorem}
	\label{1th:5.1}
	For the problem
	\begin{equation}
	\label{1eq:5.2}
	\!\left\{\!\!\!
	\begin{array}{ll}
	\p_t u-D\D u=g(u) & x_1\in\R,\!\! \quad x_2\in(-R,R),\! \quad t>0 \\
	\p_t v-d\D v=f(v) & x_1\in\R,\!\! \quad |x_2|>R,\! \quad t>0 \\
	\pm D\, \p_{x_2}u(x_1,\pm R,t) = \nu v(x_1,\pm R,t)-\tilde\mu(R) u(x_1,\pm R,t) & x_1\in\R,\!\!  \quad t>0\\
\mp d\, \p_{x_2}v(x_1,\pm R,t)= \tilde\mu(R) u(x_1,\pm R,t)-\nu v(x_1,\pm R,t) & x_1\in\R,\!\!  \quad t>0, 
	\end{array}\right.
	\end{equation}
	where $\tilde{\mu}(R)$ is a positive function satisfying
	\begin{equation}
	\label{1eq:5.2bis}
	\lim_{R\da0}\frac{\tilde{\mu}(R)}{R}=\mu>0,
	\end{equation}
	the asymptotic speed of propagation in the $x_1$ direction, $\tilde c^*(R)$, converges, as $R\to 0$, to $\cBRR$, the asymptotic speed of propagation of problem \eqref{1eq:BRR2} in the $x_1$ direction, with $\mu$ 
	given by~\eqref{1eq:5.2bis}.
	\end{theorem} 
In addition, we will study the behavior of $\tilde c^*(R)$ under different rescalings
than~\eqref{1eq:5.2bis} (see Proposition~\ref{1pr:5.2}).

As long as the limit of $c^*$ as $R\to+\infty$ is concerned, 
we will show that it coincides with the asymptotic speed of propagation in the $x_1$ direction of the following problem posed in adjacent half-spaces:
\begin{equation}
\label{0eq:1.1}
\left\{
\begin{array}{lll}
\p_t u-D\D u=g(u) & x\in\R^{N-1}\times\R^-, & t>0 \\
\p_t v-d\D v=f(v) & x\in\R^{N-1}\times\R^+, & t>0 \\
D\, \p_{x_N} u = \nu v-\mu u & x\in\R^{N-1}\times\{0\}, & t>0\\
-d\, \p_{x_N} v = \mu u-\nu v & x\in\R^{N-1}\times\{0\}, & t>0
\end{array}\right.
\end{equation}
(actually, due to the symmetries of such a problem, the speed of propagation is the same in any direction satisfying $x_N=0$). This establishes a continuous dependence of the asymptotic speed of propagation or problem \eqref{1eq:1.1} with respect to the domain. Indeed, one can think one point of the boundary $\p\O$ to be fixed; thus, as $R\to+\infty$ and the curvature of the cylinder converges to $0$, one part of the boundary approaches the hyperplane which separates the half-spaces, while the rest of the boundary disappears at $+\infty$. The results concerning the asymptotic speed of propagation for adjacent half-spaces are presented in Section~\ref{0section1} below.

In the final part of the paper (Section~\ref{section4}), we analyze the case in which the environment
outside the cylinder~$\O$ is deleterious 
for the species. Namely, the reaction $f$ is a negative, linear decay (mortality) term, i.e.,
\begin{equation}
\label{2eq:4.0} 
\left\{
\begin{array}{ll}
\p_t u-D\D u=g(u) & x\in\O, \quad t>0 \\
\p_t v-d\D v=-\r v & x\in\R^N\setminus\overline{\O}, \quad t>0 \\
D\, \p_n u = \nu v-\mu u & x\in\p\O, \quad t>0\\
-d\, \p_n v = \mu u-\nu v & x\in\p\O, \quad t>0.
\end{array}\right.
\end{equation}
An example of situation where this kind of model is relevant is 
the propagation of Scentless Chamomile in the region of Saskatchewan in Western Canada, which, being transported by agricultural vehicles, spreads faster along roads, but  faces the competition of hostile weeds in the surrounding fields  (see \cite{dCL}).
Once again, we want to compare the situation in which $u$ occupies a thick domain with the analogous road-field version in the 
plane, which has been considered in \cite{BRR4}. First, we state a summary of our results - we send the reader to Section~\ref{section4} for a complete perspective -, and then we compare them with the road-field case.

\begin{theorem}
	\label{2th:section4.1}
	Assume \eqref{eq:gKPP} and \eqref{eq:gstrongKPP}. Then, there exists a (unique) positive steady state of \eqref{2eq:4.0} if and only if
	\begin{equation}
	\label{2eq:condnecsuff}
	\frac{g'(0)}{D}>\b_0^2,
	\end{equation}
	where $\b_0^2=\b_0^2(D,d,N,R,\mu,\nu,\r)$ is the principal eigenvalue of the Robin eigenvalue problem \eqref{2eq:linearD}. If this is the case, the solution of \eqref{2eq:4.0} starting from any $(u_0,v_0)\not\equiv(0,0)$ converges locally uniformly in space, in the closure of the respective domains, to such a steady state as $t\to+\infty$. Otherwise, the solution converges to $(0,0)$, uniformly in space in the closure of each domain, as $t\to+\infty$.
	
	Moreover, in the former case, assuming that $(u_0,v_0)$ has compact support and, in addition, that $N\leq 5$ if $D<d$, there exists an asymptotic speed of propagation in the $x_1$ direction, which will be denoted by $c^*_m$.
\end{theorem}
In \cite[Theorems 2.5 and 2.6]{BRR4} a necessary and sufficient condition for invasion was derived
also for the road-field analogue of~\eqref{2eq:4.0}. 
The condition there does not depend on $D$ (neither on $R$, which is not present, nor on $N$, which is fixed to be equal to 2) 
in contrast with the condition~\eqref{2eq:condnecsuff} established in the present work.
We also study here the validity of~\eqref{2eq:condnecsuff} as the
parameters $R$ and $N$ vary, as well as 
 the qualitative behavior of the asymptotic speed $c^*_m$. The results are summarized in the following theorem.

\begin{theorem}
	\label{2th:section4.2}
	\begin{enumerate}[(i)]
		\item \label{2th:section4.2.i} Condition \eqref{2eq:condnecsuff} does not hold for sufficiently large $N$. As a consequence, under the assumptions of Theorem~\ref{2th:section4.1}, the solution of \eqref{2eq:4.0} starting with a non-negative initial datum $(u_0,v_0)$ satisfies $\lim_{t\to+\infty}(u,v)=(0,0)$, uniformly in the closure of each domain.
		\item \label{2th:section4.2.ii} There exists $R_0>0$ such that \eqref{2eq:condnecsuff} holds if and only if $R>R_0$. Moreover, the function $R\mapsto c^*_m(R)$ is increasing and satisfies
		\begin{equation}
		\label{1eq:1.13}
		\lim_{R\da R_0} c^*_m(R)=0, \qquad \lim_{R\to+\infty}c^*_m(R)=c^*_{m,\infty}:=
		\begin{cases}
		c_g & \text{ if $\frac{d}{D}\leq 2+\frac{\r}{g'(0)}$} \\
		c_{m,a} & \text{ otherwise,}
		\end{cases}
		\end{equation}
		where
		\begin{equation}
		\label{1eq:1.14}
		c_{m,a}:=\frac{D\r+dg'(0)}{\sqrt{(d-D)(\r+g'(0))}}.
		\end{equation}
		\item \label{2th:section4.2.iii} If  
		\begin{equation}
		\label{2eq:6.2}
		\frac{\mu(N-1)}{R\, g'(0)}\leq 1+\frac{\nu}{\sqrt{d\r}}\frac{K_{\tau}\left(\sqrt{\frac{\rho}{d}}R\right)}{K_{\tau+1}\left(\sqrt{\frac{\rho}{d}}R\right)},
		\end{equation}
		where $K_\tau$ is the modified Bessel function of the second kind, then \eqref{2eq:condnecsuff} holds true for every $D>0$. Moreover,
		\begin{itemize}
			\item \label{2th:6.2iv} if strict inequality holds true in \eqref{2eq:6.2}, then $\lim_{D\to+\infty}c^*_m(D)/\sqrt{D}\in(0,+\infty)$;
			\item \label{2th:6.2v} if equality holds true in \eqref{2eq:6.2}, then $\limsup_{D\to+\infty}c^*_m(D)<+\infty$.
		\end{itemize}
		Otherwise, if \eqref{2eq:6.2} does not hold true, there exists $D_0\in(0,\infty)$ such that \eqref{2eq:condnecsuff} holds true if and only if $D<D_0$, and $\lim_{D\ua D_0}c^*_m(D)=0$.
	\end{enumerate}
\end{theorem}
Part \eqref{2th:section4.2.i} of the above result is a counterpart for system~\eqref{2eq:4.0}
 of the fact that, for large~$N$, the cylinder does not enhance the propagation speed of \eqref{1eq:1.1}, since its effect becomes negligible with respect to the complement. The difference here is that, since the external domain is hostile, the good environment provided by $\O$ cannot prevent the densities to go extinct.

Along the same line, part \eqref{2th:section4.2.ii} establishes that, all the other parameters being fixed, the densities can survive only if the good environment has a sufficiently large section. In addition, in analogy with problem \eqref{1eq:1.1}, the larger the section of $\O$, the larger the asymptotic speed of propagation, and the limit as $R\to+\infty$ will be characterized as the asymptotic speed of spreading of the problem corresponding to \eqref{2eq:4.0} in two adjacent half-spaces, which will be also studied in Section \ref{section4}.

To conclude the analysis of our results, we observe from part \eqref{2th:section4.2.iii} that, on the one hand, when \eqref{2eq:6.2} does not hold true, we obtain, for $D< D_0$, $D\sim D_0$, another example in which the asymptotic speed of propagation is not monotone in $D$. On the other hand, we see that  the role of $D$ is more subtle, even in the case in which $c^*_m(D)$ is defined for all $D$,  since the asymptotic speed of propagation can be bounded as $D\to+\infty$, which was not the case in \eqref{1eq:1.1} and in all the road-field problems treated in the literature before.

Once presented the results of this paper, we conclude the introduction by relating them with those of other previous works.

\vspace{0.2cm} \noindent
\textbf{Related works and further comments.}
The notation $c_a$ used in \eqref{0eq:1.11} stands for \emph{anomalous speed}. Such a name has been introduced in \cite{WLL} in  the context of some cooperative systems of equations set in the same domain (see also \cite{H} for a system of two coupled Fisher-KPP equations again in the same domain) and reflects the fact that such speed of propagation is greater than both of the Fisher-KPP speeds at which each density would invade its domain if it was isolated - incidentally, this will be also apparent from the construction of $c^*_\infty$ in Section~\ref{0section:2.2} (see Proposition~\ref{0pr:1.6}\eqref{0pr:1.6.0}). 
Analogously, we have used the notation $c^*_{m,a}$ in \eqref{1eq:1.13}-\eqref{1eq:1.14} for the anomalous speed of propagation  arising in
the limit $R\to+\infty$ for problem \eqref{2eq:4.0}.

It turns also out that $c^*_\infty$ coincides with the asymptotic speed of propagation of another system, recently studied in \cite{MBC}, describing the evolution of two densities that represent two parts of a population with different phenotypes (see also \cite{G18} for the case of $N$ densities). Such subpopulations are assumed to compete 
between each other, and to diffuse and grow each with its own rate. 
A difference with \eqref{1eq:1.1} is that in \cite{MBC} the two densities share the same environment and can mutate at every point from one type to the other, with certain rates.
The mutation plays the same role as the exchange condition in our system, which, however, 
is only concentrated on the interface between the two domains. For this reason, the fact the asymptotic speeds for the two models coincide is far from being obvious.

Passing to the work \cite{BN}, as an application of the main result,  a unique reaction-diffusion equation has been considered in the plane, with Fisher-KPP reaction $f$ and  discontinuous diffusion coefficients: $d$ in the upper half-plane and $D$ in the lower one. At least formally, this corresponds to consider \eqref{0eq:1.1} for $N=2$ and let $\mu=\nu\to+\infty$, even if the rigorous treatment of such a limit and the relation between the two problem is a very interesting open question which deserves further work. In \cite{BN}, the authors have determined the asymptotic speed of propagation in every direction and, in particular,  they have proved that, in the $x_1$ direction, it coincides with the maximum of the two Fisher-KPP speeds $c_f$ and $c_g$, as it happens for the asymptotic speed of \eqref{0eq:1.1}, $c^*_\infty$, if $f'(0)=g'(0)$, as it can be seen from the diagram in Figure~\ref{0fig:1.1}.

Regarding the result of Theorem~\ref{1th:5.1}, in \cite{LW}, the authors have considered a similar situation with a strip whose thickness goes to $0$ and, by performing a singular rescaling to the diffusivity $D$, they have obtained a problem with effective boundary conditions on a line, which exhibits an enhanced speed of propagation with respect to the Fisher-KPP one. However, the problem in \cite{LW} is essentially different from the one here, since the model there has only one density and discontinuous diffusion coefficients. Moreover, the speed of propagation of the limiting problem is not related to the speed $\cBRR$ of \eqref{1eq:BRR2}, as it is the case in Theorem~\ref{1th:5.1}, where the singular rescaling is performed on $\mu$ instead.

As already commented above, the limitation in the spatial dimension that arises in Theorems~\ref{1th:3.3} and~\ref{2th:section4.1}  is due to the singularity at $0$ of the modified Bessel functions of the second kind $K_\tau$,  that we use because the domain has axial symmetry with an unbounded component. This was not the case in \cite{RTV}, where a Fisher-KPP reaction inside a cylinder in $\R^N$, with a different diffusion on the boundary has been studied. In such work, since the section of the domain was bounded, it was not necessary to use the functions $K\tau$, and a precise asymptotic speed of propagation has been determined for all $N\geq 2$. We believe that the limitation here is only technical and that it can be overcome either by constructing other supersolutions which allow one to continue the curves defined in Sections~\ref{1subsection:c^*_2} and~\ref{2section4}, in the spirit of the proof of Proposition~\ref{0pr:1.5}, or, alternatively, by using completely different approaches. 

The first mortality condition in the context of road-field models has been introduced in \cite{T16}, where a system like \eqref{1eq:BRR2}, with $g(u)\equiv 0$ and the field consisting of the strip~$\R\times(0,R)$ has been studied, obtaining a necessary and sufficient condition for invasion to occur. However, differently from \eqref{2eq:4.0} and \cite{BRR4}, the mortality condition in \cite{T16} was not modeled through the reaction term, but through the boundary condition $v=0$ at $x_2=R$.

Finally, we mention other works that treat systems of reaction-diffusion equations posed in a domain and on its boundary, in the vein of road-field models, but with a bounded domain: \cite{MCV} is focused on the phenomenon of the formation of Turing patterns in the interior and/or on the boundary of the domain, while \cite{FLT} considers a nonlinear coupling between the domain and its boundary, and proves the exponential convergence of the solution to the equilibrium.

\vspace{0.2cm} \noindent
\textbf{Structure of the paper.}
The paper is organized as follows: in Section~\ref{0section1} we study \eqref{0eq:1.1}, i.e., the case of two adjacent half-spaces with Fisher-KPP nonlinearities in both of them; Section~\ref{section3} is devoted to our main problem \eqref{1eq:1.1}, i.e., we consider the same type of nonlinearities in a cylinder and its complement. Later, in Section~\ref{section4}, we treat the  mortality case for $v$ in both configurations of the domain. In Appendix~\ref{sectionappA}, we recall some facts on Bessel functions that arise in the study of the cylindrical case, and we prove some results related to them which are of independent interest. In Appendix~\ref{sectionappB}, instead, we recall how the issues of existence, uniqueness and regularity for the kind of parabolic problems considered in our constructions can be handled. In particular, we focus on the question of how to deal with merely uniformly continuous initial data that do not satisfy the transmission condition at the interface of the adjacent domains.

\setcounter{equation}{0}
\setcounter{equation}{0}
\section{Two adjacent half-spaces}
\label{0section1}
For the sake of clearness in the exposition, we start by studying model~\eqref{0eq:1.1} about adjacent subspaces. 
Indeed, for such a problem we will not need to introduce the Bessel functions, which will instead appear in the treatment of the cylindrical domain. Moreover, the construction itself of the asymptotic speed of propagation for problem \eqref{0eq:1.1} will be later needed to derive some of the qualitative properties of its 
counterpart \eqref{1eq:1.1} in the~cylinder.

In the first part of the section we study the long-time behavior of the solutions of \eqref{0eq:1.1}, 
while in the second part we determine the asymptotic speed of propagation and prove some of its properties.
The functions $g$ and  $f$ will both be assumed to be of the Fisher-KPP type, 
i.e.~satisfying \eqref{eq:gKPP} and \eqref{eq:fKPP}.

\subsection{Long-time behavior for system \eqref{0eq:1.1}}
\label{section2.1}
	In order to determine the asymptotic behavior of the solutions as $t\to+\infty$, we will use some comparison principles, which can be shown to hold true with the same techniques of \cite[Proposition 3.2]{BRR1}. The key point that lies behind these principles is the fact that the systems under consideration are cooperative. The weak comparison principle establishes that, if a subsolution lies below, in the large sense, a supersolution of the evolution problem at initial time, the same order is maintained for all times. The strong version ensures in addition that, if they touch at a point at some positive time, they 
	must coincide everywhere for all previous times.
	
	Apart from classical sub- and supersolutions, i.e., pairs that satisfy the systems with equality replaced, respectively, by $\lq\lq\leq"$ and $\lq\lq\geq"$, the comparison principles hold true, under additional assumptions,  for the case of \emph{generalized subsolutions}, which are constructed as the supremum of two subsolutions in a fixed open set, are extended outside it with the value of the one that is larger at the boundary of such a set, and satisfy an additional condition on the boundary that separates the two densities. We refer the reader to \cite[Proposition 3.3]{BRR1} for a general precise statement, and to \cite[Proposition 2.4]{RTV} for the specific case in which we use it here.
	
	These comparison principles, together with the local boundary estimates given by \cite[Theorem 5.19]{L}, allow one to show, with the same arguments of \cite[Appendix A]{BRR1} and \cite[Section 2]{BRR2}, that problem \eqref{0eq:1.1} complemented with a bounded and uniformly continuous initial datum is well posed, and that its solution is regular, in the sense that both $u$ and $v$ are regular up to the boundary of their respective domains. We point out that the initial datum is not required to satisfy any compatibility condition at the interface of the adjacent domains of our systems, as we show in Appendix \ref{sectionappB}.
	
	Here, by using the comparison principles, we will prove existence, uniqueness and some symmetries of positive steady states of \eqref{0eq:1.1}, i.e., solutions  which do not depend on $t$. Such steady state is important since it controls the asymptotic behavior of the solutions of the evolutionary problem. We start with the following result.

\begin{proposition}
	\label{0pr:1.1}
	Assume \eqref{eq:gKPP} and \eqref{eq:fKPP}. Then, for every $(u_0,v_0)\not\equiv(0,0)$, there exist two positive, bounded stationary solutions $(U_i,V_i)$ of \eqref{0eq:1.1}, $i\in\{1,2\}$, which only depend on $x_N$ and such that the solution of \eqref{0eq:1.1} starting from $(u_0,v_0)$ satisfies
	\begin{equation}
	\label{0eq:1.2}
	(U_1,V_1)\leq\liminf_{t\to+\infty}(u,v)\leq\limsup_{t\to+\infty}(u,v)\leq(U_2,V_2)
	\end{equation}
	locally uniformly in space, in the closure of the corresponding domains.
	\begin{proof}
		The last conditions in \eqref{eq:gKPP} and \eqref{eq:fKPP} imply that, setting
		\begin{equation}
		\label{0eq:1.2bis}
		\ov{u}:=\max\left\{1, \|u_0\|_{\infty},\frac{\nu}{\mu}S, \frac{\nu}{\mu}\|v_0\|_{\infty}\right\}, \qquad \ov{v}:=\frac{\mu}{\nu}\ov{u},
		\end{equation}
		the pair $(\ov u,\ov v)$ is a positive stationary supersolution of \eqref{0eq:1.1}. Thanks to the parabolic strong comparison principle, the solution of \eqref{0eq:1.1} starting from $(\ov u,\ov v)$ is decreasing in time, and - being positive, thus bounded - the local boundary estimates of \cite[Theorem 5.19]{L} guarantee that it converges to a stationary state $(U_2,V_2)$ locally uniformly in space, in the closure of the respective domains, as $t\to+\infty$. Moreover, since $(\ov u,\ov v)$ lies above $(u_0,v_0)$, the solution of \eqref{0eq:1.1} with $(u_0,v_0)$ as an initial datum maintains the same order for all $t>0$. Taking the $\limsup$ as $t\to+\infty$, the last inequality in \eqref{0eq:1.2} is therefore proved. 
		By the invariance of \eqref{0eq:1.1} and of the initial datum $(\ov u,\ov v)$ with respect to translations by $x_i$, $i\in\{1,\ldots,N-1\}$, the same property holds true for the unique solution of \eqref{0eq:1.1} and therefore for $(U_2,V_2)$.

		In order to find the other steady state, we perform quite a classical construction of a stationary subsolution with compact support. To this end, consider the eigenvalue problem
		\begin{linenomath*}
		\begin{equation}
		\label{0eq:1.3}
		\left\{
		\begin{array}{ll}
		-\D \phi=\l \phi & \text{ in $B^L:=B((0,\ldots,0,-(L+1)),L)$} \\
		\phi=0 & \text{ on $\p B^L$,}
		\end{array}\right.
		\end{equation}
	\end{linenomath*}
		where $B((0,\ldots,0,-(L+1)),L)$ denotes the ball of radius $L>0$ and center $(0,\ldots,0,-(L+1))\in\{x\in\R^N: x_N<0\}$. It is well known that \eqref{0eq:1.3} admits a smallest eigenvalue $\l=\l_1(L)$, which is positive and to which  a positive eigenfunction $\phi_1$, that will be normalized so that $\max_{B^L}\phi_1=1$, is associated. Moreover, it is well known that $\l_1(L)\to 0$ as $L\to+\infty$; thus, if we take a sufficiently large $L$ so that $\l_1(L)<\frac{g'(0)}{D}$, the regularity of $g$ ensures that there exists $\e_0>0$ such that $-D\D(\e\phi_1)<g(\e\phi_1)$ in $B^L$ for all $\e\in(0,\e_0)$, and the pair $\e(\un u,0)$, where
		\begin{linenomath*}
		\begin{equation*}
		\un u:=\left\{
		\begin{array}{ll}
		\phi_1(x) & \text{ for $x\in B^L$} \\
		0 & \text{ for $x\in\{x_N<0\}\setminus B^L$,}
		\end{array}\right.
		\end{equation*}
	\end{linenomath*}
		 is a strict non-negative, stationary subsolution of \eqref{0eq:1.1}, again for all $\e\in(0,\e_0)$. Reasoning as above, we have that the solution of \eqref{0eq:1.1} starting from $\e(\un u,0)$ converges locally uniformly in space, now strictly increasing in time, towards a positive steady state $(U_1,V_1)$ as $t\to+\infty$, in the closure of the corresponding domains. Moreover, up to reducing $\e_0$ if necessary, we can assume that $\e(\un u,0)$ lies strictly below the solution $(u,v)$ evaluated at  $t=1$. The comparison principle then ensures that the order is preserved for all positive times, and, by taking the $\liminf$ as $t\to+\infty$, we obtain the first inequality in \eqref{0eq:1.2}.
		
		To conclude, we need to prove that $(U_1,V_1)$ is $x_i$-independent for every $i\in\{1,\ldots,N-1\}$. Since the domain and the system are invariant by rotations around the axis $x_N$, it suffices to prove the invariance with respect to $x_1$. As $\e(\un u,0)$ has compact support and the solution starting from it is increasing in time, we have that it lies strictly below $(U_1,V_1)$. Thus, that there exists $h_0>0$, $h_0\sim 0$, such that $\e(\un{u}(x_1-h,x_2,\dots,x_N),0)$, which is still a subsolution of \eqref{0eq:1.1} by the invariance of the problem by translations in $x_1$, lies below $(U_1(x_1,x_2,\dots,x_N),V_1(x_1,x_2,\dots,\tilde x_N))$ for all $h\in(-h_0,h_0)$, $x_1,\dots,x_{N-1}\in\R$, $x_N<0$ and $\tilde x_N>0$. From the uniqueness of the Cauchy problem associated to \eqref{0eq:1.1}, we obtain that the solution of the system, with the translated subsolution as initial datum, converges to the corresponding translation of $(U_1,V_1)$ as $t\to+\infty$. Thus, by comparison, $(U_1,V_1)$ is smaller than small translations in the $x_1$ direction of itself, which proves that it does not depend on $x_1$.
	\end{proof}
\end{proposition}

Thanks to the previous proposition, we are lead to consider bounded steady states which depend only on $x_N$, for which we obtain the following result.

\begin{proposition}
	\label{0pr:1.2}
	Assume \eqref{eq:gKPP} and \eqref{eq:fKPP}. Then, there exists a unique positive, bounded steady state of problem \eqref{0eq:1.1} that does not depend  on $x_1,\ldots,x_{N-1}$. Such a steady state will be denoted by $(U,V)$.
	\end{proposition}

\begin{proof}
	The existence part has been proved in Proposition~\ref{0pr:1.1}, thus we focus on uniqueness. Using the symmetries of the steady states that we are considering, \eqref{0eq:1.1} reduces to the following system of ordinary differential equations
	\begin{linenomath*}
	\begin{equation}
	\label{0eq:1.5}
	\left\{
	\begin{array}{ll}
	-D U''(y)=g(U(y)) & y<0\\
	-d V''(y)=f(V(y)) & y>0 \\
	D\, U'(0) = \nu V(0)-\mu U(0) & \\
	-d\, V'(0) = \mu U(0)-\nu V(0), & 
	\end{array}\right.
	\end{equation}
\end{linenomath*}
	for which we will show that there exists at most one bounded positive solution. Thanks to the assumptions on $f$ and $g$, the origin is a center in the phase plane associated to each of the differential equations, thus any positive bounded solution must satisfy
	\begin{linenomath*}
	\begin{equation*}
	U(-\infty)=1, \quad U'(-\infty)=0, \quad V(+\infty)=S, \quad V'(+\infty)=0.
	\end{equation*}
\end{linenomath*}
	We will now prove that one of the following mutually exclusive possibility occurs in the whole domain of definition of the functions:
	\begin{enumerate}[i)]
		\item $U$ and $V$ are both decreasing, $0<U<1$ and $V>S$,
		\item $U$ and $V$ are both constant, $U=1$ and $V=S$,
		\item $U$ and $V$ are both increasing, $U>1$ and $0<V<S$.
	\end{enumerate}
	Indeed, we distinguish three cases according to the value of $U(0)$: $0<U(0)<1$,  $U(0)=1$, or $U(0)>1$. 
	
	In the first case, i.e., if $0<U(0)<1$, assume by contradiction there exists $y_0<0$ for which $U(y_0)\geq 1$, then, since $U(-\infty)=1$, $U$ will have a maximum value equal to  or greater than $1$. Such a maximum cannot be equal to $1$, because, at any maximum point $y_M$, $U'(y_M)=0$, and the uniqueness of the Cauchy problem associated to the differential equation for $U$ will imply $U\equiv 1$, against  our assumption on $U(0)$. If the maximum was greater than $1$, then, at any maximum point,
		\[0\geq D U''(y_M)=-g(U(y_M))>0,\]
		again a contradiction. Once we know that $0<U<1$ for every $y<0$, the sign of $g$ entails that $g$ is concave and, since $U'(-\infty)=0$, $U$ is decreasing.
	Moreover, the third and fourth equations of \eqref{0eq:1.5} give that $dV'(0)=DU'(0)<0$, thus, for $V$ not to vanish somewhere, the sign of $f$ forces $V(0)>S$ and, by reasoning as we have done for $U$, it follows that $V>S$ everywhere, and that it is decreasing. Hence, we have shown that, if $0<U(0)<1$, possibility (i) occurs.
		
		Analogously, it is possible to show that, if $U(0)=1$, possibility (ii) occurs, while, if $U(0)>1$, possibility (iii) occurs.
	
	To conclude the proof, we show that, in any of the three cases, system \eqref{0eq:1.5} admits at most one solution. Let us start with the first case:
	by multiplying the differential equations of \eqref{0eq:1.5}  by $U'$ and $V'$ respectively and then integrating, we obtain that the quantities
	\begin{equation}
		\label{0eq:1.6bis}
		\frac{1}{2}U'(y)^2+\frac{1}{D}\int_{U(0)}^{U(y)}g(s)\,ds \quad \text{ and } \quad \frac{1}{2}V'(y)^2+\frac{1}{d}\int_{V(0)}^{V(y)}f(s)\,ds
		\end{equation}
	are constant for every $y$ and, evaluating them at the extrema of their respective domains, we get the following identities
	\begin{equation}
		\label{0eq:1.6ter}
	\int_{U(0)}^{1}g(s)\,ds=\frac{D}{2}U'(0)^2, \qquad  \int_{V(0)}^{S}f(s)\,ds=\frac{d}{2}V'(0)^2.
	\end{equation}
	The left-hand side of the first equation is decreasing with respect to $U(0)$, thus so is $U'(0)^2$ and, as
	\begin{equation}
	\label{0eq:1.7}
	\left\{
	\begin{array}{l}
	d\,V'(0)=D\, U'(0)  \\
	\nu V(0)=\mu U(0)+dV'(0)=\mu U(0)+DU'(0),
	\end{array}\right.
	\end{equation}
	$V'(0)^2$ is decreasing with respect to $U(0)$. Since $V'(0)$ is negative, it is an increasing function of $U(0)$, and the second relation of \eqref{0eq:1.7} gives that $V(0)$ also increases with $U(0)$. Finally, we observe that $\int_{V(0)}^Sf(s)\,ds=\int_S^{V(0)}(-f(s))\,ds$, thus these integrals are increasing with respect to $U(0)$. These considerations entail that the second relation in \eqref{0eq:1.6ter} has at most one solution as $U(0)$ varies in $(0,1)$.
	
	In the second case it is easy to see that, since $S$ and $1$ are the unique positive zeros of $f$ and $g$ respectively, the unique solution of \eqref{0eq:1.5} is $(U,V)\equiv (1,S)$.
	
	Finally, the third case can be treated like the first one, by interchanging the roles of $U$ and $V$ and showing that the first relation in \eqref{0eq:1.6ter} has at most one solution as $V(0)$ varies in $(0,S)$.
\end{proof}

By combining the two previous results, we get the following long-time behavior for system \eqref{0eq:1.1}.

\begin{theorem}
	\label{0th:1.3}
	Assume \eqref{eq:gKPP} and \eqref{eq:fKPP}. Then, for every $(u_0,v_0)\not\equiv(0,0)$, the solution of \eqref{0eq:1.1} starting from $(u_0,v_0)$ satisfies
	\begin{linenomath*}
	\begin{equation*}
		\lim_{t\to+\infty}(u,v)=(U,V)
	\end{equation*}
\end{linenomath*}
	locally uniformly in space in the closure of each domain, where $(U,V)$ is the steady state given by Proposition~\ref{0pr:1.2}.
\end{theorem}

\subsection{Asymptotic speed of propagation for \eqref{0eq:1.1}}
\label{0section:2.2}
In this section we construct super- and subsolutions to \eqref{0eq:1.1} moving with certain speeds that will provide us with upper and lower bounds, respectively, for the asymptotic speed of propagation for problem \eqref{0eq:1.1} in the $x_i$ direction, $i\in\{1,\ldots,N-1\}$. Thanks to the symmetry of the problem under rotations around the axis $x_N$, it will suffice to consider the positive $x_1$ direction; accordingly, we will detail all the constructions in such a case.
 
 In order to obtain supersolutions, we consider the linearization of \eqref{0eq:1.1} around $(0,0)$, which reads
\begin{equation}
	\label{0eq:1.8}
	\left\{
	\begin{array}{lll}
		\p_t u-D\D u=g'(0)u & x\in\R^{N-1}\times\R^-, & t>0 \\
		\p_t v-d\D v=f'(0)v & x\in\R^{N-1}\times\R^+, & t>0 \\
		D\, \p_{x_N} u = \nu v-\mu u & x\in\R^{N-1}\times\{0\}, & t>0 \\
		-d\, \p_{x_N} v = \mu u-\nu v & x\in\R^{N-1}\times\{0\}, & t>0,
	\end{array}\right.
\end{equation}
and we look for the values of $c$ for which such a system admits (super)solutions  of the type
\begin{equation}
	\label{0eq:1.9}
(\ov u,\ov v)=e^{-\a(x_1- ct)}\left(1,\frac{\mu}{\nu}\right),
\end{equation}
where $\a$ is a positive constant to be chosen. Indeed, thanks to the Fisher-KPP hypothesis \eqref{eq:gKPP} and \eqref{eq:fKPP}, (super)solutions of \eqref{0eq:1.8} are supersolutions of the nonlinear problem \eqref{0eq:1.1}. The result is the following.

\begin{proposition}
	\label{0pr:1.2bis}
	Let $c^*_\infty$ be the quantity defined in \eqref{0eq:1.10}. Then, problem \eqref{0eq:1.8} admits supersolutions of the form \eqref{0eq:1.9} for every $c\geq c^*_\infty$. 
	\end{proposition}

\begin{proof}
	The functions of the form \eqref{0eq:1.9} automatically satisfy the third and the fourth equations in \eqref{0eq:1.8}. Thus, by plugging \eqref{0eq:1.9} into the system, we have that such a guess is a supersolution if and only if
	\begin{linenomath*}
	\begin{equation*}
		\left\{
		\begin{array}{l}
			D\a^2-c\a +g'(0)\leq 0 \\
			d\a^2-c\a+f'(0)\leq 0.
		\end{array}\right.
	\end{equation*}
\end{linenomath*}
This shows that we have to look, according to the different values of $D, d, g'(0), f'(0)$, for the values of $c>0$ for which the intervals
\begin{equation}
\label{0eq:1.11bis}
I_D(c):=\left\{r_D^-(c)\leq \a\leq r_D^+(c)\right\} \quad \text{ and } \quad I_d(c):=\left\{r_d^-(c)\leq \a\leq r_d^+(c)\right\},
\end{equation}
where
 \begin{equation}
 \label{0eq:1.12bis}
 r_D^{\pm}(c):=\frac{c\pm\sqrt{c^2-c_g^2}}{2D}, \qquad r_d^{\pm}(c):=\frac{c\pm\sqrt{c^2-c_f^2}}{2d},
 \end{equation}
 satisfy $I_D(c)\cap I_d(c)\neq\emptyset$ and take $\a$ in such intersection. We define such intervals as the empty set when the quantities in \eqref{0eq:1.12bis} are not real; thus, for the intersection to be non-empty, necessarily $c\geq\max\{c_f,c_g\}$. In addition, we have that the intervals increase (as sets) as $c$ increases, since $r_d^+(c)$ and $r_D^+(c)$ are increasing, while $r_d^-(c)$ and $r_D^-(c)$ are decreasing.
 
 After these preliminary observations, we shall consider two cases according to the relative position of $c_f$ and $c_g$; anyway, we detail only the case $c_f\leq c_g$, since the other one can be treated analogously and amounts to interchanging the roles of the parameters in the upper and the lower half-planes. This remark explains also the emergence of the condition in the first line of \eqref{0eq:1.10}, which is obtained from the condition in the second line by replacing $d$ with $D$, $f'(0)$ with $g'(0)$.
 
 Thus, consider $c_f\leq c_g$ and distinguish three sub-cases:
 \begin{enumerate}[(i)]
 	\item if the point $I_D(c_g)$ lies inside $I_d(c_g)$, i.e., $r_d^-(c_g)\leq r_D^{\pm}(c_g)=\frac{c_g}{2D}\leq r_d^+(c_g)$, which can be equivalently written as
 	\[c_g\left|\frac{1}{2D}-\frac{1}{2d}\right|\leq\frac{\sqrt{c^2_g-c_f^2}}{2d},\]
 	and it is easy to see that such a relation holds true if and only if the second condition in \eqref{0eq:1.10} is satisfied, then the monotonicity of the intervals \eqref{0eq:1.11bis} with respect to $c$ entails that they intersect for every $c\geq c_g$;
 	\item  if the point $I_D(c_g)$ lies at the right of $I_d(c_g)$, i.e., $r_D^{\pm}(c_g)> r_d^+(c_g)$, the first value of $c$ for which the intervals in \eqref{0eq:1.11bis} intersect is the one satisfying $r_D^{-}(c)= r_d^+(c)$, which equals the quantity $c_a$ defined in \eqref{0eq:1.11}, and the intersection is non-empty for all larger values of $c$;
 	\item similarly, if the point $I_D(c_g)$ lies at the left of $I_d(c_g)$, i.e., if $r_D^{\pm}(c_g)< r_d^+(c_g)$, we have intersections for every $c$ starting from the value satisfying $r_D^{+}(c)= r_d^-(c)$, which is again $c_a$.
 	\end{enumerate}
	\end{proof}

We now consider problem \eqref{0eq:1.1} in the frame moving with speed $c$ in the direction $x_1$, that is 
\begin{equation}
\label{0eq:1.13}
\left\{
\begin{array}{lll}
\p_t u-D\D u+ c\p_{x_1} u=g(u) & x\in\R^{N-1}\times\R^-, & t>0 \\
\p_t v-d\D v+ c\p_{x_1} v=f(v) & x\in\R^{N-1}\times\R^+, & t>0 \\
D\, \p_{x_N} u = \nu v-\mu u & x\in\R^{N-1}\times\{0\}, & t>0 \\
-d\, \p_{x_N} v = \mu u-\nu v & x\in\R^{N-1}\times\{0\}, & t>0, \\
\end{array}\right.
\end{equation}
and look for compactly supported, stationary, generalized subsolutions for such a system. The first result gives such subsolutions with support contained in the upper half-space.

\begin{proposition}
	\label{0pr:1.3}
	For every $0<c<c_f$, there exist $\e_0=\e_0(c)>0$ and $L=L(c)>0$ such that, for every $0<\e<\e_0$,  problem \eqref{0eq:1.13} admits a non-negative, stationary, generalized, subsolution of the form $\e(0, \un v)$, with $\un v$ bounded and compactly supported.
\end{proposition}
	
	\begin{proof}
		Fix $c\in(0,c_f)$, take $\t=\t(c)>0$ so that 
		\begin{equation}
			\label{0eq:1.14}
			f'(0)-\frac{c^2}{4d}-\t>0,
			\end{equation}
			and consider the following eigenvalue problem
			\begin{linenomath*}
			\begin{equation*}
			\left\{\begin{array}{ll}
			-d\D\phi=\l_1(L)\phi & x\in B^L:=B((0,\ldots,0,(L+1)),L)\subseteq\R^N \\
			\phi=0 & x\in\p B^L \\
			\max_{B^L}\phi=1. &
			\end{array}\right.
			\end{equation*}
		\end{linenomath*}
			Let $\phi_1$ denote its unique positive solution.	Since the principal eigenvalue $\l_1$ decreases in $L$ and converges to $+\infty$ as $L\to 0$ and to $0$ as $L\to+\infty$, there exists a unique value of $L$ for which $\l_1$ equals the left-hand side of \eqref{0eq:1.14}. For such an $L$,  the function $\un v:=e^{\frac{c}{2d}x_1}\phi_1(x)$ satisfies the linear equation $-d\D v+c\p_{x_1}v=(f'(0)-\t)v$ in $B^L$; thus, if we extend it to $0$ in the rest of $\R^{N-1}\times\R^+$, the regularity of $f$ ensures that, for $\e$ small enough, $\e\un v$ is a subsolution of the second equation in \eqref{0eq:1.13}.
		It is then immediate to see that the pair $\e(0,\un v)$ is a generalized subsolution of the whole system. We observe that both the $u$ and the $v$ component of such subsolution are identically equal to $0$ on $\left\{x_N=0\right\}$, thus the condition given in \cite[Proposition 3.3]{BRR1} - properly modified in the natural way to deal with system \eqref{0eq:1.1} - for the generalized comparison principle to hold true is  satisfied.
	\end{proof}

By reasoning analogously in the lower half-space, we immediately obtain the following result.

\begin{proposition}
	\label{0pr:1.4}
	For every $0<c<c_g$, there exist $\e_0=\e_0(c)>0$ and $L=L(c)>0$ such that, for every $0<\e<\e_0$,  problem \eqref{0eq:1.13} admits a non-negative, stationary, generalized subsolution of the form $\e(\un u, 0)$, with $\un u$ bounded and compactly supported.
\end{proposition}

The subsolutions provided by Propositions~\ref{0pr:1.3} and~\ref{0pr:1.4} are supported in the upper and lower half-space respectively. Thus, they are simply given by very classical subsolutions of the Fisher-KPP equation. Instead, in order to capture the anomalous speed $c_a$, we need to construct some subsolutions exploiting the coupling of the system. This is the result of the next proposition.

\begin{proposition}
	\label{0pr:1.5}
	Assume
	\begin{equation}
		\label{0eq:1.14bis}
		\frac{D}{d}> 2-\frac{g'(0)}{f'(0)} \qquad \text{and} \qquad \frac{d}{D}> 2-\frac{f'(0)}{g'(0)}.
		\end{equation}
	Then, for $c<c_a$, $c\sim c_a$, where $c_a$ is the one given in \eqref{0eq:1.11}, there exist arbitrarily small, non-negative, compactly supported, stationary, generalized subsolutions of
	\eqref{0eq:1.13}.
	
	\begin{proof}
		Since the proof is quite long and requires many constructions, we divide it into several steps. In Step 1 we consider the linearization of \eqref{0eq:1.13} around $(0,0)$ with a penalization, i.e., with $f'(0)$ and $g'(0)$ replaced by slightly smaller numbers, and a truncation at the hyperplanes $\{x_N=\pm L\}$, where we impose homogeneous Dirichlet boundary conditions.  In this setting, our goal will be the construction of compactly supported solutions with a specified form by adjusting some auxiliary parameters. Such an adjustment will lead us to the study of the relative position of certain planar curves. In Step 2 we define, for large enough $L$ and using the curves introduced in Step 1, an approximate speed of propagation $c^*_L$, and we construct arbitrarily small, compactly supported subsolutions of \eqref{0eq:1.13} for $c<c^*_L$, $c\sim c^*_L$. Finally, in Step 3, we show that, as $L\to+\infty$, $c^*_L$ converges to $c_a$. This, together with Step 2, will imply, by taking a sufficiently large $L$, the existence of the desired subsolutions for $c$ smaller than and arbitrarily close to $c_a$.
		
		\vspace{0.2cm}
		\emph{Step 1.}
		We first look for stationary subsolutions with compact support of the following system:
\begin{equation}
\label{0eq:1.linpenaltrunc}
\left\{
\begin{array}{lll}
\p_t u-D\D u+ c\p_{x_1}u=\left(g'(0)-\t\right)u & x\in\R^{N-1}\times(-L,0), & t>0 \\
\p_t v-d\D v+ c\p_{x_1}v=\left(f'(0)-\t\right)v & x\in\R^{N-1}\times(0,L), & t>0 \\
D\, \p_{x_N} u = \nu v-\mu u & x\in\R^{N-1}\times\{0\}, & t>0 \\
-d\, \p_{x_N} v = \mu u-\nu v & x\in\R^{N-1}\times\{0\}, & t>0 \\
u=0 & x\in\R^{N-1}\times\{-L\}, & t>0 \\
v=0 & x\in\R^{N-1}\times\{L\}, & t>0, 
\end{array}\right.
\end{equation}
		where $\t>0$, $\t\sim 0$.
		In analogy with the guess \eqref{0eq:1.9}, we seek solutions having an exponential profile in the $x_1$ variable, while the vanishing conditions lead us to consider three types of profile in the variable $x_N$:
\begin{equation}
\label{0eq:1.15}
\begin{aligned}
\un{u}_i(x_1,x_2,\ldots,x_{N-1},x_N)&= e^{\a x_1}\prod_{j=2}^{N-1}\cos\left(\eta\, x_j\right) \phi_{u,i}(x_N), \\
\un{v}_i(x_1,x_2,\ldots,x_{N-1},\tilde x_N)&=e^{\a x_1}\prod_{j=2}^{N-1}\cos\left(\eta\, x_j\right) \g\phi_{v,i}(\tilde x_N),
\end{aligned}
 \qquad i\in\{1,2,3\},
\end{equation}
		where $\a$ and $\g$ are positive constants,
		\begin{equation}
		\label{0eq:1.17}
				\begin{array}{rll}
		(\phi_{u,1},\phi_{v,1})\!\!\!\!&=(\sin(\b(x_N+L)),\sin(\d(L-\tilde x_N))), & \text{with } \b,\d\in\left(0,\frac{\pi}{L}\right), \\
		(\phi_{u,2},\phi_{v,2})\!\!\!\!&=(\sin(\b(x_N+L)),\sinh(\d(L-\tilde x_N))), & \text{with } \b\in\left(0,\frac{\pi}{L}\right), \quad \d>0, \\
		(\phi_{u,3},\phi_{v,3})\!\!\!\!&=(\sinh(\b(x_N+L)),\sin(\d(L-\tilde x_N))), & \text{with } \d\in\left(0,\frac{\pi}{L}\right), \quad \b>0,
		\end{array}
	\end{equation}
	and, if $N>2$, $\eta$ is a small constant satisfying, in particular,
	\begin{equation}
	\label{0eq:1.18}
	0<\eta<\min\left\{\sqrt{\frac{g'(0)-\theta}{D(N-2)}},\sqrt{\frac{f'(0)-\theta}{d(N-2)}}\right\}.
	\end{equation}
	With such choices the functions $\left(\un{u}_i(x_1,\ldots,x_{N-1},x_N),\un{v}_i(x_1,\ldots,x_{N-1},\tilde x_N)\right)$, $i\in\{1,2,3\}$, defined in \eqref{0eq:1.15}  are positive for $x_j\in\left(-\frac{\pi}{2\eta},\frac{\pi}{2\eta}\right)$, for all $j\in\{2,\ldots,N-1\}$, $x_N\in(-L,0)$, $\tilde x_N\in(0,L)$, while they vanish on the hyperplanes $x_j=\pm\frac{\pi}{2\eta}$, $j\in\{2,\ldots,N-1\}$, $x_N=-L$, $\tilde x_N=L$. Thus, if we extend such functions to $(0,0)$ in $\{x\in\R^N:x_j\notin\left(-\frac{\pi}{2\eta},\frac{\pi}{2\eta}\right) \text{ for some $j\in\{2,\ldots,N-1\}$}\}$, we obtain a compactly supported pair in all variables except for $x_1$. We will show in Step 2 how to solve this issue.

	By plugging \eqref{0eq:1.15} for $i=1$ into \eqref{0eq:1.linpenaltrunc}, we obtain the following system for the unknowns $\a,\b,\g,\d$:
		\begin{equation}
			\label{0eq:1.16zero}
			\left\{ \begin{array}{l}
				c\a-D\a^2+D(N-2)\eta^2+D\b^2=g'(0)-\t \\
				c\a-d\a^2+d(N-2)\eta^2+d\d^2=f'(0)-\t \\
				D\b\cos(\b L)=\nu\g\sin(\d L)-\mu\sin(\b L) \\
				d\g\d\cos(\d L)=\mu\sin(\b L)-\nu\g\sin(\d L).
				\end{array} \right.
			\end{equation}
		From the last equation of \eqref{0eq:1.16zero} we find
		\[
		\g=\frac{\mu\sin(\b L)}{\nu\sin(\d L)+d\d\cos(\d L)},
		\]
		and, for $\g$ to be positive, we take $0<\d<\ov \d$, where $\ov \d\in\left(\frac{\pi}{2L},\frac{\pi}{L}\right)$ is the first value of $\d$ for which the denominator vanishes. Then, we substitute such a value of $\g$ in the third equation of \eqref{0eq:1.16zero}, obtaining
		\begin{equation}
			\label{0eq:1.16.1}
			\chi_{u,1}(\b)=\chi_{v,1}(\d),
		\end{equation}
		where we have set
		\begin{equation}
			\label{0eq:1.16.2}
		\chi_{u,1}(\b):=D\b\cot(\b L), \qquad \chi_{v,1}(\d):=\frac{-\mu d\d\cos(\d L)}{\nu\sin(\d L)+d\d\cos(\d L)}.
		\end{equation}
	The function $\chi_{u,1}$ is analytic and decreasing in $\left(0,\frac{\pi}{L}\right)$, positive in $\left(0,\frac{\pi}{2L}\right)$, negative in $\left(\frac{\pi}{2L},\frac{\pi}{L}\right)$ and satisfies
	\begin{equation}
		\label{0eq:1.16bis}
	\lim_{\b\da 0}\chi_{u,1}(\b)=\frac{D}{L}, \qquad \lim_{\b\da 0}\chi_{u,1}'(\b)=0, \qquad \lim_{\b\ua\frac{\pi}{L}}\chi_{u,1}(\b)=-\infty.
	\end{equation}
	On the other hand, $\chi_{v,1}$ is analytic and increasing in $\left(0,\ov\d\right)$, negative in $\left(0,\frac{\pi}{2L}\right)$, positive in $\left(\frac{\pi}{2L},\ov \d\right)$ and satisfies
	\[
	\lim_{\d\da 0}\chi_{v,1}'(\d)=0, \qquad \lim_{\d\ua\ov\d}\chi_{v,1}(\d)=+\infty.
	\]
	After all the considerations, if we denote by $\un\b\in\left(\frac{\pi}{2L},\frac{\pi}{L}\right)$ and $\un\d\in\left(\frac{\pi}{2L},\ov\d\right)$ the unique values of $\b$ and $\d$ satisfying, respectively, $\chi_{u,1}(\un\b)=\lim_{\d\da 0}\chi_{v,1}(\d)$ and $\chi_{v,1}(\un{\d})=\lim_{\b\da 0}\chi_{u,1}(\b)$, the implicit function theorem gives the existence of a 1-1 function $\d_1(\b):(0,\un\b)\to(0,\un{\d})$ such that relation \eqref{0eq:1.16.1} holds true if and only if $\d=\d_1(\b)$. Such a function is analytic in $(0,\un \b)$, decreasing and satisfies $\lim_{\b\da 0}\d_1'(\b)=0$.
	
	In a similar way, if we pass to the case $i=2$, we obtain the system
	\begin{equation}
		\label{0eq:1.17zero}
		\left\{ \begin{array}{l}
			c\a-D\a^2+D(N-2)\eta^2+D\b^2=g'(0)-\t \\
			c\a-d\a^2+d(N-2)\eta^2-d\d^2=f'(0)-\t \\
			D\b\cos(\b L)=\nu\g\sinh(\d L)-\mu\sin(\b L) \\
			d\g\d\cosh(\d L)=\mu\sin(\b L)-\nu\g\sinh(\d L),
		\end{array} \right.
	\end{equation}
	whose last equation gives
	\[
	\g=\frac{\mu\sin(\b L)}{\nu\sinh(\d L)+d\d\cosh(\d L)},
	\]
	and no restriction on $\d$ is now necessary for $\g$ to be well-defined and positive. By substituting such an expression into the third equation, we obtain
	\begin{equation}
		\label{0eq:1.17.1}
		\chi_{u,2}(\b)=\chi_{v,2}(\d),
	\end{equation}
	where we have set
	\begin{equation}
		\label{0eq:1.17.2}
		\chi_{u,2}(\b):=D\b\cot(\b L), \qquad \chi_{v,2}(\d):=\frac{-\mu d\d\cosh(\d L)}{\nu\sinh(\d L)+d\d\cosh(\d L)}.
	\end{equation}
	The function $\chi_{u,2}$ coincides with $\chi_{u,1}$, while $\chi_{v,2}$ is analytic, negative, decreasing and satisfies
	\begin{linenomath*}
		\begin{equation*}
	\lim_{\d\da 0}\chi_{v,2}(\d)=\lim_{\d\da 0}\chi_{v,1}(\d)>\lim_{\d\to +\infty}\chi_{v,2}(\d)>-\infty, \qquad \lim_{\d\da 0}\chi_{v,2}'(\d)=0.
	\end{equation*}
\end{linenomath*}
	Thus, if $\un\b$ is as above, and we denote by $\ov\b\in\left(\un\b,\frac{\pi}{L}\right)$ the unique value of $\b$ such that $\chi_{u,2}(\b)=\lim_{\d\to+\infty}\chi_{v,2}(\d)$, there exists a 1-1 analytic function $\d_2(\b):(\un\b,\ov\b)\to(0,+\infty)$ such that relation \eqref{0eq:1.17.1} holds true if and only if $\d=\d_2(\b)$. Such a function is increasing and satisfies $\lim_{\b\ua\ov\b}\d_2(\b)=+\infty$.
		
		Finally, for $i=3$, the system obtained by plugging \eqref{0eq:1.15} into the linearization reads
		\begin{equation}
			\label{0eq:1.18zero}
			\left\{ \begin{array}{l}
				c\a-D\a^2+D(N-2)\eta^2-D\b^2=g'(0)-\t \\
				c\a-d\a^2+d(N-2)\eta^2+d\d^2=f'(0)-\t \\
				D\b\cosh(\b L)=\nu\g\sin(\d L)-\mu\sinh(\b L) \\
				d\g\d\cos(\d L)=\mu\sinh(\b L)-\nu\g\sin(\d L),
			\end{array} \right.
		\end{equation}
		whose last equation gives
		\[
		\g=\frac{\mu\sinh(\b L)}{\nu\sin(\d L)+d\d\cos(\d L)},
		\]
		which is positive for $0<\d<\ov\d$, and, once substituted in the the third equation, yields to
		\begin{equation}
			\label{0eq:1.18.1}
			\chi_{u,3}(\b)=\chi_{v,3}(\d),
		\end{equation}
		where we have set
		\begin{linenomath*}
		\begin{equation*}
			\chi_{u,3}(\b):=D\b\coth(\b L), \qquad \chi_{v,3}(\d):=\frac{-\mu d\d\cos(\d L)}{\nu\sin(\d L)+d\d\cos(\d L)}.
		\end{equation*}
	\end{linenomath*}
		The function $\chi_{v,3}$ coincides with $\chi_{v,1}$, while $\chi_{u,3}$ is analytic, positive and increasing in $\left(0,+\infty\right)$, and it satisfies
			\begin{equation}
			\label{0eq:1.18bis}
		\lim_{\b\da 0}\chi_{u,3}(\b)=\frac{D}{L}, \qquad \lim_{\b\da 0}\chi_{u,3}'(\b)=0, \qquad \lim_{\b\to+\infty}\chi_{u,3}(\b)=+\infty.
		\end{equation}
		Thus, recalling that we have set $\un\d\in\left(\frac{\pi}{2L},\ov\d\right)$ to be the value satisfying $\chi_{v,3}(\un\d)=\frac{D}{L}=\lim_{\b\da 0}\chi_{u,3}(\b)$, there exists a 1-1 increasing analytic function $\d_3(\b):(0,+\infty)\to(\un\d,\ov\d)$ such that relation \eqref{0eq:1.18.1} holds true if and only if $\d=\d_3(\b)$. In addition, $\d_3$ satisfies $\lim_{\b\da 0}\d_3'(0)=0$.
		
		We now  glue the functions defined above as follows:
		\begin{linenomath*}
		\begin{equation*}
		\mathfrak{d}(\b):=\begin{cases}
		\d_1^2(\b) & \text{for $\b\in(0,\un\b)$,} \\
		-\d_2^2(\b) & \text{for $\b\in(\un\b,\ov\b)$,} \\
		\d_3^2(-\b) & \text{for $\b\in(-\infty,0)$.}
		\end{cases}
		\end{equation*}
	\end{linenomath*}
		Observe that here we consider the even extension of $\delta_3^2(\beta)$ for negative $\b$'s, while in \eqref{0eq:1.17} we require $\b$ to be positive. This is only due to technical reasons, since it will simplify the study of the curves that we are going to introduce; however, when we will finally go back to the construction of $(\un{u}_i,\un{v}_i)$, we will take a positive $\b$, as we will detail in the sequel.
		
		We claim that $\mathfrak{d}\in\mc{C}^1(-\infty,\ov\b)$ and that it is analytic in $(0,\ov\b)$. For this latter point, since we already know the analyticity of $\d_1$ and $\d_2$ in their respective domains, it remains to check it at the gluing point. To this end, consider the extended function $\d_1:(0,\ov\b)\to\C$ obtained by applying the complex version of the implicit  function theorem to relation \eqref{0eq:1.16.1}. As $\cosh(i x)=\cos(x)$ and $\sinh(i x)=i\sin(x)$ for every $x\in\C$, by comparing \eqref{0eq:1.16.1}-\eqref{0eq:1.16.2} and \eqref{0eq:1.17.1}-\eqref{0eq:1.17.2}, the uniqueness given by the implicit function theorem entails that, for $\b\in(\un\b,\ov\b)$, $\d_2(\b)=i\d_1(\b)$. Thus $\d_2^2(\b)=-\d_1^2(\b)$, and we obtain that, by construction, $\mathfrak{d}(\b)$ is analytic at $\b=\un\b$. Passing to the continuity and differentiability in $0$, such properties follow from the definition of $\d_1$ and $\d_3$, together with the first two relations in \eqref{0eq:1.16bis} and \eqref{0eq:1.18bis}.

		With this definition of $\mathfrak{d}(\b)$ and recalling the monotonicities of the functions $\d_i$ discussed above, we have that $\mathfrak{d}(\b)$ is decreasing and bounded from above. Thus, if we set 
		\begin{equation}
\label{0eq:1.19bis}
\un c:=\inf\{c>0: \exists \, \b\in(-\infty,\ov\b) \text{ s.t. } c^2=4d\left(f'(0)-\t-d(N-2)\eta^2-d\mathfrak{d}(\b)\right)\},
\end{equation}
		for every $c>\un c$ we have that the second equation of each of the systems \eqref{0eq:1.16zero}, \eqref{0eq:1.17zero} and \eqref{0eq:1.18zero} can be solved for $\a$ as a function of $\b$ as follows
		\begin{equation}
		\label{0eq:1.19ter}
		\a_d^{\pm}(c,\b,L):=\frac{c\pm\sqrt{c^2-4d\left[f'(0)-\t-d(N-2)\eta^2-d\mathfrak{d}(\b)\right]}}{2d},
		\end{equation}
		with $\b\in(-\infty,\breve{\b}(c))$, where $\breve{\b}(c)\in(-\infty,\ov\b)$ is the unique value of $\b$ for which the argument of the square root in \eqref{0eq:1.19ter} vanishes. Moreover,  $\b\mapsto\a_d^{-}(c,\b,L)$ is increasing in its domain, while $\b\mapsto\a_d^{+}(c,\b,L)$ is decreasing. By gluing together these curves, we obtain a differentiable curve in the upper half-plane $(\beta,\alpha)$, namely
		\[
		\Si_d(c,L):=\left\{\left(\b,\a_d^{\pm}(c,\b,L)\right):\b\leq\breve{\b}(c)\right\}.
		\]
		
		Turning now to the first equation of systems \eqref{0eq:1.16zero}, \eqref{0eq:1.17zero} and \eqref{0eq:1.18zero}, we can handle all of them together by setting, when the argument of the square roots are non-negative,
		\begin{equation}
			\label{0eq:1.20}
		\a_D^{\pm}(c,\b):=\begin{cases}
		\frac{c\pm\sqrt{c^2-4D\left[g'(0)-\t-D(N-2)\eta^2+D\b^2\right]}}{2D} & \text{for $\b<0$} \\
		\frac{c\pm\sqrt{c^2-4D\left[g'(0)-\t-D(N-2)\eta^2-D\b^2\right]}}{2D} & \text{for $\b\geq 0$.}
		\end{cases}
		\end{equation}
		In order to study in detail the domain of definition of these functions we denote,  if $0<c<2\sqrt{D\left[g'(0)-\t-D(N-2)\eta^2\right]}$, by $\hat{\b}(c)$ the unique positive value of $\b$ such that $4D^2\b^2=4D(g'(0)-\t-D(N-2)\eta^2)-c^2$, while, if $c\geq 2\sqrt{D\left[g'(0)-\t-D(N-2)\eta^2\right]}$, we define $\hat{\b}(c)$ as the unique non-positive value of $\b$ such that $4D^2\b^2=c^2-4D\left[g'(0)-\t-D(N-2)\eta^2\right]$. With this notation, $\a_D^{\pm}(c,\b)$ is defined for $\b\geq\hat{\b}(c)$, and we consider  in the upper half-plane $(\b,\a)$ the curve
		\[
		\Si_D(c):=\left\{(\b,\a_D^{\pm}(c,\b)):\b\geq\hat{\b}(c)\right\}.
		\]
		The right part of $\Si_D$, corresponding to the case $\b\geq 0$ in \eqref{0eq:1.20}, consists of the branches of an hyperbola, while the left one, corresponding to the case $\b<0$, consists, when it is defined, of a half-circle with center $\left(0,\frac{c}{2D}\right)$ and radius $\frac{\sqrt{c^2-4D\left[g'(0)-\t-D(N-2)\eta^2\right]}}{2D}$. In addition, it is easy to see that, even when $\Si_D$ is defined both for $\b<0$ and $\b\geq 0$, such a curve is differentiable.
		
		\vspace{0.2cm}
		\emph{Step 2.}
		We now analyze the behavior with respect to $c$ of the curves that we have introduced in the previous step. First of all, thanks to \eqref{0eq:1.18} and since $\d_3(\b)\in(\un\d,\ov\d)\subset\left(\frac{\pi}{2L},\frac{\pi}{L}\right)$, which implies that $\mathfrak{d}(\b)$ converges to $0$ uniformly for $\b\in(-\infty,0)$ as $L\to+\infty$, we can take a sufficiently large $L$  so that the quantity $\un c$ introduced in \eqref{0eq:1.19bis} is strictly positive. In such a case, as $c\da\un c$,  $\breve{\b}(c)\to-\infty$, while $\hat{\b}(c)\to\hat{\b}(\un c)$, entailing that $\Si_d(c,L)$ and $\Si_D(c)$ do not intersect for $c>\un c$, $c\sim \un c$.
		
		In addition, we observe that $c\mapsto\a_d^{+}(c,\b,L)$ is increasing, while $c\mapsto\a_d^{-}(c,\b,L)$ is decreasing whenever $\a_d^{-}(c,\b,L)>0$, which is the case that we consider, since we look for positive $\a$'s. The same monotonicities with respect to $c$ hold true for $\a_D^\pm$.
		
		Finally, as $c\to+\infty$, $\hat{\b}(c)\to-\infty$, $\breve{\b}(c)\to\ov\b$, and the curves $\a_d^{-}$ and $\a_D^{-}$ converge to $0$, while $\a_d^{+}$ and $\a_D^{+}$ to $+\infty$, always in their domain of definition.
		
		As a consequence, all these considerations and the differentiability of the curves $\Si_d$ and $\Si_D$ imply that there is a smallest $c$ (greater than $\un c$)  for which such curves intersect, being tangent. We denote such a value by $c^*_L$.
		
		By construction, the curves $\Si_d(c,L)$ and $\Si_D(c)$  have no real intersections for $c<c^*_L$, $c\sim c^*_L$. Nonetheless, we will now show that they have complex intersections. 
		
		Indeed, consider the case in which the tangency point for $c=c_L^*$, which will be denoted by $(\b^*_L,\a^*_L)$, satisfies $\b^*_L\neq 0$, i.e., it lies in the region where these curves are analytic, the tangent is not vertical in the $(\b,\a)$ plane, and the branches that are tangent are $\a_D^+$ and $\a_d^-$. In such a case, we can apply Rouch\'e's theorem as in \cite[Theorem A.1]{RTV} to find, for $c<c^*_L$, $c\sim c^*_L$, complex zeros of the function $\a_D^+(c,\cdot)-\a_d^-(c,\cdot,L)$ which are non-real and have non-zero imaginary part. Such zeros are solutions of the $\b$ variable in \eqref{0eq:1.16zero}, if $\b^*_L\in(0,\un\b)$, or \eqref{0eq:1.17zero}, if $\b^*_L\in(\un\b,\ov\b)$, or \eqref{0eq:1.18zero}, if $\b^*_L<0$. By looking at the first two equations of these systems, it is apparent that the corresponding solutions for $\a$ are also non-real and have non-zero imaginary part. Such solutions give rise to a complex solution $(\un u,\un v)$ of \eqref{0eq:1.linpenaltrunc}, whose real part, which is still a solution of \eqref{0eq:1.linpenaltrunc} and which we still denote as $(\un u,\un v)$, oscillates in $x_1$. 
		
		We now want to modify $(\un u,\un v)$ in order to get a non-negative, compactly supported pair that satisfies the condition on the supports required in \cite[Proposition 3.3]{BRR1} for the generalized comparison principle to hold true. Observe that the frequency of oscillation in the $x_1$ variable of~$\un u$ and $\un v$ is the same; thus, if we fix one positive bump $E_x$ of $\un u$,  there will be at most one connected component  where $\un v$ is positive and whose closure (in $\R^{N-1}\times[0,L]$) intersects $\ov{E_x}\cap \left\{x_N=0\right\}$ in a set with open interior. If such a component exists, we denote it as $F_x$, otherwise - i.e., if the oscillations of $\un u$ and $\un v$ are shifted exactly of half a period in the variable $x_1$ - we set $F_x$ to be any finite combination of positive bumps of $\un v$. With this choice, if we multiply $(\un u,\un v)$ by a small positive constant and extend it to $(0,0)$ outside $E_x\times F_x$, we have that \cite[Proposition 3.3]{BRR1} can be applied with $E:=E_x\times(0,+\infty)$ and $F:=F_x\times(0,+\infty)$, providing us with the desired generalized compactly supported solution of \eqref{0eq:1.13}.

		If the branches that are tangent are $\a_D^-$ and $\a_d^+$, we can reason similarly, starting from the function $\a_d^+(c,\cdot,L)-\a_D^-(c,\cdot)$.
		
		If $\b^*_L=0$, instead, we can consider the curves only for $\b\geq 0$, extend them analytically for $\b<0$, $\b\sim 0$ and now apply Rouch\'e's theorem to such analytic continuation to get compactly supported subsolutions of \eqref{0eq:1.1} as before (see~\cite{T16} for the details).
		
		Finally, if the tangent at $(\b^*_L,\a^*_L)$ is vertical, we can use the first two equations of the system under consideration, i.e., \eqref{0eq:1.16zero} or \eqref{0eq:1.17zero} or \eqref{0eq:1.18zero}, according to the value of $\b^*_L$, to express $\b$ in terms of $\a$ and repeat the above argument to such functions.
		
		To conclude this step, we emphasize that, when $\b^*_L<0$, one shall consider the even reflection of the curves for $\b>0$ for the conditions in \eqref{0eq:1.17} to be satisfied at the tangency point and apply Rouch\'e's theorem starting from such reflections. On this point, a last exception to be treated is when $\b^*_L=\un{\b}$, because in such a case the corresponding $\d$ is equal to $0$. If this happens, one shall repeat the whole construction for $i=1,2$, with the difference that $\b$ has to be expressed as a function of $\d$ in \eqref{0eq:1.16.1} or \eqref{0eq:1.17.1}, consider the curves obtained in the $(\delta,\alpha)$ plane and reason as in the case $\b^*_L=0$ described above in order to construct complex solutions.
		
		\vspace{0.2cm}
		\emph{Step 3.} In this last step we show that $c^*_L$ converges to $c_a$ as $L\to+\infty$. For convenience, we prove it for $\t=\eta=0$ and, since all the constructions performed above depend in a differentiable way on $\t$ and $\eta$, the same will hold true for $\t\sim 0$ and $\eta\sim 0$.
		
		To achieve our purpose, we need to study the behavior of $\Si_d(c,L)$ as $L\to+\infty$. As already shown, when $L\to+\infty$, $\mathfrak{d}(\b)\to 0$ uniformly for $\{\b<0\}$. Recall that $\b\to\mathfrak{d}(\b)$ is decreasing; thus, if $c<c_f$ and $L$ is large enough, we have $c^2-4d(f'(0)-d\mathfrak{d}(\b))<0$ for all $\b<\ov\b$, entailing that $\Si_d(c,L)$ is not defined. Thus, we consider $c\geq c_f$. In such a case, since $c^2-4d(f'(0)-d\mathfrak{d}(\b))$ is positive for $\b=0$, we have $\breve{\b}(c)>0$. Moreover, $\breve{\b}(c)<\ov\b<\frac{\pi}{L}$ thus $\breve{\b}(c)\to 0$ as $L\to+\infty$. This, together with the monotonicity of $\b\to\mathfrak{d}(\b)$, shows that the curve $\Si_d(c,L)$ converges in the Hausdorff distance, for all $c\geq c_f$, to the following set
		\begin{equation}
			\label{0eq:1.30}
		\Si_d(c,+\infty):=\left\{\left(\b,r_d^\pm(c)\right):\b<0\right\}\bigcup\left\{(0,\a):r_d^-(c)\leq\a\leq r_d^+(c)\right\},
		\end{equation}
		where $r_d^\pm(c)$ are the quantities defined in \eqref{0eq:1.12bis}.

		We now prove that $c^*_L<c_a$ for sufficiently large $L$. Indeed, recall from Proposition~\ref{0pr:1.2bis} that $c_a>\max\{c_f,c_g\}$, and, as shown above, this implies $\breve{\b}(c_a)>0$ and $\hat\b(c_a)<0$. 
		Moreover, from the definition of $c_a$ (see again Proposition~\ref{0pr:1.2bis}), we have $r_d^+(c_a)=r_D^-(c_a)$ or $r_d^-(c_a)=r_D^+(c_a)$. Let us treat the first case, the other one being analogous. The expression of the curves gives
		\[
		\a_d^+(c_a,0,L)>r_d^+(c_a)=r_D^-(c_a)=\a_D^-(c_a,0),
		\]
		and, since the left quantity converges to the right ones for $L\to+\infty$, we have, again for large $L$, that $\a_d^+(c_a,\b,L)$ lies inside the circle forming the left part of $\Si_D(c_a)$ for $\b<0$, $\b\sim 0$, while it lies outside such a circle for very negative $\b$, entailing that the curves intersect and, by definition, $c^*_L<c_a$.
		
		In order to conclude and prove the desired limit, we shall distinguish three cases, as we did in the proof of Proposition~\ref{0pr:1.2bis}, according to the relative position of $c_f$ and $c_g$. If $c_f<c_g$, then, for $c=c_g$, $\Si_D(c)$ is the degenerate hyperbola $\{(\b,\frac{c_g}{2D}\pm\b),\b\geq 0\}$, which, thanks to the assumption \eqref{0eq:1.14bis}, does not intersect the limiting curve \eqref{0eq:1.30}. By continuity, the curve  $\Si_d(c_g,L)$ does not intersect $\Si_D(c_g)$ for large $L$, implying that $\liminf_{L\to+\infty}c^*_L>c_g$. If, by contradiction, $\liminf_{L\to+\infty}c^*_L<c_a$, a similar reasoning shows that the curves $\Si_D(c)$ and $\Si_d(c,L)$ do not intersect for large $L$, against the definition of $c^*_L$. Since we have proved that $c^*_L<c_a$, we obtain that $\lim_{L\to+\infty}c^*_L=c_a$, as desired.
		
		The other cases can be handled similarly, the key point being that the region delimited by $\Si_D(c)$ on the positive $\a$ half-axis is the interval $I_D(c)$, while the region delimited by $\Si_d(c,L)$, again on the positive $\a$ half-axis, converges, as $L\to+\infty$, to the interval $I_d(c)$ (recall \eqref{0eq:1.11bis}-\eqref{0eq:1.12bis} for the definition of such intervals).
	\end{proof}
\end{proposition}

Now that we have in hand, for the problem in the moving frame with speed $c$ in the direction $x_1$, supersolutions decaying to $0$ as $\left|x_1\right|\to+\infty$ for any $c>c^*_{\infty}$, where the quantity $c^*_\infty$ is the one introduced in \eqref{0eq:1.10}, as well as compactly supported subsolutions provided $c<c^*_{\infty}$, the result on the asymptotic speed of propagation readily follows.

\begin{theorem}
	\label{0th:1.4} Assume \eqref{eq:gKPP} and \eqref{eq:fKPP} and recall the quantity $c^*_\infty$ introduced in \eqref{0eq:1.10}. Then, for any $i\in\{1,\ldots,N-1\}$ and every compactly supported initial datum $(u_0,v_0)\not\equiv(0,0)$, the solution $(u,v)$ of \eqref{0eq:1.1} satisfies:
	\begin{enumerate}[(i)]
		\item \label{0th:1.4.i} for all $c>c^*_\infty$, 
				\begin{linenomath*}
			\begin{equation*}
		\lim_{t\to\infty}\sup_{\substack{|x_i|\geq ct \\ x_N\leq 0, \; \tilde x_N\geq 0}}\left(u(x_1,\dots,x_N,t),v(x_1,\dots,\tilde x_N,t)\right)=(0,0),
			\end{equation*}
	\end{linenomath*}
		uniformly in $x_j$, for all $j\in\{1,\ldots,N-1\}\setminus\{i\}$,
		\item \label{0th:1.4.ii} for all $0<c<c^*_\infty$ and $a>0$, 
		\begin{linenomath*}
			\begin{equation*}
		\lim_{t\to\infty}\sup_{\substack{|x_i|\leq ct \\ -a\leq x_N\leq 0, \; 0\leq \tilde x_N\leq a}}\left(|u(x_1,\dots,x_N,t)-U(x_N)|+|v(x_1,\dots,\tilde x_N,t)-V(\tilde x_N)|\right)=0,
	\end{equation*}
\end{linenomath*}
		locally uniformly in $x_j$, for all $j\in\{1,\ldots,N-1\}\setminus\{i\}$, where $(U,V)$ is the unique positive, bounded steady state of problem \eqref{0eq:1.1}.
	\end{enumerate}

	\begin{proof}
		As already remarked before, using the symmetry of the problem, we can take without loss of generality $i=1$. For part \eqref{0th:1.4.i}, we consider the supersolution  $(\ov u,\ov v)$ of the form \eqref{0eq:1.9} with $c=c^*_\infty$, and denote by $\alpha^*$ the associated value of $\alpha$ found in Proposition \ref{0pr:1.2bis}. If we multiply such a supersolution by a positive constant $k$, we still have a supersolution of the linearized problem \eqref{0eq:1.10}, thus a supersolution of \eqref{0eq:1.1} thanks to the second relation of \eqref{eq:gKPP} and \eqref{eq:fKPP}. As a consequence, since the initial datum is bounded,  $k(\ov u,\ov v)$ lies, for $t=0$ and large $k$, above $(u_0,v_0)$. Then, by the comparison principle, for any $c>c^*_\infty$ and $x_1\geq ct$, there holds
\begin{equation}
\label{0eq:1.31}
\begin{aligned}
u(x_1,\dots,x_N,t)&\leq &k e^{-\a^*(x_1-c^*_\infty t)}&\leq & k e^{-\a^*(c-c^*_\infty) t}&\to 0 \\
v(x_1,\dots,\tilde x_N,t)&\leq& \frac{\mu}{\nu}k e^{-\a^*(x_1-c^*_\infty t)}&\leq&  \frac{\mu}{\nu}k e^{-\a^*(c-c^*_\infty) t}&\to 0
\end{aligned}
 \qquad \text{as $t\to+\infty$.}
\end{equation}
		The same limits holds for $x_1\leq -ct$ owing to the symmetry of the problem with respect to the change of variable $x_1\mapsto -x_1$. Finally, observe that \eqref{0eq:1.31} holds uniformly with respect to $x_j\in\R$, $j\in\{2,\ldots,N-1\}$, $x_N\leq 0$, $\tilde x_N\geq 0$. Statement \eqref{0th:1.4.i} is thereby proved.
		
For part \eqref{0th:1.4.ii}, take $c<c^*_\infty$, $c\sim c^*_\infty$ and consider 
\[\tilde u(x_1,\ldots,x_N,t):=u(x_1-ct,x_2,\ldots,x_N,t), \quad \tilde v(x_1,\ldots,\tilde x_N,t):=v(x_1-ct,x_2,\ldots,\tilde x_N,t).
\]		
		We claim that $(\tilde u,\tilde v)$ converges to $(U,V)$ as $t\to+\infty$  locally uniformly in space, in the closure of the respective domains. Then, by symmetry, the same conclusion holds for the translation $x_1+ct$. Once we have this, we can recover the whole interval $|x_1|\leq ct$ by using 
		\cite[Lemma 4.1]{BRR2}, which, loosely speaking, asserts that the convergence to $(U,V)$ occurs on a convex set (see also \cite[Lemma 4.4]{RTV} for a proof of such a result). Let us prove this claim.
		
		The pair $(\tilde u, \tilde v)$ solves \eqref{0eq:1.13} and has $(u_0,v_0)$ as an initial datum. To get an upper bound, we reason as in the proof of Proposition~\ref{0pr:1.1}: we take a sufficiently large constant as an initial datum for \eqref{0eq:1.13}; then, the solution, which lies above $(\tilde u, \tilde v)$ thanks to the comparison principle, does not depend on $x_1,\dots,x_{N-1}$, due to the invariance of the problem. Thus, it is a solution of \eqref{0eq:1.1} and converges to $(U,V)$ thanks to Theorem~\ref{0th:1.3}.
		
		Passing to the lower bound, the strong comparison principle ensures that $(\tilde u,\tilde v)$ is strictly positive at any positive time. As a consequence, at time $t=1$, we can place below such a solution the arbitrarily small, stationary subsolutions of \eqref{0eq:1.13} constructed in Proposition~\ref{0pr:1.3} (respectively, in Proposition~\ref{0pr:1.4}, or in Proposition~\ref{0pr:1.5}), when the value of $c^*_{\infty}$ is given by the first case in \eqref{0eq:1.10} (respectively, the second or the third case). By reasoning as in  Proposition~\ref{0pr:1.1}, one obtains that the solution starting from the subsolution converges to a steady state which depends only on $x_N$, thus a steady state of the problem without the drift term, which necessarily is $(U,V)$ thanks to Proposition~\ref{0pr:1.2}. Once again, the comparison principle gives the lower bound for $(\tilde u,\tilde v)$.
	\end{proof}
\end{theorem}

We conclude this section with a diagram (see Figure~\ref{0fig:1.1}) that summarizes the possible values of $c^*_\infty$ according to the different ranges of the diffusion and reaction parameters ($D,d,g'(0),f'(0)$) appearing in the system, and we prove the following qualitative properties of $c^*_\infty$.

\begin{figure}[h]
	\begin{center}
		\begin{overpic}[width=0.45\textwidth]{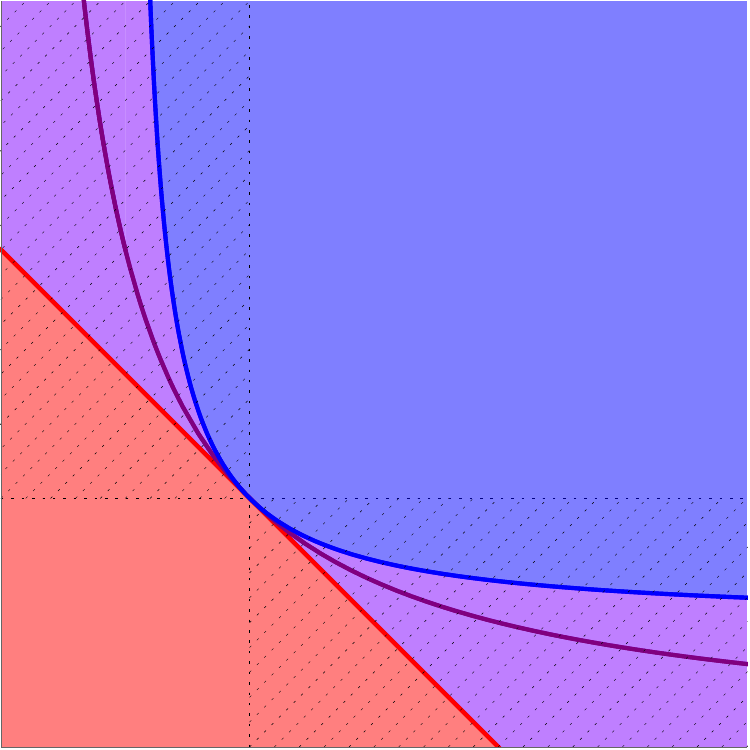}
			\put (101.5,-1) {$\displaystyle{\frac{D}{d}}$}
			\put (-15,98) {$\displaystyle{\frac{g'(0)}{f'(0)}}$}
			\put (-4,-5.2) {$0$}
			\put (31.5,-5.5) {$1$}
			\put (-4,31.5) {$1$}
		\end{overpic}
		\vspace{0.2cm}
		\caption{The value of the asymptotic speed of propagation $c^*_\infty$, defined in \eqref{0eq:1.10}, for problem \eqref{0eq:1.1} in any direction satisfying $x_N=0$: in red, below the line $\frac{D}{d}= 2-\frac{g'(0)}{f'(0)}$, the region of the parameters for which $c^*_\infty=c_f$ and in blue, above the upper hyperbola $\frac{d}{D}= 2-\frac{f'(0)}{g'(0)}$, which is obtained from the previous line by interchanging the diffusion parameters and the derivatives at $0$ of the reaction terms, the region where  $c^*_\infty=c_g$. In between, the region (containing the purple hyperbola $\frac{D}{d}=\frac{f'(0)}{g'(0)}$, which corresponds to $c_f=c_g$) where $c^*_\infty$ equals the anomalous speed $c_a$ introduced in \eqref{0eq:1.11}.  The dotted regions correspond to the cases in which the last inequality in \eqref{1eq:1.12} is strict.} \label{0fig:1.1}
	\end{center} 
\end{figure}

\begin{proposition}
	\label{0pr:1.6}
	The asymptotic speed of propagation $c^*_\infty$ of problem \eqref{0eq:1.1} satisfies the following properties:
	\begin{enumerate}[(i)]
		\Item  \label{0pr:1.6.0} 
		\begin{equation}
			\label{1eq:1.12}
			\max\{c_f,c_g\}\leq c^*_\infty\leq 2\sqrt{\max\{D,d\}\max\{g'(0),f'(0)\}},
		\end{equation} 
		and the last inequality is strict if and only if either $D>d$ and $f'(0)>g'(0)$, or $D<d$ and $f'(0)<g'(0)$;
		\item \label{0pr:1.6.i}  the function $c^*_\infty(d,D,f'(0),g'(0))$ is continuous;
		\item \label{0pr:1.6.ii} the map $D\mapsto c^*_\infty(D)$, with the other parameters fixed, is non-decreasing;
		\item \label{0pr:1.6.iii} $\lim_{D\to+\infty}\frac{c^*_\infty(D)}{\sqrt{D}}\in(0,+\infty)$.
	\end{enumerate}
\end{proposition}

\begin{proof}
	(i) The first inequality immediately follows from the construction of $c^*_\infty$ performed in Proposition~\ref{0pr:1.2bis}.
	
	Regarding the second one, up to interchanging the role of the two half-spaces, it will suffice to treat the case $D>d$. The inequality is trivial if $c^*_\infty=c_f$, since in this case, as it can be seen from Figure~\ref{0fig:1.1}, necessarily $f'(0)>g'(0)$, thus
	\[
	c^*_\infty=c_f<2\sqrt{Df'(0)}=2\sqrt{\max\{D,d\}\max\{g'(0),f'(0)\}}.
	\] 
	If $c^*_\infty=c_g$ and $g'(0)\geq f'(0)$, then
	\[
	c^*_\infty=c_g=2\sqrt{Dg'(0)}=2\sqrt{\max\{D,d\}\max\{g'(0),f'(0)\}},
	\]
	thus equality holds true.	Moreover, if $c^*_\infty=c_g$ and $g'(0)<f'(0)$, then
	\[c^*_\infty=c_g=2\sqrt{Dg'(0)}<2\sqrt{Df'(0)}=2\sqrt{\max\{D,d\}\max\{g'(0),f'(0)\}},\]
	as desired.	 Finally, when $c^*_\infty=c_a$, recalling the expression of $c_a$ (see \eqref{0eq:1.11}), our goal is to show that
	\[
	\frac{|Df'(0)-dg'(0)|}{\sqrt{(D-d)(f'(0)-g'(0))}}<2\sqrt{Df'(0)},
	\] 
	which is equivalent to
	\[
	y^2-y(6x-4x^2)+4x-3x^2<0, \qquad \text{ where we have set } x:=\frac{D}{d}, \quad y:=\frac{g'(0)}{f'(0)}.
	\]
	Direct computations show that, if $x>1$ and $0<y<1$, the region in the diagram satisfying such a relation contains the one that we are considering, which is $2-x<y<\frac{x}{2x-1}$.
	
	The above analysis also shows when the inequality is strict or large, thus the proof is concluded.
	
	(ii) It is obvious that $c^*_\infty$ is continuous in the interior of the regions where it coincides with $c_f$, $c_g$ or $c_a$, thus we only have to prove the continuity when we pass from one region to the other. Direct computations show that relation $c_a>c_g$ can be equivalently rewritten as
	\[
	(y(2x-1)-x)^2>0, \qquad \text{ where } x:=\frac{D}{d}, \quad y:=\frac{g'(0)}{f'(0)},
	\]
	thus we see that, as $y$ converges to the hyperbola $\frac{x}{2x-1}$, which is the one that separates the two regimes, the value of $c_a$ converges to $c_g$. (Incidentally, this also shows that indeed $c_a>c_g$ in the region where $c^*_\infty=c_a$). Similarly, $c_a>c_f$ is equivalent to $(y-(2-x))^2>0$, thus $c_a$ converges to $c_f$ as $y$ approaches the line $2-x$.
	
	(iii) Once that we know the continuity from the previous point, in order to obtain the monotonicity of $c^*_\infty(D)$ is suffices to prove the monotonicity in each region separately. Once again, the result is direct when $c^*_\infty=c_f$ or $c^*_\infty=c_g$. Let us focus on the case when $c^*_\infty=c_a$ for $D>d$ and $f'(0)<g'(0)$ - the other case can be treated analogously: we have $c^{*}_\infty=\frac{Df'(0)-dg'(0)}{\sqrt{(D-d)(f-g)}}$ and, by differentiating with respect to $D$,
	\[
	\frac{\p}{\p D}\left(\frac{Df'(0)-dg'(0)}{\sqrt{(D-d)(f-g)}}\right)=\frac{Df'(0)-2df'(0)+dg'(0)}{2(D-d)^{3/2}(f'(0)-g'(0))^{1/2}},
	\]
	and the numerator in the last expression is positive, since we are in the region where $\frac{D}{d}>2-\frac{g'(0)}{f'(0)}$.
	 
	(iv) This property immediately follows by analyzing the expressions of $c_g$ and $c_a$.
	\end{proof}

\begin{remark}
	It is not possible to expect, in general, higher regularity than the one stated in Proposition~\ref{0pr:1.6}\eqref{0pr:1.6.i}, since, if we consider $g'(0)=f'(0)$, we observe that $c^*_\infty(D)$ is not differentiable at $D=d$, as it passes from being constant to take the value $2\sqrt{Dg'(0)}$.
	\end{remark}

\setcounter{equation}{0}
\section{Cylinder and its complement}
\label{section3}
We now pass  to problem \eqref{1eq:1.1}, i.e., we consider a cylinder $\O=\R\times B_R\subset\R^N$ with Fisher-KPP nonlinearities both in the interior and its complement. As already said in the introduction, for this geometry it will be convenient to denote a generic point of $\R^N$ by $(x,y)$ with $x\in\R$ and $y\in\R^{N-1}$.

As in the previous section, we will first determine the long-time behavior of the solutions of~\eqref{1eq:1.1}, while, in the second part, we characterize their asymptotic speed of propagation in the $x$ direction, as well as some qualitative properties of such a speed.

\subsection{Long-time behavior for system \eqref{1eq:1.1}}
\label{1section2}
The first result that we obtain is the counterpart of Proposition~\ref{0pr:1.1} for this new geometry.

\begin{proposition}
	\label{1pr:2.1}
	Assume \eqref{eq:gKPP} and \eqref{eq:fKPP}. Then, for every $(u_0,v_0)\not\equiv(0,0)$, there exist two positive, bounded stationary solutions $(U_i,V_i)$ of \eqref{1eq:1.1}, $i\in\{1,2\}$, which are $x$-independent, invariant by rotations around the axis $\{y=0\}$, and such that the solution of \eqref{1eq:1.1} starting from $(u_0,v_0)$ satisfies
	\begin{linenomath*}
	\begin{equation*}
	(U_1,V_1)\leq\liminf_{t\to+\infty}(u,v)\leq\limsup_{t\to+\infty}(u,v)\leq(U_2,V_2)
	\end{equation*}
\end{linenomath*}
	locally uniformly in space, in the closure of the corresponding domains.
	\begin{proof}
		
		The same constant supersolution defined in \eqref{0eq:1.2bis} can be used here to show the existence of $(U_2,V_2)$ as in the proof of Proposition~\ref{0pr:1.1}. In addition, the $x$-invariance and the $y$-radial symmetry of the supersolution are inherited by the steady state, thanks to the invariance of \eqref{1eq:1.1}.
		
		As for the stationary subsolution that allows us to construct $(U_1,V_1)$, we have to reason a little differently, since now $\O$ has a bounded section. In this case, we consider the eigenvalue problem
		\begin{equation}
		\label{1eq:2.2}
		\left\{
		\begin{array}{ll}
		-\D_y \phi=\l \phi & y\in A_L:=A(R+1,L) \\
		\phi=0 & y\in\p A_L,
		\end{array}\right.
		\end{equation}
		where $\D_y$ denotes the Laplacian with respect to the $y$ variables and $A(R+1,L)$ the annulus of points in $\R^{N-1}$ such that $R+1<|y|<L$. Let $\l_1=\l_1(L)$ be the principal eigenvalue of \eqref{1eq:2.2} and $\phi_1$ the associated positive eigenfunction, normalized by $\max_{A_L}\phi_1=1$. We fix $\a:=\sqrt{\frac{f'(0)}{2d}}$ and, since $\l_1(L)\to 0$ as $L\to+\infty$ (this can be seen, e.g., by bounding $\l_1$ from above with the principal eigenvalue of a ball contained in $A_L$ and whose radius goes to $+\infty$ as $L\to+\infty$), we can take a sufficiently large $L$  so that $d\a^2+d\l_1(L)<f'(0)$. Then, there exists $\e_0>0$ such that, setting
		\begin{linenomath*}
		\begin{equation*}
		\un v:=\left\{
		\begin{array}{ll}
		\cos(\a x)\phi_1(y) & x\in\left(\frac{-\pi}{2\a},\frac{\pi}{2\a}\right), \, y\in A_L \\
		0 & \text{otherwise in $\R^N\setminus\O$,}
		\end{array}\right.
		\end{equation*}
	\end{linenomath*}
		the pair $\e(0,\un v)$ is a strict non-negative, stationary subsolution of \eqref{1eq:1.1} for all $\e\in(0,\e_0)$ which is, in addition, $y$-radial. We then conclude the proof by reasoning as in Proposition~\ref{0pr:1.1}.
	\end{proof}
\end{proposition}

The next question that we address is the uniqueness of the symmetric steady state. Numerical observations that we have performed show that, without adding additional hypothesis on the nonlinearities, uniqueness does not hold in general. For this reason we assume the strong KPP conditions \eqref{eq:gstrongKPP} and \eqref{eq:fstrongKPP}. Moreover, we provide a different proof than the one for the analogous result in Section~\ref{0section1}, which is based on a PDE argument (see, e.g., \cite{BHR}).

\begin{proposition}
	\label{1pr:2.2}
	Assume \eqref{eq:gKPP}--\eqref{eq:fstrongKPP}.
	Then, there exists a unique positive, bounded, $x$-independent and $y$-radial steady state $(U,V)$ for problem \eqref{1eq:1.1}.
	
	\begin{proof}
		The existence of a steady state with the properties of the statement follows from Proposition~\ref{1pr:2.1}. Let now $(U,V)$ be any such a steady state which depends only on the $y$ variable. As a first step, we prove that
		\begin{linenomath*}
		\begin{equation*}
		\inf_{\ov{B_R}} U>0 \quad \text{and} \quad \inf_{\R^{N-1}\setminus B_R} V>0.
		\end{equation*}
	\end{linenomath*}
		The elliptic strong maximum principle excludes that $U$ and $V$ vanish in the interior of their corresponding domains. Moreover, the fact that neither of them approaches $0$ at a point $y_0\in\p B_R$ can be excluded by using the Hopf boundary lemma. Indeed, if for example $U(y_0)=0$, then
		\begin{equation}
		\label{1eq:2.4bis}
		0> D\p_nU(y_0)=\nu V(y_0)\geq 0,
		\end{equation}
		a contradiction. The case  $V(y_0)=0$ can be excluded analogously. Finally, using the rotational symmetry of $(U,V)$ and studying the monotonicity of their profiles, it is possible to show as in Proposition~\ref{0pr:1.2} that $V(y)$ converges, as $|y|\to+\infty$, to the value $S>0$ defined in \eqref{eq:fKPP}. We have therefore proved that $U$ and $V$ have positive infimum.
		
		Assume now that there exist two $x$-independent and $y$-radial steady states $(U_i,V_i)$, $i\in\{1,2\}$. Since $(U_1,V_1)$ is positive and bounded, we have $\g(U_2,V_2)<(U_1,V_1)$ for $\g=0$, while $\g(U_2,V_2)>(U_1,V_1)$ for large $\g$. By continuity, if we set 
		\begin{linenomath*}
		\begin{equation*}
		\g^*:=\sup\left\{\g>0: \g(U_2,V_2)<(U_1,V_1) \right\},
		\end{equation*}
	\end{linenomath*}
		which is positive and finite, we have $U_1-\g^*U_2\geq 0$, $V_1-\g^*V_2\geq 0$, and
		\begin{equation}
		\label{1eq:2.5}
		\inf_{\ov {B_R}}\left( U_1-\g^*U_2\right)=0 \quad \text{or} \quad \inf_{\R^{N-1}\setminus B_R} \left(V_1-\g^*V_2\right)=0.
		\end{equation}
		We claim that $\g^*\geq 1$, which will entail that $(U_1,V_1)\geq(U_2,V_2)$. Then, by exchanging the roles of $(U_1,V_1)$ and $(U_2,V_2)$, we will conclude that the two steady states coincide. Assume by contradiction $\g^*<1$ and consider the several possibilities given by \eqref{1eq:2.5}. Suppose, for example, that
		$\g^*U_2$ touches $U_1$ from below at an interior point in $B_R$. Thanks to the contradiction hypothesis $\g^*<1$, it satisfies
		\begin{equation}
		\label{1eq:2.6}
		-D \D\left(\g^* U_2\right)=-D\g^*\D U_2=\g^*g\left(U_2\right)<g\left(\g^* U_2\right),
		\end{equation}
		i.e., it is a strict subsolution of the equation satisfied by $U_1$. This is impossible thanks to the strong maximum principle.
		
		In the same way we exclude the possibility of a contact point between $V_1$ and $\g^* V_2$ in $\R^{N-1}\setminus\ov{B_R}$, while  the existence of a contact point on $\p B_R$ can be excluded as in \eqref{1eq:2.4bis}.
		
		Finally, if there existed a sequence $(y_n)_{n\in\N}$ such that $|y_n|\to\infty$ and $\inf_{n\in\N} (V_1(y_n)-\g^*V_2(y_n))=0$, then, thanks to elliptic estimates,  the functions
		\[V_{n,1}(y):=V_1(y+y_n), \qquad  V_{n,2}(y):=V_2(y+y_n),\]
		which are bounded and bounded away from 0, would respectively converge locally uniformly in $\R^{N-1}$ to some bounded, non-negative functions $V_{\infty,1}$ and $V_{\infty,2}$ satisfying $-d\D V_{\infty,i}=f(V_{\infty,i})$ for $i\in\{1,2\}$ and such that
		$\g^*V_{\infty,2}$ touches $V_{\infty,1}$ from below at $0$. By reasoning as in \eqref{1eq:2.6}, we exclude also this last possibility.
	\end{proof}
\end{proposition}

As an immediate consequence of Propositions~\ref{1pr:2.1} and~\ref{1pr:2.2}, we derive the invasion result of Theorem~\ref{1th:3.3}.

\subsection{Construction of the asymptotic speed of propagation for \eqref{1eq:1.1}}
\label{1subsection:c^*_2}
The general scheme of construction of the asymptotic speed of propagation for problem \eqref{1eq:1.1} is the same as in Section~\ref{0section1}. Nonetheless, the different geometry entails some technical changes in the study of the linearization of \eqref{1eq:1.1} around $(0,0)$, which we detail hereafter. Due to the symmetry of the problem, we now look for positive, $y$-radial (super)solutions to such a system of the following type
\begin{equation}
\label{1eq:3.2}
\ov u(x,y)=e^{-\a(x- ct)}|y|^{-\tau}J_\tau(\b|y|), \qquad \ov v(x,\tilde y)=\g e^{-\a(x- ct)}|\tilde y|^{-\tau} K_\tau(\d|\tilde y|),
\end{equation}
where $J_\tau$ is the Bessel function of the first kind and $K_\tau$ the modified Bessel function of the second kind of order $\tau=\frac{N-3}{2}$ (observe that $\tau\geq -1/2$ is either integer or half-integer), while $\a,\b,\g,\d$ are positive constants to be determined.

The definitions and properties of the Bessel functions that we are going to use are collected in Appendix~\ref{sectionappA}.

Letting $j_{\tau}$ denote the first positive zero of $J_\tau$, we will require $\b R<j_\tau$, for $\ov u$ to be positive in $\ov\O$. Moreover, thanks to \eqref{1eq:A.1}, $\ov u$ is well defined and smooth in the whole domain~$\O$. On the other hand, $K_\tau(r)$ is positive and defined for all $r>0$.

Since $J_\tau$ (respectively $K_\tau$) satisfies the Bessel equations \eqref{1eq:A.-1} (respectively \eqref{2eq:A.1}), we have that $e^{\a(x- ct)}\ov u$ (respectively $e^{\a(x- ct)}\ov v$) is a positive solution of 
\begin{linenomath*}
	\begin{equation*}
-\D_y\phi=\b^2\phi \text{ in $\{|y|<R\}$} \qquad \text{(respectively of $\D_{\tilde y}\phi=\d^2\phi$ in $\left\{|\tilde y|>R\right\}$)}.
\end{equation*}
\end{linenomath*}

In addition, taking into account relations \eqref{1eq:A.2} and \eqref{2eq:A.4} that involve the derivatives of $J_\tau$ and $K_\tau$, we have that looking for positive solutions (respectively supersolutions) of the form \eqref{1eq:3.2} of the linearized problem around $(0,0)$  is equivalent to finding positive solutions of the following system for the parameters $\a,\b,\g,\d$ (respectively with the $\lq\lq="$ sign replaced by $\lq\lq\geq"$):
\begin{equation}
\label{1eq:3.5}
\left\{\begin{array}{l}
c\a-D\a^2+D\b^2=g'(0) \\
c\a-d\a^2-d\d^2=f'(0) \\
-D\b J_{\tau+1}(\b R)=\nu\g K_{\tau}(\d R)-\mu J_{\tau}(\b R) \\
d\g\d K_{\tau+1}(\d R)=\mu J_{\tau}(\b R)-\nu\g K_{\tau}(\d R).\\
\end{array}\right.
\end{equation}
From the third equation we obtain
\begin{equation}
\label{1eq:3.6}
\g=\frac{\mu J_{\tau}(\b R)-D\b J_{\tau+1}(\b R)}{\nu K_{\tau}(\d R)},
\end{equation}
thus, by letting
\begin{align}
\chi_u(\b)&:=\frac{\nu D\b J_{\tau+1}(\b R)}{\mu J_{\tau}(\b R)- D \b J_{\tau+1}(\b R)}=\frac{\nu D h_u(\b R)}{\mu R- D h_u(\b R)} \label{1eq:3.7.u}\\
\chi_v(\d)&:=d\d \frac{K_{\tau+1}(\d R)}{K_{\tau}(\d R)}=\frac{d}{R}h_v(\d R), \label{1eq:3.7.v}
\end{align}
where $h_u$ and $h_v$ are the functions defined, respectively, in \eqref{1eq:A.3} and \eqref{2eq:A.6}, the fourth equation in \eqref{1eq:3.5} can be equivalently written as
\begin{equation}
\label{1eq:3.8}
\chi_v(\d)=\chi_u(\b).
\end{equation}
Since $J_\tau$ is positive in $\left(0,j_\tau\right)$ and vanishes at $j_\tau$, and, 
as recalled in Appendix~\ref{sectionappA}, $j_\tau$ is increasing with respect to $\tau$, we can then call
\begin{equation}
	\label{1eq:3.ovb}
	\ov \b \text{ the smallest positive zero in $\left(0,\frac{j_\tau}{R}\right)$ of the function } \b\mapsto\mu J_{\tau}(\b R)\!-\!D \b J_{\tau+1}(\b R).
	\end{equation}
Hence, $\g$ in \eqref{1eq:3.6} is positive for $\b\in\left(0,\ov\b\right)$, and from \eqref{1eq:3.7.u} we see that the function $\chi_u(\b)$ is positive and well defined there.

Thanks to Lemma~\ref{le:A.1}, $h_u$ is increasing in $(0,j_\tau)$, while \eqref{1eq:A.6} gives that $h_u(0)=h'_u(0)=0$. By collecting all these properties, we have
\begin{linenomath*}
\begin{equation*}
\chi_u(0)=\chi'_u(0)=0, \qquad \chi_u(\b),\chi'_u(\b)>0 \text{ for $\b\in(0,\ov\b)$}, \qquad \lim_{\b\ua\ov\b}\chi_u(\b)=+\infty.
\end{equation*}
\end{linenomath*}
On the other hand, using \eqref{2eq:A.7} and \eqref{2eq:A.3}, we deduce that
\begin{equation}
\label{1eq:3.9}
\lim_{\d\da 0}\chi_v(\d)=\frac{d}{R}\max\{0,N-3\}, \qquad \lim_{\d\da 0}\chi'_v(\d)=\begin{cases}
d & \text{if $N=2,4$} \\
+\infty & \text{if $N=3$} \\
0 & \text{otherwise.} \\
\end{cases}
\end{equation}
Moreover, thanks to \eqref{2eq:A.6}, we have that $\chi_v$ is increasing, while \eqref{2eq:A.2} gives
\begin{linenomath*}
\begin{equation*}
\lim_{\d\ua+\infty}\frac{\chi_v(\d)}{\d}=d.
\end{equation*}
\end{linenomath*}
Let us call $\un\b$ the unique value of $\b\in[0,\ov\b)$ such that  $\lim_{\d\da0}\chi_v(\d)=\chi_u(\b)$.
From the first relation of \eqref{1eq:3.9}, we have that
\begin{equation}
\label{1eq:3.9bis}
\un\b=0 \,\text{ for $N=2,3$,} 
\end{equation}
and, for $N\geq 4$, by \eqref{1eq:3.7.u}--\eqref{1eq:3.8}
\begin{equation}
	\label{1eq:3.9bisbis}
	\un\b>0 \quad \text{ satisfies } \quad	h_u(\un\b R)=\frac{\mu R}{D}\frac{d(N-3)}{\nu R+ d(N-3)}.
\end{equation}
Thanks to all these properties and the implicit function theorem, there exists an analytic function $\d(\b):(\un\b,\ov\b)\to(0,+\infty)$ which satisfies $\chi_v(\d(\b))=\chi_u(\b)$. By differentiating this relation, we obtain 
\begin{equation}
\label{1eq:3.9ter}
\chi'_v(\d(\b))\d'(\b)=\chi_u'(\b),
\end{equation}
therefore $\d(\b)$ is increasing and, thanks to \eqref{1eq:3.9}, satisfies
\begin{equation}
\label{1eq:3.10}
\lim_{\b\da \un\b}\d(\b)=0, \quad \lim_{\b\ua \ov\b}\d(\b)=+\infty, \quad \lim_{\b\da \un\b}\d'(\b)=\begin{cases}0 & \text{if $N=2,3$} \\
\frac{\chi'_u(\un\b)}{d}>0 & \text{if $N=4$} \\
+\infty & \text{if $N\geq 5$.}  \end{cases}
\end{equation}

This analysis allows us to reduce the search for solutions (respectively, supersolutions) of \eqref{1eq:3.5} to the search for solutions of the following system for the variables $\a,\b$ and the parameter $c$ (respectively, the corresponding one with $\lq\lq\geq"$ instead of $\lq\lq="$)
\begin{equation}
\label{1eq:3.11}
\left\{\begin{array}{l}
c\a-D\a^2+D\b^2=g'(0) \\
c\a-d\a^2-d\d^2(\b)=f'(0). \\
\end{array}\right.
\end{equation}
In the first quadrant of the $(\b,\a)$ plane, the first equation describes an hyperbola $\Si_D(c)$ that can be parameterized through the graphs
\begin{equation}
\label{1eq:3.13}
\a_D^{\pm}(c,\b):=\frac{c\pm\sqrt{c^2-c_g^2+4D^2\b^2}}{2D},
\end{equation}
which are defined for $\b\geq\hat{\b}(c)$, where 
\begin{equation}
\label{1eq:3.13bis}
\hat{\b}(c):=\begin{cases}
0 & \text{if $c\geq c_g$} \\
\frac{\sqrt{c_g^2-c^2}}{2D} & \text{otherwise}.
\end{cases}
\end{equation}
Observe that $c\mapsto\hat{\b}(c)$ is continuous, non-increasing, and that
\[
\hat{\b}(c)
=\sqrt{\frac{g'(0)}{D}}  \text{ for $c=0$}, \quad \hat{\b}(c)\in\left(\!0,\sqrt{\frac{g'(0)}{D}}\right)  \text{ for $c\in(0,c_g)$}, \quad
\hat{\b}(c)=0 \text{ for $c\geq c_g$}.
\]
Note also that $\a_D^{-}\left(c,\sqrt{\frac{g'(0)}{D}}\right)=0$ for every $c>0$.

The second equation of \eqref{1eq:3.11} represents, instead, a curve which will be denoted by $\Si_d(c)$ and can be parameterized through the positive functions
\begin{equation}
\label{1eq:3.12}
\a_d^\pm(c,\b)=\frac{c\pm\sqrt{c^2-c_f^2-4d^2\d^2(\b)}}{2d}.
\end{equation}
Thus, $\a_d^\pm(c,\b)$ is not defined for $c< c_f$, for $c=c_f$ it consists of the point $\left(\un\b,\frac{c}{2d}\right)$, and, for $c>c_f$, if we set $\breve{\b}(c)\in(\un\b,\ov\b)$ to be the unique value of $\b$ such that $c^2=c_f^2+4d^2\d^2(\b)$, then $\a_d^{\pm}(c,\b)$ is defined for $\un\b\leq\b\leq\breve{\b}(c)$. Moreover, since $\d(\b)$ is increasing, we have that, in the domain of definition,
\begin{equation}
	\label{1eq:monotonicityb}
	\b\mapsto\a_d^+(c,\b) \text{ is decreasing} \qquad \text{and} \qquad \b\mapsto\a_d^-(c,\b) \text{ is increasing}.
	\end{equation}
To sum up, $\Si_d(c)$ looks like a deformed arc of circle which is symmetric about $\a=\frac{c}{2d}$ (see Figure~\ref{fig2}).

By analyzing the behavior of these curves with respect to $c$, we have that, for $i\in\{d,D\}$ and within their domain of definition in the first quadrant of the $(\b,\a)$ plane, 
\begin{equation}
\label{1eq:3.14}
c\mapsto \a_i^+(c,\b) \text{ increases to $+\infty$}, \quad 			c\mapsto \a_i^-(c,\b) \text{ decreases to $0$}.
\end{equation}

Finding solutions of \eqref{1eq:3.11} corresponds to find intersections between the curves $\Si_D(c)$ and $\Si_d(c)$, while finding supersolutions amounts to find intersections between the region lying between the graphs of $\a_D^\pm$, which will be denoted by $\mathcal{S}_D(c)$, and the region lying between the graphs of $\a_d^\pm$, which will be denoted by $\mathcal{S}_d(c)$.

Observe that a necessary condition for the existence of this kind of (super)solutions is $c\geq c_f$, for $\mathcal{S}_d(c)$ to be non-empty. If the curve $\Si_d(c)$ appears, for $c=c_f$, inside $\mathcal{S}_D(c)$, which is equivalent to
\begin{equation}
\label{1eq:3.15}
\hat\b(c_f)\leq\un\b, \qquad \a_D^-(c_f,\un\b)\leq \frac{c_f}{2d}\leq \a_D^+(c_f,\un\b),
\end{equation}
 then supersolutions of the considered form will exist for all $c\geq c_f$ and, in this case, we set $c^*:=c_f$. Otherwise, we define $c^*$ as the first value of $c> c_f$ for which $\Si_D(c)$ and $\Si_d(c)$ intersect. Observe that such a value of $c$ exists by continuity and thanks to \eqref{1eq:3.14}. Moreover,  \eqref{1eq:3.14} also guarantees that supersolutions to \eqref{1eq:3.11} will exist for every $c\geq c^*$. The definition of $c^*$ in a situation where $c^*>c_f$ is depicted in Figure~\ref{fig2}.
 
 \begin{figure}[ht]
 		\begin{center}
 			\begin{tabular}{ccc}
 				\includegraphics[scale=0.36]{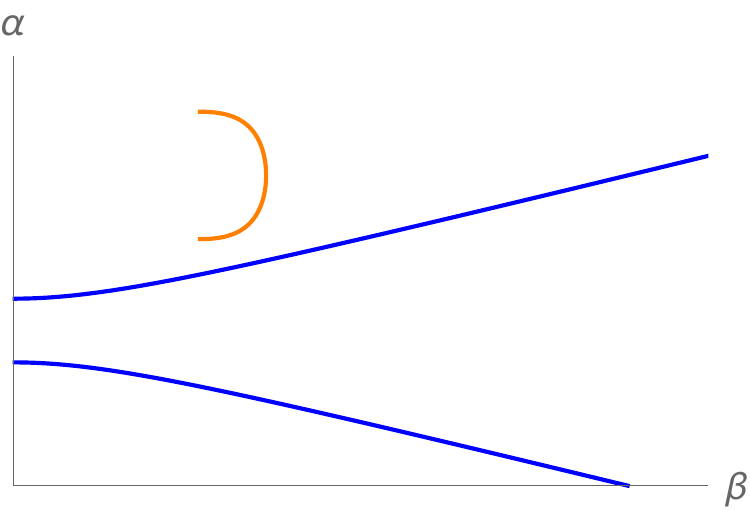} & \includegraphics[scale=0.36]{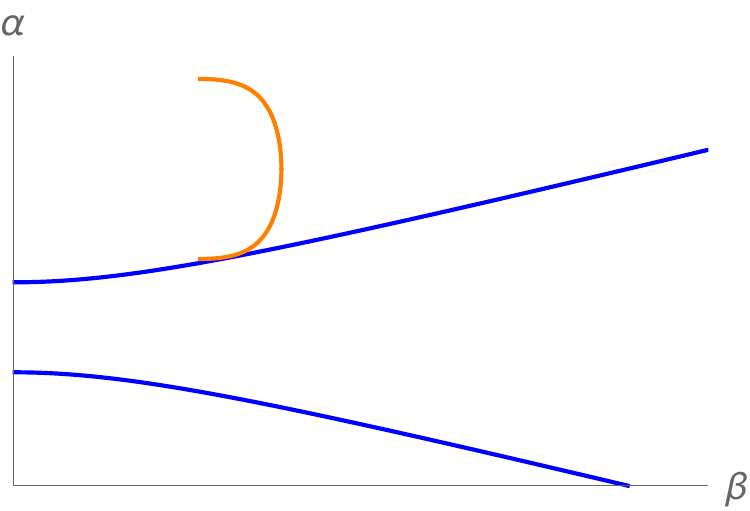} &
 				\includegraphics[scale=0.36]{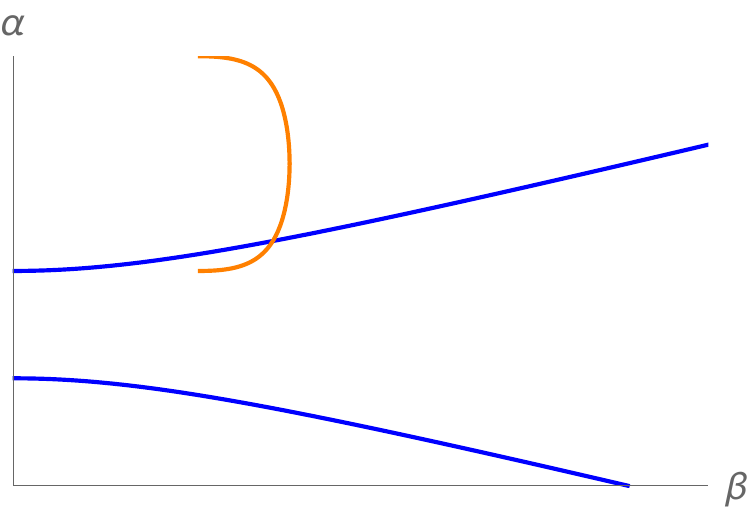}
 			\end{tabular}
 		\caption{The position of the curves $\Si_D(c)$ (blue) and $\Si_d(c)$ (orange) when $N\geq 4$ (thus $\un\b>0$) and $c^*>c_f$: case $c_g<c<c^*$ (left), $c=c^*$ (center), $c>c^*$ (right).} \label{fig2}
 	\end{center} 
 \end{figure}
 
 We claim that
 \begin{equation}
 	\label{eq:claim}
 \text{if $N\leq 5$ and $c^*>c_f$, then the curves $\Si_D(c^*)$ and $\Si_d(c^*)$ are tangent.}
 \end{equation}
 Observe that $\Si_d(c^*)$ is a smooth curve with endpoints $\left(\un\b,\frac{c^*\pm\sqrt{(c^*)^2-c_f^2}}{2d}\right)$. Then, owing to the definition of $c^*$ and the regularity of the curves, the claim will follow if we show that any intersection point $(\b^*,\a^*)$ between $\Si_D(c^*)$ and $\Si_d(c^*)$ satisfies $\b^*>\un\b$.
 
 Let us assume by contradiction that $\b^*=\un\b$ (we implicitly suppose that $\hat\b(c^*)\leq\un\b$, otherwise no intersection can occur at $\un\b$). By evaluating some of the derivatives of the curves with respect to $\b$ at $(c^*,\un\b)$, we will show that $c^*$ is not the first value of $c$ for which the curves intersect, obtaining a contradiction.
 
 We consider the following derivatives
 \begin{gather}
 	\p_\b\,\a_d^\pm(c,\b)=\mp\frac{2d\d(\b)\d'(\b)}{\left(c^2-c_f^2-4d^2\d^2(\b)\right)^{1/2}}, \label{1eq:3.18a}\\ \p^2_{\b\b}\,\a_d^\pm(c,\b)=\mp2d\frac{\left(\d'(\b)^2+\d(\b)\d''(\b)\right)\left(c^2-c_f^2\right)-4d^2\d^3(\b)\d''(\b)}{\left(c^2-c_f^2-4d^2\d^2(\b)\right)^{3/2}},
 	 \label{1eq:3.18b}\\
 	\p_\b\,\a_D^\pm(c,\b)=\pm\frac{2D\b}{\left(c^2-c_g^2+4D^2\b^2\right)^{1/2}}, \label{1eq:3.19a}\\ \p^2_{\b\b}\,\a_D^\pm(c,\b)=\pm\frac{2D\left(c^2-c_g^2\right)}{\left(c^2-c_g^2+4D^2\b^2\right)^{3/2}}, \label{1eq:3.19b}
 \end{gather}
 and assume for example that $\a_d^-(c^*,\un\b)=\a_D^+(c^*,\un\b)$ (the case $\a_d^+(c^*,\un\b)=\a_D^-(c^*,\un\b)$ can be treated analogously).
 
For $N\in\{2,3\}$, i.e., when $\un\b=0$, \eqref{1eq:3.10}, \eqref{1eq:3.18a} and \eqref{1eq:3.19a} give that $\p_\b\,\a_d^-(c^*,0)=0=\p_\b\,\a_D^+(c^*,0)$, thus we consider the second derivatives of such curves with respect to $\b$. To this end, we have to compute $\lim_{\b\da0}\d''(\b)$: we invert \eqref{1eq:3.8}, obtaining
\begin{linenomath*}
\begin{equation*}
\d(\b)=\frac{1}{R}k_v\left(\frac{R}{d}\chi_u(\b)\right),
\end{equation*}
\end{linenomath*}
where $k_v:(0,+\infty)\to(0,+\infty)$ is the inverse function of $h_v$ introduced in \eqref{2eq:A.6}, and  elementary differentiations give
\begin{equation}
\label{1eq:3.20}\d''(\b)=\frac{R}{d^2}k_v''\left(\frac{R}{d}\chi_u(\b)\right)\chi_u'^2(\b)+\frac{1}{d}k_v'\left(\frac{R}{d}\chi_u(\b)\right)\chi_u''(\b),
\end{equation}
with
\begin{equation}
\label{1eq:3.22} k_v'(h_v(r))=\frac{1}{h_v'(r)}, \quad \text{and} \quad k_v''(h_v(r))=-\frac{h_v''(r)}{h_v'(r)^3}.
\end{equation}
On the one hand, \eqref{1eq:A.6}, entails $\chi_u'(0)=0$ and $\chi_u''(0)\in(0,+\infty)$; on the other hand,
for $N=2$, \eqref{2eq:A.5} implies that $k_v(r)=r$,  while, for $N=3$, \eqref{2eq:A.3} and \eqref{1eq:3.9} imply $\lim_{r\da0}k_v'(r)=\lim_{r\da0}k_v''(r)=0$. Thus, \eqref{1eq:3.20} gives, in both cases, $\lim_{\b\da 0}\d''(\b)\in[0,+\infty)$ and, as a consequence, $\lim_{\b\da 0}\d(\b)\d''(\b)=0$.
From \eqref{1eq:3.10}, \eqref{1eq:3.18b} and \eqref{1eq:3.19b}, we obtain
\[\p^2_{\b\b}\a_d^-(c^*,0)=0<\frac{2D}{\sqrt{c^{*2}-c_g^2}}=\p^2_{\b\b}\a_D^+(c^*,0),\]
thus, by continuity, $\a_d^-(c^*,\b)$ will lie below $\a_D^+(c^*,\b)$ for $\b\sim 0$. But $\a_d^+(c^*,\b)$ lies, again for $\b\sim 0$, above $\Si_D(c^*)$, and this means that $c^*$ is not the first value of $c$ for which $\Si_d(c)$ and $\Si_D(c)$ intersect, which gives the desired contradiction.

For $N=4$, we have $\un\b>0$ and \eqref{1eq:3.10} gives that $\lim_{\b\da\un\b}\d(\b)\d'(\b)=0$. The same holds true for $N=5$ since, from \eqref{1eq:3.9ter}, we have
\begin{equation}
\label{1eq:3.24}
\frac{\chi_v'(\d(\b))}{\d(\b)}\d(\b)\d'(\b)=\chi_u'(\b),
\end{equation}
and, as $\b\da\un\b$, the right-hand side converges to $\chi_u'(\un\b)>0$, while the first factor on the left-hand side converges to $\chi_v''(0)$, which, thanks to \eqref{2eq:A.3}, equals $+\infty$.
 Thus, for $N=4,5$, \eqref{1eq:3.18a} and \eqref{1eq:3.19a} give
\[\p_{\b}\a_d^-(c^*,\un\b)=0<\frac{2D\un\b}{\sqrt{c^{*2}-c_g^2+4D^2\un\b^2}}=\p_{\b}\a_D^+(c^*,\un\b),\]
and we obtain a contradiction by reasoning on the relative positions of the curves, as before. This concludes the proof of \eqref{eq:claim}.

We point out that, for $N\geq 6$, thanks again to \eqref{2eq:A.3}, we have $\chi_v''(0)=\frac{2dR}{N-5}$, thus, using \eqref{1eq:3.24}, $\lim_{\b\da\un\b}\d(\b)\d'(\b)\in(0,+\infty)$, and we cannot derive an immediate contradiction as above, valid for general values of the parameters. Actually, the following figure shows that the result of the claim \eqref{eq:claim} may not hold true in dimensions higher than 5, as shown in the example of Figure \ref{fig3}.

\begin{figure}[ht]
	\begin{center}
			\includegraphics[scale=0.5]{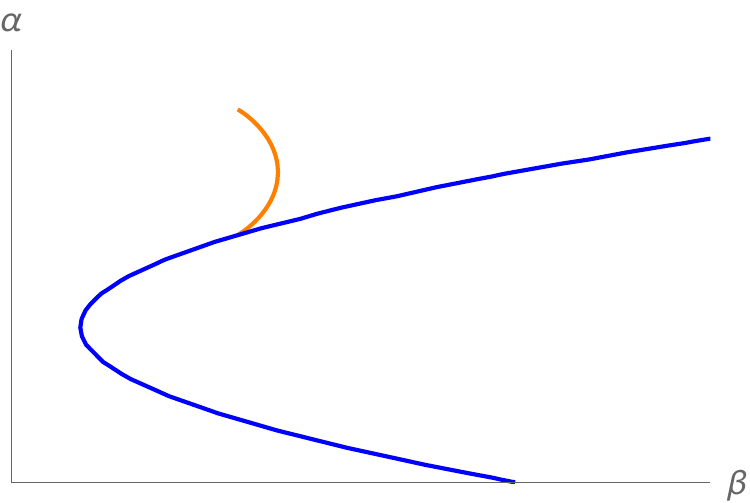}
		\caption{Example of non-tangency of the curves $\Si_D(c^*)$ (blue) and $\Si_d(c^*)$ (orange), obtained for the following values of the parameters: $N=6$, $d = 1$, $D= 2$, $\mu = 1$, $\nu=0.9$, $f'(0)=0.1$, $g'(0)=0.38$, $R=5$.} \label{fig3}
	\end{center} 
\end{figure}

As done in Section~\ref{0section1}, the next step is the construction of compactly supported, small, stationary, generalized subsolutions of the problem with an additional drift term in the $x$ direction and speed $c$ for $c<c^*$, $c\sim c^*$. Such a problem reads as follows
\begin{equation}
	\label{0eq:1.13.cyl}
	\left\{
	\begin{array}{ll}
		\p_t u-D\D u+c\p_x u=g(u) & (x,y)\in\O, \quad t>0 \\
		\p_t v-d\D v+c\p_x v=f(v) & (x,y)\in\R^N\setminus\ov \O, \quad t>0 \\
		D\, \p_n u = \nu v-\mu u & (x,y)\in\p\O, \quad t>0\\
		-d\, \p_n v = \mu u-\nu v & (x,y)\in\p\O, \quad t>0,
	\end{array}\right.
\end{equation}
and our first result covers the case $c^*=c_f$.

\begin{proposition}
	\label{1pr:3.1}
	For every $0<c<c_{f}$, there exist $\e_0=\e_0(c)>0$ and $L=L(c)>R+1$ such that, for $0<\e<\e_0$,  problem \eqref{0eq:1.13.cyl} admits a non-negative, stationary, generalized subsolution of the form $\e(0,\un v)$, with $\un v$ bounded and compactly supported.
	\begin{proof}
		Fix $c<c_f$ and consider $\phi_1$, the positive eigenfunction associated to the principal eigenvalue $\l_1(L)$ of problem \eqref{1eq:2.2}, normalized so that $\max_{A_L}\phi_1=1$.	Take $\t>0$ small enough and $L$ large enough so that
		\begin{linenomath*}
		\begin{equation*}
		\a:=\sqrt{\frac{1}{d}\left(f'(0)-\frac{c^2}{4d}-\t-d\l_1(L)\right)}
		\end{equation*}
	\end{linenomath*}
		is real (this is possible because $c^2<c_f^2=4df'(0)$ and $\l_1(L)\to 0$ as $L\to+\infty$).
		
		We then define $\un v:=e^{\frac{c}{2d}x}\cos(\a x)\phi_1(y)$ for $(x,y)\in(-\frac{\pi}{2\a},\frac{\pi}{2\a})\times A_L$ and $\un v:= 0$  otherwise in $\R^N\setminus{\overline{\O}}$, and the proof follows as the one of Proposition~\ref{0pr:1.3}, once observed that $\un v$ solves $-d\D v+c\p_xv=(f'(0)-\t)v$ in $(-\frac{\pi}{2\a},\frac{\pi}{2\a})\times A_L$ and vanishes on the boundary of such a set.
	\end{proof}
\end{proposition}

The definition of the generalized subsolutions in the case $c^*>c_f$ relies on the previous construction of the curves $\Si_D$ and $\Si_d$, which involve the Bessel functions. The key point will be property \eqref{eq:claim}; this is why the limitation $N\leq 5$ in Theorem~\ref{1th:3.3} arises.

\begin{proposition}
	\label{1pr:3.2}
	Assume $N\leq 5$ and $c^*>c_f$. Then, for $c<c^*$, $c\sim c^*$, there exist arbitrarily small, non-negative, compactly supported, stationary, generalized subsolutions of	\eqref{0eq:1.13.cyl}. 
	
	\begin{proof}
	As in Proposition~\ref{0pr:1.5}, we look for stationary solutions of a penalized version of the linearization of \eqref{0eq:1.13.cyl} around $(0,0)$, i.e., with $g'(0)$ and $f'(0)$ respectively replaced by $g'(0)-\t$ and $f'(0)-\t$, $\t>0$, $\t\sim 0$. In addition, we require  $v$ to vanish at $|y|=L\gg R$. Thus, in analogy with the construction previous developed in this section, such solutions will be sought of the form
\begin{linenomath*}
	\begin{equation*}
	\begin{aligned}
\un u(x,y)&= e^{\a x}|y|^{-\tau}\frac{\nu J_{\tau}(\b |y|)}{\mu J_{\tau}(\b R)-D\b J_{\tau+1}(\b R)}, \\
 \un v(x,\tilde y)&=e^{\a x}|\tilde y|^{-\tau}\frac{K_{\tau}(\d |\tilde y|)I_{\tau}(\d L)-I_{\tau}(\d |\tilde y|)K_{\tau}(\d L)}{K_{\tau}(\d R)I_{\tau}(\d L)-I_{\tau}(\d R)K_{\tau}(\d L)},
 \end{aligned}
\end{equation*}
\end{linenomath*}
where $I(r)$ is the modified Bessel function of the first kind (see Appendix~\ref{sectionappA}). Observe that $\un v$ is positive for $|\tilde y|<L$, since the function $r\mapsto r^{-\tau}(K_{\tau}(\d r)I_{\tau}(\d L)-I_{\tau}(\d r)K_{\tau}(\d L))$ is decreasing thanks to \eqref{2eq:A.4} and \eqref{3eq:A.2}, and vanishes for $r=L$.

	With similar computations to those of the first part of this section, the search for such solutions can be reduced to the search for intersections between the curve $\tilde\Si_D(c)$, which is obtained by replacing $g'(0)$ in $\Si_D(c)$  by $g'(0)-\t$, and the curve
	\[\tilde{\Si}_d(c,L):=\{(\b,\tilde\a_d^\pm(c,\b,L))\},\]
	where 
	\begin{equation}
	\label{eq:tildead}
	\tilde\a_d^\pm(c,\b,L):=\frac{c\pm\sqrt{c^2-c_f^2+4d\t-4d^2\tilde{\d}^2(\b,L)}}{2d},
	\end{equation}
	with $\tilde{\d}(\b,L)$ implicitly defined through
	\[d\tilde\d\frac{K_{\tau+1}(\tilde\d R)I_{\tau}(\tilde\d L)+I_{\tau+1}(\tilde\d R)K_{\tau}(\tilde\d L)}{K_{\tau}(\tilde\d R)I_{\tau}(\tilde\d L)-I_{\tau}(\tilde\d R)K_{\tau}(\tilde\d L)}=\frac{\nu D\b
		J_{\tau+1}(\b R)}{\mu J_{\tau}(\b R)-D\b J_{\tau+1}(\b R)},\]
	and $\b\in(0,\ov\b)$ varies in the range for which the argument of the square root in \eqref{eq:tildead} is positive.
	
	Since $\frac{K_{\tau}(r)}{I_{\tau}(r)}\to 0$ as $r\to+\infty$, thanks to \eqref{2eq:A.2} and \eqref{3eq:A.3}, it is easy to see that, as $L\to +\infty$, $\tilde{\d}(\cdot,L)$ converges,  locally uniformly for $\b\in(\un\b,\ov\b)$ to the $\d(\cdot)$ implicitly defined by \eqref{1eq:3.7.u}--\eqref{1eq:3.8}. Hence, the curves $\tilde\Si_D(c)$ and $\tilde\Si_d(c,L)$ converge, as $L\to+\infty$ and $\t\to0$, locally uniformly in their domain, to the curves $\Si_D(c)$ and $\Si_d(c)$ used to define $c^*$.
	
	As shown in the proof of \eqref{eq:claim}, the tangency point between $\Si_D(c^*)$ and $\Si_d(c^*)$ satisfies $\b^*>\un\b$, thus, by the analyticity of $\a_D^\pm$ and $\a_d^\pm$, there exists a closed disk $\mc D$, centered at $\left(\b^*,\a^*\right)$ and contained in $(\un\b,\ov\b)\times\R$, in which $\left(\b^*,\a^*\right)$ is the unique intersection point between $\Si_D(c^*)$ and $\Si_d(c^*)$. Then, calling $\tilde{\mc S}_D(c)$ (respectively $\tilde{\mc S}_d(c,L)$) the closed region lying in the first quadrant of the $(\b,\a)$ plane between the two branches that form $\tilde\Si_D(c)$ (respectively between those that form $\tilde\Si_d(c,L)$), $\tilde{\mc S}_D(c)$ and $\tilde{\mc S}_d(c,L)$ do not intersect on $\p\mc D$   for $c\sim c^*$, $L\gg 1$ and $\t\sim 0$. On the other hand, by the strict monotonicity with respect to $c$, for any $c_1<c^*<c_2$, the restrictions to $\mc D$ of $\tilde{\mc S}_D(c)$ and $\tilde{\mc S}_d(c,L)$ do not intersect for $c=c_1$, while they do for $c=c_2$. It follows that, for $L\gg 1$ and $\t\sim 0$, there exists $\tilde c^*(L,\t)$ such that the regions $\tilde{\mc S}_D(c)$ and $\tilde{\mc S}_d(c,L)$ do not intersect for $c<\tilde c^*(L,\t)$, are tangent for $c=\tilde c^*(L,\t)$ and intersect for $c>\tilde c^*(L,\t)$.
	
	By applying Rouch\'e's theorem as in the proof of Proposition~\ref{0pr:1.5}, it is possible to find complex intersections for $c<\tilde c^*(L,\t)$, $c\sim \tilde c^*(L,\t)$ and, consequently, to construct the desired subsolutions, again along the same lines of the proof of Proposition~\ref{0pr:1.5}.
	
	Finally, by using the behavior of the curves $\tilde\Si_D(c)$ and $\tilde\Si_d(c,L)$ as $L\to+\infty$ and $\t\to 0$, we can show that $\tilde c^*(L,\t)\to c^*$ and conclude that the desired subsolutions exists for $c<c^*$, $c\sim c^*$.
	\end{proof}
	\end{proposition}

With these elements, the proof of parts \eqref{1th:3.3.i} and \eqref{1th:3.3.ii} of Theorem~\ref{1th:3.3}, which assert that $c^*$ is the asymptotic speed of propagation of \eqref{1eq:1.1} in the $x$ direction, follows in the same way as the one of Theorem~\ref{0th:1.4}.

\subsection{Properties of the asymptotic speed of propagation of \eqref{1eq:1.1}}
\label{1section4}
The main goal of this section is to prove, among other results, the rest of Theorem \ref{1th:3.3} and Theorem~\ref{1th:5.1}, i.e., we will study the qualitative behavior of the speed of propagation~$c^*$ as several parameters appearing in \eqref{1eq:1.1} vary.  For this reason, we will explicitly point out the dependence of the quantities introduced in the previous section on the parameters taken into consideration in each case. Moreover, unless otherwise specified, we will consider $2\leq N\leq 5$, i.e., the values of the dimensions for which we have been able to establish the existence of $c^*$ through a precise geometrical characterization.

By construction, we know that $c^*\geq c_f$. The first result of this section establishes for which values of the parameters equality holds, i.e., when the cylinder has no effect on the propagation and when, on the contrary, the presence of diffusion and reaction heterogeneities inside the cylinder enhances the global propagation with respect to the homogeneous case.

\begin{proposition}
	\label{1pr:4.1}
	The asymptotic speed of propagation of problem \eqref{1eq:1.1} satisfies $c^*>c_f$ if and only if
	\begin{equation}
	\label{1eq:4.1new}
	D(f'(0)-d\un\b^2(D))> d(2f'(0)-g'(0)),
	\end{equation}
	where $\un\b(D)$ is defined by \eqref{1eq:3.9bis}--\eqref{1eq:3.9bisbis}.
	More precisely:
	\begin{enumerate}[(i)]
		\item \label{1pr:4.1i} if $N=2,3$, $c^*(D)>c_f$ if and only if
		\begin{equation}
		\label{1eq:4.1bis}
		\frac{D}{d}>2-\frac{g'(0)}{f'(0)};
		\end{equation}
		\item \label{1pr:4.1ii} if $N=4,5$ and $R\sqrt{\frac{f'(0)}{d}}\geq j_\tau$, then there exists $\ov D\geq 0$ such that $c^*(D)>c_f$ if and only if $D>\ov D$. Moreover, if $2f'(0)\leq g'(0)$, then $\ov D=0$, otherwise $\ov D>0$;
		\item \label{1pr:4.1iii} if $N=4,5$ and $R\sqrt{\frac{f'(0)}{d}}<j_\tau$, then $M:=\min_{D>0}D(f'(0)-d\un\b^2(D))$ exists, is negative, and there holds:
		\begin{enumerate}
			\item if $2f'(0)-g'(0)<\frac{M}{d}$, then $c^*(D)>c_f$ for all $D>0$;
			\item if $\frac{M}{d}\leq 2f'(0)-g'(0)<0$, then there exist $0<D_1\leq D_2<+\infty$ such that $c^*(D)>c_f$ if and only if $D<D_1$ or $D>D_2$;
			\item if $2f'(0)-g'(0)\geq 0$, then there exists $\ov D> 0$ such that $c^*(D)>c_f$ if and only if $D>\ov D$.
		\end{enumerate}
	(The quantities $D_1$, $D_2$, $\ov D$ depend on $d$, $f'(0)$, $g'(0)$, $\mu$, $\nu$, $R$, $N$.)
	\end{enumerate}

	\begin{proof}
		As already pointed out in Section~\ref{1subsection:c^*_2}, $c^*=c_f$ if and only if \eqref{1eq:3.15} holds true, that is, if and only if \eqref{1eq:4.1new} fails.	Part \eqref{1pr:4.1i} then directly follows by recalling from \eqref{1eq:3.9bis} that $\un\b(D)\equiv 0$ if $N=2,3$.
		
		When $N=4,5$, instead, we have to study the function $\un\b(D)>0$. From its definition in \eqref{1eq:3.9bisbis}, we have
		\begin{equation}
			\label{1eq:4.4}
			\un\b(D)=\frac{1}{R}k_u\left(\frac{\mu R}{D}\frac{d(N-3)}{\nu R+ d(N-3)}\right),
			\end{equation}
		where $k_u:(0,+\infty)\to(0,j_\tau)$ is the inverse of the function $h_u$ introduced in \eqref{1eq:A.3}. Thus, $\un\b(D)$ is continuous, positive, decreasing, and satisfies
		\begin{equation}
			\label{1eq:4.4bis}
			\lim_{D\da0}\un\b(D)=\frac{j_\tau}{R}, \qquad \lim_{D\to+\infty}\un\b(D)=0.
			\end{equation}
			As a consequence, the function $\zeta(D):=D\left(f'(0)-d\un\b^2(D)\right)$ converges to $0$ as $D\da 0$ and to $+\infty$ as $D\to+\infty$. Moreover, $\zeta(D)$ is positive for all $D>0$ (thus increasing, as product of two positive, increasing functions) if and only if $R\sqrt{\frac{f'(0)}{d}}\geq j_\tau$. Statement \eqref{1pr:4.1ii} then immediately follows.

		For statement \eqref{1pr:4.1iii}, we show that $\zeta$ is convex. Indeed, from \eqref{1eq:4.4}, we have
		\begin{linenomath*}
		\begin{align*}
\zeta''(D)&=-2d\left(2\un\b(D)\un\b'(D)+D\un\b'^{2}(D)+D\un\b(D)\un\b''(D)\right) \notag \\
&=-2d\frac{\mu^2A^2}{D^3}\left(k_u'^{2}\left(\frac{\mu R A}{D}\right)+k_u\left(\frac{\mu R A}{D}\right)k_u^{''}\left(\frac{\mu R A}{D}\right)\right) 
		\end{align*}
	\end{linenomath*}
		where we have set $A:=\frac{d(N-3)}{\nu R+d(N-3)}\in(0,1)$. Thus, thanks to the analogous relation of \eqref{1eq:3.22} satisfied by $k_u$, showing that $\zeta''(D)>0$ for $D>0$ amounts to prove that $\frac{rh_u''(r)-h_u'(r)}{h_u'^3(r)}>0$ for $r\in~(0,j_\tau)$, which holds true, since both the numerator and the denominator are positive, the former by Proposition~\ref{pr:A.3} and the latter by Lemma~\ref{le:A.1}.

		Thus, with the assumptions of part \eqref{1pr:4.1iii}, $\zeta'(0)<0$, and there exists $D_0>0$ such that $\zeta(D)$, which still vanishes for $D=0$, is negative for $D<D_0$, positive for $D>D_0$ and has a unique critical point which is a global minimum. With this analysis, the proof of the three sub-cases readily follows.
	\end{proof}
\end{proposition}

\begin{remark}
	\label{1re:4.2}
	\begin{enumerate}[(a)]
		\item \label{1re:4.20} We observe that condition \eqref{1eq:4.1new} for the enhancement of $c^*$ can be equivalently written as
		\[
		\frac{D}{d}>2-\frac{g'(0)}{f'(0)}+\frac{D\un\b^2(D)}{f'(0)}.
		\]
		By comparing this condition with the corresponding one for the road-field problem \eqref{1eq:5.1bis}, we obtain the inclusion $\mc E\subseteq(D_{\mathrm{rf}},+\infty)$ stated in Theorem~\ref{1th:3.3}\eqref{1th:3.3.iii}.
		
		Moreover, by taking into account the definition of $\un\b$ given in \eqref{1eq:3.9bis}--\eqref{1eq:3.9bisbis} and the different cases in Proposition~\ref{1pr:4.1}, we immediately obtain the proof of the other properties in Theorem~\ref{1th:3.3}\eqref{1th:3.3.iii}.

		\item \label{1re:4.2ii} For the values of the parameters detailed in case \eqref{1pr:4.1iii}(b) of Proposition~\ref{1pr:4.1}, which include, for example, $R\sim 0$ and $2f'(0)<g'(0)\leq2f'(0)-\frac{M}{d}$ (observe that $M$ does not depend on $g'(0)$), we have that $c^*(D)$ is decreasing in certain ranges of $D$, in contrast with the speed of propagation $c^*_\infty$ for the problem in adjacent half-spaces \eqref{0eq:1.1} (see Proposition~\ref{0pr:1.6}\eqref{0pr:1.6.ii}) and the speed $\cBRR$ of \eqref{1eq:BRR2} (see \cite{BRR2}).
		
		For the range of parameters described above, a possible interpretation of the mechanism behind this surprising phenomenon  is the following: since $g'(0)>2f'(0)$, the reaction inside the cylinder is more advantageous for the propagation than the one in the exterior. Thus, if $D$ is very small, the density $u$ does not reach the boundary of the cylinder too fast and uses the reaction $g$ to speed up the global speed of propagation of the system. On the other hand, if $D$ is sufficiently large, the cylinder becomes more convenient also for the diffusion; as a consequence, the speed of propagation in the axial direction is enhanced by both the diffusion and the reaction inside the cylinder. For intermediate values of $D$, instead, the diffusion mechanism inside the cylinder is not advantageous enough, and - at the same time - the density $u$ reaches the boundary of the cylinder too fast, without being able to take advantage of the reaction inside the cylinder.
		
		Finally, we point out that case \eqref{1pr:4.1iii}(b) of Proposition~\ref{1pr:4.1} in particular establishes that the set of $D$'s for which the speed of propagation is enhanced can be disconnected. This is another remarkable difference with respect to the speed $c^*_\infty$  (compare, e.g., with Figure~\ref{0fig:1.1}).

		\item \label{1re:4.2iiibis} As observed above, 
		the value of $c^*$ as well as the condition  for enhancement, i.e., $c^*>c_f$   depend on the dimension $N$. Nonetheless, such a dependence is lost in the limit $R\to+\infty$, 
		where we recover the quantity $c^*_{\infty}$ defined in \eqref{0eq:1.10},
		as stated in Theorem~\ref{1th:3.3}\eqref{1th:3.3.vi} and proved later in this section.
		
		\item \label{1re:4.2iv} By combining the construction of supersolutions of Section~\ref{1subsection:c^*_2} with Proposition~\ref{1pr:3.1}, we can give, for some values of the parameters, a characterization of $c^*$ also for $N\geq 6$. Indeed, we obtain that, if the complementary condition of \eqref{1eq:4.1new} holds true, which reads
			\begin{equation}
			\label{1eq:4.1old}
			D(f'(0)-d\un\b^2(D))\leq d(2f'(0)-g'(0)),
		\end{equation}
		 then $c^*=c_f$ is the asymptotic speed of propagation of problem \eqref{1eq:1.1}.
	\end{enumerate}
\end{remark}

Although we are not able to precisely characterize  $c^*$ for $N\geq 6$, Remark~\ref{1re:4.2}\eqref{1re:4.2iv} gives a sufficient condition, with no restriction on the spatial dimension, for $\O$ not to have any effect on the propagation. In the following proposition, we use such a result to show that, fixing all the other parameters, $c^*(N)=c_f$ for $N$ sufficiently large. This can be interpreted by considering any cylinder $\R\times B_{N-1}(r)$, with $r>R$ so that it contains $\O$, and observing that the ratio of the volumes of the sections is given by $\left(\frac{r}{R}\right)^{N-1}$, entailing that the volume that has to be filled by the solution of \eqref{1eq:1.1} grows to $+\infty$ as $N\to+\infty$, or, the other way around, the relative volume of $\O$ becomes negligible.

\begin{proposition}
	\label{1pr:4.2bis}
	The other parameters being fixed, there exists $N_0\geq 4$ such that, for every $N\geq N_0$, $c^*(N)=c_f$.
	\begin{proof}
		Observe that the argument of $k_u$ in \eqref{1eq:4.4} converges to $\mu R/D$ as $N\to+\infty$. In addition, \eqref{1eq:A.8}  implies that $h_u$ converges to $0$ locally uniformly in $(0,+\infty)$ as $N\to+\infty$, thus its inverse $k_u$ converges to $+\infty$ locally uniformly $(0,+\infty)$. This, together with \eqref{1eq:4.4}, shows that
		\begin{equation}
		\label{1eq:4.4quater}
		\un\b(N)\to+\infty \quad \text { as $N\to+\infty$,}
		\end{equation}
		 and Remark~\ref{1re:4.2}\eqref{1re:4.2iv} allows us to conclude.
		\end{proof}
	\end{proposition}

We now pass to the study of the asymptotic behavior of $c^*$ as $D$ converges to $0$ or $+\infty$. 

	\begin{proof}[Proof of Theorem~\ref{1th:3.3}\eqref{1th:3.3.iv} and \eqref{1th:3.3.v}]
		For (iv), observe that \eqref{1eq:4.4} and Proposition~\ref{1pr:4.1} give that $c^*(D)>c_f$ for large $D$. First of all, we prove that $c^*(D)$ is unbounded for $D\to+\infty$. Assume by contradiction that this is not the case. Then, for large $D$, $c^*(D)<c_g=2\sqrt{Dg'(0)}$ and $\frac{c^*(D)}{2D}<\frac{c^*(D)}{2d}$. These considerations entail that, for $c=c^*(D)$, the curves $\a_d^-(c,\b)$ and $\a_D^+(c,\b)$ are tangent at a point satisfying $\b>0$.
		
		In order to get a contradiction we perform the change of variable $\b'=\b\sqrt{D}$. In this way, we have that the curve
		\begin{linenomath*}
		\begin{equation*}
			\a_D^+\left(c^*(D),\frac{\b'}{\sqrt{D}}\right)=\frac{c^*(D)+\sqrt{c^{*2}(D)-4D(g'(0)-\b'^2)}}{2D}
		\end{equation*}
	\end{linenomath*}
		is defined for $\b'$ greater than or equal to a value $\hat\b'(c^*(D))$ converging to $\sqrt{g'(0)}$ as $D\to+\infty$, and it goes to $0$ locally uniformly in $\b'$ as $D\to+\infty$.
		Regarding the other curve, we have to study the behavior of $\d\left(\frac{\b'}{\sqrt{D}},D\right)$, which, according to \eqref{1eq:3.7.u} and \eqref{1eq:3.8}, is given through
		\begin{linenomath*}
		\begin{equation*}
			\chi_v(\d)=\frac{\nu D h_u\left(\frac{\b'}{\sqrt{D}} R\right)}{\mu R - D h_u\left(\frac{\b'}{\sqrt{D}} R\right)}.
		\end{equation*}
	\end{linenomath*}
		Thanks to \eqref{1eq:A.6}, the right-hand side converges, as $D\to+\infty$, to $\frac{\nu\b'^2R}{\mu(N-1)-\b'^2R}$	locally uniformly in $\b'$. Thus, by using \eqref{1eq:3.7.v},
		\begin{equation}
		\label{eq:delta}
		\d\left(\frac{\b'}{\sqrt{D}},D\right)\to\frac{1}{R}k_v\left(\frac{\nu\b'^2R^2}{d\left(\mu(N-1)-\b'^2R\right)}\right) \qquad \text{as $D\to+\infty$},
		\end{equation}
		and the curve $\a_d^-\left(c^*(D),\frac{\b'}{\sqrt{D}}\right)$ converges for $D\to+\infty$ to a curve that lies away from $\{\a=0\}$ and cannot be tangent to $\a_D^+\left(c^*(D),\frac{\b'}{\sqrt{D}}\right)$, giving the desired contradiction.
		
		Next, we establish the order of convergence of $c^*(D)$ to $+\infty$.
		By performing the same change of variable $\b'=\b\sqrt{D}$ as above, we observe that the curves $\a_D^\pm\left(c,\frac{\b'}{\sqrt{D}}\right)$, for fixed $c$, converge to $0$ as $D\to+\infty$, while $\a_d^\pm\left(c,\frac{\b'}{\sqrt{D}}\right)$ converges to a curve lying away from $\{\a=0\}$. Thus, for large $D$, the tangency at $c=c^*(D)$ occurs between $\a_D^+$ and $\a_d^-$. Such a tangency condition can be rewritten as
		\begin{linenomath*}
		\begin{equation*}
			\frac{1}{D}\left(1+\sqrt{1-\frac{4Dg'(0)}{c^{*2}(D)}+\frac{4D\b^{'2}}{c^{*2}(D)}}\right)=\frac{1}{d}\left(1-\sqrt{1-\frac{c_f^2}{c^{*2}(D)}-\frac{4d^2\d^{2}\left(\frac{\b'}{\sqrt{D}},D\right)}{c^{*2}(D)}}\right),
			\end{equation*}
		\end{linenomath*}
			that is,
			\begin{linenomath*}
			\begin{equation*}
			\frac{c^{*2}(D)}{D}\left(1+\sqrt{1-\frac{4Dg'(0)}{c^{*2}(D)}+\frac{4D\b^{'2}}{c^{*2}(D)}}\right)=\frac{4f'(0)+4d\d^{2}\left(\frac{\b'}{\sqrt{D}},D\right)}{1+\sqrt{1-\frac{c_f^2}{c^{*2}(D)}-\frac{4d^2\d^{2}\left(\frac{\b'}{\sqrt{D}},D\right)}{c^{*2}(D)}}}.
		\end{equation*}
	\end{linenomath*}
		Since, thanks to \eqref{eq:delta}, the right-hand side converges to a positive constant as $D\to+\infty$, the same occurs for $c^{*2}(D)/D$, as we wanted.
		
		Passing to (v), Proposition~\ref{1pr:4.1} guarantees that only two situations can occur in a right neighborhood of $D=0$: either $(a)$ $c^*(D)\equiv c_f$, or $(b)$ $c^*(D)>c_f$.
		
		In the former case, trivially, $\lim_{D\to 0}c^*(D)=c_f$, while in the latter we now show that the limit in \eqref{1eq:4.5} holds. Indeed, if $(b)$ occurs, by the construction of $c^*(D)$, $\Si_d(c)$ arises, for $c=c_f$, outside $\mc{S}_D(c)$ and, since $\a_D^+(c,\b)$ converges to $+\infty$ as $D\da 0$, this implies that $\Si_d(c)$ arises, in the first quadrant, below $\a_D^-(c,\b)$. As a consequence,  for $c=c^*(D)$, $\a_d^+(c,\b)$ is tangent to $\a_D^-(c,\b)$. Regarding the latter curve, we have
		\begin{equation}
			\label{1eq:4.4-5}
			\a_D^-(c,\b)=\frac{2(g'(0)-D\b^2)}{c+\sqrt{c^2-c_g^2+4D^2\b^2}}\to\frac{g'(0)}{c}
		\end{equation}
		as $D\to 0$, locally uniformly in $\{\b\geq 0\}$.
		
		In order to conclude, we study the behavior of $\a_d^+(c,\b,D)$ as $D\to 0$. To this end, we observe that
		\begin{equation}
			\label{1eq:4.5bis}
			\ov\b(D)\to \frac{j_\tau}{R} \text{ for } D\to 0,
		\end{equation}
		where $\ov\b$ is the one defined in \eqref{1eq:3.ovb}.
		
		When $N=2,3$, since $\chi_u(\b)$ converges to $0$, as $D\to 0$, locally uniformly  in $(0,j_\tau/R)$, \eqref{1eq:3.8} implies that $\d(\b)$  converges to $0$, and $\a_d^+(c,\b,D)$ converges to the horizontal segment $\a=\frac{c+\sqrt{c^2-c_f^2}}{2d}$ locally uniformly in $(0,j_\tau/R)$.
		For $N=4,5$, instead, \eqref{1eq:4.4bis} and \eqref{1eq:4.5bis} give that $\a_d^+(c,\b,D)$ converges, for $D\to 0$, to the vertical segment
		$\left\{\frac{j_\tau}{R}\right\}\times\left[\frac{c}{2d},\frac{c+\sqrt{c^2-c_f^2}}{2d}\right]$. This, together with the geometric characterization of $c^*(D)$ and \eqref{1eq:4.4-5}, entails that, for all $N\leq 5$, $\lim_{D\to 0}c^*(D)$ is the unique value of $c$, denoted by $c_0$, for which
		\[\frac{c+\sqrt{c^2-c_f^2}}{2d}=\frac{g'(0)}{c},\]
		which is given by the expression in \eqref{1eq:4.5}.
		
		To conclude the proof of Theorem~\ref{1th:3.3}\eqref{1th:3.3.v}, we observe that, if $g'(0)>2f'(0)$, then \eqref{1eq:4.1new} holds for $D\to 0$, since the left-hand side converges to $0$ and the right-hand side is negative; thus, by Proposition~\ref{1pr:4.1}, case $(b)$ occurs, and the first part of the statement is proved. Moreover, direct computations show that $c_0>c_f$. If, instead, $g'(0)<2f'(0)$, then \eqref{1eq:4.1new} fails for $D\sim 0$, thus case $(a)$ occurs. Finally, if $g'(0)=2f'(0)$, then $c_0=c_f$, thus Theorem~\ref{1th:3.3}\eqref{1th:3.3.v} holds independently of whether case $(a)$ or $(b)$ occurs.
	\end{proof}

\begin{remark}
	By comparing with the diagram in Figure~\ref{0fig:1.1} and recalling the expression of $c_a$ given in \eqref{0eq:1.11}, we readily observe that the limit of $c^*(D)$ as $D\to 0$ coincides with the same limit for $c^*_\infty$.
	\end{remark}

We now pass  to study the qualitative properties of the speed of propagation $c^*$ as a function of the radius of the cylinder $\O$.

\begin{proof}[Proof of Theorem~\ref{1th:3.3}\eqref{1th:3.3.vi}]
	We start by showing that, when $c^*>c_f$, $c^*$ is an increasing function of $R$. Thus, we assume that $c^*>c_f$, and we recall that, in such a case, $c^*$ is the first value of $c$ for which the curves $\Si_D(c)$ and $\Si_d(c)$ are tangent. On the one hand, $\Si_D(c)$ does not depend on $R$; on the other one, we will now show that 
	\begin{equation}
		\label{1eq:monotoniaR}
		R\mapsto \a_d^+(c,\b,R) \text{ is decreasing}, \quad R\mapsto \a_d^-(c,\b,R) \text{ is increasing},
	\end{equation}
	which will imply the desired monotonicity.
	In order to prove \eqref{1eq:monotoniaR}, we show that, in its domain of definition, $R\mapsto\d(\b,R)$ is increasing. To this end, observe that \eqref{1eq:3.8}, which implicitly defines  $\d(\b,R)$, can be rewritten by using \eqref{1eq:3.7.u} and \eqref{1eq:3.7.v} as
	\begin{linenomath*}
	\begin{equation*}
		d\d \frac{K_{\tau+1}(\d R)}{K_{\tau}(\d R)}=\frac{\nu D\b}{\mu \frac{J_{\tau}(\b R)}{J_{\tau+1}(\b R)}- D \b }.
	\end{equation*}
\end{linenomath*}
	The right-hand side increases with $R$ thanks to Lemma~\ref{le:A.1}, while \eqref{2eq:A.5} implies that the left-hand side is non-increasing in $R$. This, together with the fact that $\chi_v$ is increasing in $\d$, allows us to conclude \eqref{1eq:monotoniaR}.
	
	\vspace{0.15cm}
	\emph{Convergence as $R\to 0$.}	Consider first the case $N=2,3$. We have that the threshold for enhancement, given by \eqref{1eq:4.1bis}, does not depend on $R$. Thus, if \eqref{1eq:4.1bis} fails, then $c^*(R)\equiv c_f$ for every $R>0$ and the first limit in \eqref{0eq:1.10} trivially follows.
	
	Suppose now that \eqref{1eq:4.1bis} holds. From \eqref{1eq:3.ovb}, we observe that
	\begin{equation}
		\label{1eq:4.10bis}
		\ov\b(R)=\frac{1}{R}k_u\left(\frac{\mu R}{D}\right);
	\end{equation}
	thus, \eqref{1eq:A.7} implies that $\ov\b(R)\to+\infty$ for $R\da 0$. Moreover, by using \eqref{1eq:3.7.v} and the implicit definition of $\d(\b,R)$ given by \eqref{1eq:3.8}, we obtain
	\begin{equation}
		\label{1eq:4.11}
		\d(\b,R)=\frac{\chi_u(\b,R)}{d}\frac{k_v\left(\frac{R\chi_u(\b,R)}{d}\right)}{\frac{R\chi_u(\b,R)}{d}},
	\end{equation}
	where $k_v:(0,+\infty)\to(0,+\infty)$ is the inverse function of $h_v$, and we observe from \eqref{1eq:A.6} that
	\begin{equation}
		\label{1eq:4.12}
		\lim_{R\da 0}\chi_u(\b,R)=0 \quad \text{locally uniformly in $\{\b>0\}$.}
	\end{equation}
	Thus, the first factor of the right-hand side of \eqref{1eq:4.11} converges to 0, while the second one to $k_v'(0)$, which, from \eqref{1eq:3.22} and \eqref{1eq:3.9}, can be shown to be equal to 1 for $N=2$ and equal to $0$ for $N=3$. We conclude that $\lim_{R\da 0}\d(\b,R)$=0 locally uniformly in $\b$ and, therefore, the graphs of the curves $\a_d^{\pm}$ converge, in the same way, to the horizontal lines $\a=\frac{c\pm\sqrt{c^2-c_f^2}}{2d}$. Thus, for every $c>c_f$, there are intersection with $\Si_D(c)$ for $R\sim 0$, and, as a consequence, $c_f<c^*<c$, which concludes the proof in this case.
	
	For $N=4,5$, we recall that $\un\b(R)$ is defined as the unique positive solution smaller than $j_\tau/R$ of
	\begin{linenomath*}
	\begin{equation*}
		R\chi_u(\un\b(R),R)=d(N-3).
	\end{equation*}
\end{linenomath*}
	If $\un\b(R)$ was bounded as $R\da 0$, such a relation, together with \eqref{1eq:4.12}, would give a contradiction. Thus $\lim_{R\da 0}\un\b(R)=+\infty$, which entails that \eqref{1eq:4.1old} holds true for $R\sim 0$. As a consequence, $c^*(R)\equiv c_f$ for small $R$, and the limit trivially follows.
	
	\vspace{0.15cm}	
	\emph{Convergence as $R\to+\infty$.} 
	First of all observe that, if the first condition in the definition of $c^*_\infty$ in \eqref{0eq:1.10} holds, then \eqref{1eq:4.1old} is also satisfied. 
	Thus, in this case, $c^*(R)=c_f=c^*_\infty$ for every $R$, and the same obviously holds true for the limit $R\to+\infty$.
	
	Assume now that the first condition in the definition of $c^*_\infty$ in \eqref{0eq:1.10} fails. Then, since~$\un\b(R)<\ov\b(R)<j_\tau/R$, $\un\b(R)\to 0$ as $R\to+\infty$, and Proposition \ref{1pr:4.1} implies that,  for large $R$, $c^*(R)>c_f$, i.e., $\Si_d(c_f)$ appears outside $\mc{S}_D(c_f)$.
	
	In addition, since $\d(\b,R)$ is defined for $\b\in(\un\b(R),\ov\b(R))$ such a domain shrinks to $0$ as $R\to+\infty$. As a consequence, recalling the monotonicity of $\b\mapsto\a_d^\pm(c,\b)$ given in \eqref{1eq:monotonicityb} and the fact that $\a_d^\pm(c,\un\b)=\frac{c\pm\sqrt{c^2-c_f^2}}{2d}$, the curve~$\Si_d(c,R)$ collapses to the vertical segment 
	\begin{equation}
		\label{1eq:4.12bis}
		\left\{0\right\}\times I_d(c)
	\end{equation}
	where the interval $I_d(c)$ is the one defined in \eqref{0eq:1.11bis}-\eqref{0eq:1.12bis} and used in the construction of~ $c^*_\infty$. This implies that $\lim_{R\to+\infty}c^*(R)\geq c_g$, otherwise no intersection could occur for large $R$ between $\Si_d$ and $\Si_D$. For this reason, we consider $c\geq c_g$.	
	
	Finally, we observe that
	\begin{linenomath*}
		\begin{equation*}
	\mc S_D(c)\cap\{\b=0\}=\left\{0\right\}\times I_D(c)
	\end{equation*}
\end{linenomath*}
	(see again \eqref{0eq:1.12bis}). This shows that $\lim_{R\to+\infty}c^*(R)$ can be characterized by  repeating exactly the construction of $c^*_\infty$ performed in Proposition~\ref{0pr:1.2bis} and, therefore, that the two quantities coincide in all the cases.	
\end{proof}

In the last part of this section, we consider problem \eqref{1eq:1.1} in the specific case $N=2$. Up to a translation in the $y$ direction (recall that our notation in this setting implies $y=x_2$) and by using the symmetry of the domain and of the equations, it reduces to
\begin{equation}
\label{1eq:5.1}
\left\{
\begin{array}{lllll}
\p_t u-D\D u=g(u) & & x\in\R, & y\in(-R,0), & t>0 \\
\p_t v-d\D v=f(v) &  &x\in\R, & y>0, & t>0 \\
D\, \p_y u= \nu v-\mu u & & x\in\R, & y=0, & t>0 \\
-d\, \p_y v= \mu u-\nu v & & x\in\R, & y=0, & t>0 \\
\p_y u=0 & & x\in\R, & y=-R, & t>0.
\end{array}\right.
\end{equation}
Our main purpose is to prove Theorem~\ref{1th:5.1}, which establishes the relationship between the limit of \eqref{1eq:5.1} as $R\to 0$ and the road-field problem \eqref{1eq:BRR2}, through a singular rescaling of the exchange parameter $\mu$.

\begin{proof}[Proof of Theorem~\ref{1th:5.1}.]
	Observe that problem \eqref{1eq:5.2} is obtained by replacing $\mu$ by $\tilde\mu(R)$ in \eqref{1eq:5.1}. By reasoning as in Section~\ref{1subsection:c^*_2} for $N=2$ (see \eqref{1eq:3.5}--\eqref{1eq:3.7.v} and \eqref{2eq:A.5}), we therefore have  that $\tilde c^*(R)$ can be geometrically characterized as the smallest value of $c$ for which the regions delimited by the curves
	\begin{equation}
		\label{1eq:5.4}
		\left\{\begin{array}{l}
			c\a-d\a^2-d\d^2=f'(0), \\
			c\a-D\a^2+D\b^2=g'(0), 
		\end{array}\right.
	\end{equation}
	where $\b$ and $\d$ are related through
	\begin{equation}
		\label{1eq:5.4bis}
		d\d=\displaystyle{\frac{\nu D h_u(\b R)}{\tilde\mu( R)R-Dh_u(\b R)}\;
		\left(=\frac{\nu D \b R \tan(\b R)}{\tilde\mu( R)R-D\b R\tan(\b R)}\right)},
	\end{equation}
	intersect. 
		
	If $\frac{D}{d}\leq 2-\frac{g'(0)}{f'(0)}$ (observe that this threshold does not depend neither on $\mu$, nor $\tilde\mu(R)$, nor $R$), we have, from Proposition~\ref{1pr:4.1}\eqref{1pr:4.1i} and \eqref{1eq:5.1bis}, 
	that both $\tilde c^*(R)$ and $\cBRR$ coincide with $c_f$, thus the result trivially holds.
	
	On the other hand, for $\frac{D}{d}> 2-\frac{g'(0)}{f'(0)}$, we briefly recall that in \cite[Section 4]{BRR2}, following the same strategy of this work, $\cBRR$ is characterized as the smallest $c$ for which the system
	\begin{equation}
		\label{1eq:5.3}
		\left\{\begin{array}{l}
			c\a-d\a^2-d\d^2=f'(0) \\
			c\a-D\a^2=g'(0)-\frac{\mu d \d}{\nu+d\d},
		\end{array}\right. 
	\end{equation}
	possesses real solutions, which, if they exist, satisfy $\d>0$.
	
	By comparing \eqref{1eq:5.4} with \eqref{1eq:5.3}, in order to conclude, it is sufficient to prove that, if we use \eqref{1eq:5.4bis} to express $\b$ as a function of $\d$ (and of $R$), we have, locally uniformly in $\{\d>0\}$,
	\begin{equation}
		\label{1eq:5.5}
		D\b^2(\d,R)\to\frac{\mu d \d}{\nu+d\d} \quad \text{ as $R\da 0$.}
	\end{equation} 
	To this end, we explicitly compute $\b$ from \eqref{1eq:5.4bis} getting
	\begin{linenomath*}
	\begin{equation*}
		\b(\d,R)=\frac{1}{R}k_u\left(\frac{\tilde\mu(R)R}{D}\frac{d\d}{\nu+d\d}\right),
	\end{equation*}
\end{linenomath*}
	and, by taking the limit as $R\to 0$, using \eqref{1eq:5.2bis} and \eqref{1eq:A.7}, we obtain \eqref{1eq:5.5}.
\end{proof}

The last result of this section complements the ones of Theorem~\ref{1th:3.3}\eqref{1th:3.3.vi} and Theorem~\ref{1th:5.1}, since we determine the behavior of the speed of propagation $\tilde c^*(R)$ of problem~\eqref{1eq:5.2} in the~$x$ direction  for all the possible behaviors of the positive function $\tilde{\mu}(R)$.

\begin{proposition}
	\label{1pr:5.2}
	The  asymptotic speed of propagation $\tilde c^*(R)$ of problem~\eqref{1eq:5.2} in the~$x$ direction  satisfies:
	\begin{enumerate}[(i)]
		\item \label{1pr:5.2i} if \eqref{1eq:4.1bis} does not hold, then $\lim_{R\da 0}\tilde c^*(R)=c_f$;
		\item \label{1pr:5.2ii} if \eqref{1eq:4.1bis} holds, then
		\begin{linenomath*}
		\begin{equation*}
		\lim_{R\da 0}\tilde c^*(R)=\begin{cases}
		c_f & \text{if $\lim_{R\da0}\frac{\tilde{\mu}(R)}{R}=+\infty$,} \\
		c^*_\infty>c_f & \text{if $\lim_{R\da0}\frac{\tilde{\mu}(R)}{R}=0$}.
		\end{cases}
		\end{equation*}
	\end{linenomath*}

	\end{enumerate}
	\begin{proof}
		(i) As already remarked, since the threshold given in \eqref{1eq:4.1bis} does not depend on $\tilde\mu$ nor on $R$, we have $\tilde c^*(R)= c_f$ for all $R$, and the conclusion is immediate.
		
		(ii) Under the assumption of this part, the curve $\Si_d(c)$ arises, for $c=c_f$, outside $\mathcal{S}_D(c)$. When $\lim_{R\da0}\frac{\tilde{\mu}(R)}{R}=+\infty$ we can reason as in the proof of Theorem~\ref{1th:3.3}\eqref{1th:3.3.vi} to show that $\ov\b(R)\to+\infty$, where now $\ov\b(R)$ is defined by \eqref{1eq:4.10bis} with $\mu$ replaced by $\tilde\mu(R)$. Namely, if $\tilde{\mu}(R)R$ is bounded away from $0$, the same holds true for $k_u\left(\frac{\tilde{\mu}(R)R}{D}\right)$, and the conclusion follows directly from \eqref{1eq:4.10bis}. Otherwise, if $\tilde{\mu}(R)R$ approaches $0$ for some sequence of $R$'s converging to $0$, \eqref{1eq:A.7} gives $\ov\b^2(R)\sim\frac{1}{D}\frac{\tilde{\mu}(R)}{R}$ along such a sequence, and $\ov\b(R)\to+\infty$ also in this case.
		
		Moreover, since $\lim_{R\da0}\frac{\tilde{\mu}(R)}{R}=+\infty$,  \eqref{1eq:5.4bis} and \eqref{1eq:A.6} imply that $\d(\b,R)\to 0$ locally uniformly for $\{\b\geq 0\}$. Thus, $\Si_d(c)$ converges locally uniformly to two horizontal lines, and it intersects $\mathcal{S}_D(c)$ for every $c>c_f$, which gives the desired result.
		
		Finally, when $\lim_{R\da0}\frac{\tilde{\mu}(R)}{R}=0$, by \eqref{1eq:4.10bis} and \eqref{1eq:A.7}, we have that $\ov\b(R)\to 0$ as $R\da 0$, thus, as in the proof of Theorem~\ref{1th:3.3}\eqref{1th:3.3.vi}, $\Si_d(c)$ converges to the vertical segment \eqref{1eq:4.12bis}. Thanks to the assumption of part \eqref{1pr:5.2ii},  such a segment arises, for $c=c_f$, outside $\mathcal{S}_D(c)$, and we can reason as in the last part of the proof of  Theorem~\ref{1th:3.3}\eqref{1th:3.3.vi} to show that $\lim_{R\da0}\tilde c^*(R)=c^*_\infty$.
	\end{proof}
\end{proposition}

\setcounter{equation}{0}
\section{The case of mortality for $v$}
\label{section4}
In this last section we consider the system with the mortality term \eqref{2eq:4.0} and the corresponding one set in two adjacent half-spaces. As in the first part of the paper, we start with the latter case, which is easier.

Since many constructions and arguments are similar to those of the previous sections, we will only detail the main differences.

\subsection{Two half-spaces with mortality}
\label{section4.1}
When $f(s)=-\r s$, system \eqref{0eq:1.1} reads
\begin{equation}
\label{0eq:2.1}
\left\{
\begin{array}{lll}
\p_t u-D\D u=g(u) & x\in\R^{N-1}\times\R^-, & t>0 \\
\p_t v-d\D v=-\r v & x\in\R^{N-1}\times\R^+, & t>0 \\
D\, \p_{x_N} u = \nu v-\mu u & x\in\R^{N-1}\times\{0\}, & t>0\\
-d\, \p_{x_N} v = \mu u-\nu v & x\in\R^{N-1}\times\{0\}, & t>0,
\end{array}\right.
\end{equation}
and the long-time behavior for \eqref{0eq:2.1} follows the same patterns as system \eqref{0eq:1.1}, as shown in the following result.

\begin{theorem}
	\label{0th:2.1}
	Assume \eqref{eq:gKPP}. Then, problem \eqref{0eq:2.1} admits a unique bounded, positive steady state $(U,V)$, which only depends on $x_N$, and  for every non-trivial $(u_0,v_0)$, the solution of \eqref{0eq:2.1} starting from $(u_0,v_0)$ satisfies
	\begin{linenomath*}
	\begin{equation*}
	\lim_{t\to+\infty}(u,v)=(U,V)
	\end{equation*}
\end{linenomath*}
	locally uniformly in space, in the closure of the corresponding domains.
\end{theorem}

\begin{proof}
	The proof follows most of the lines developed in Section~\ref{0section1}. Indeed, the same super- and subsolutions allow us to prove a result like the one given in Proposition~\ref{0pr:1.1}, while for the uniqueness result of the bounded, positive steady states with the desired symmetries, it is easy to see that, in this case, the unique possibility for $V$ is $V(\tilde x_{N})=V(0)e^{-\sqrt{\r/d}\,\tilde x_{N}}$. Thus, reasoning as in the proof of Proposition~\ref{0pr:1.2}, we obtain that $U$ is decreasing and lies in $(0,1)$ for all $x_N<0$. With these elements, we have that the second relation in \eqref{0eq:1.6bis} has to be replaced by
	\[V'(0)^2=\frac{\r}{d}V(0)^2,\]
	and, by studying, as in Proposition~\ref{0pr:1.2}, the monotonicities of $V'(0)^2$ and $V(0)^2$ as $U(0)$ varies in $(0,1)$, we conclude that such a relation has at most one solution.
\end{proof}

Turning to the asymptotic speed of propagation, we have the following analogue of Theorem~\ref{0th:1.4}.

\begin{theorem}
	\label{0th:2.4} Assume \eqref{eq:gKPP} and let $c^*_{m,\infty}$ be as in \eqref{1eq:1.13}. Then, for any $i\in\{1,\ldots,N-1\}$ and every compactly supported initial datum $(u_0,v_0)\not\equiv(0,0)$, the solution $(u,v)$ of \eqref{0eq:2.1} satisfies:
	\begin{enumerate}[(i)]
		\item \label{0th:2.4.i} for all $c>c^*_{m,\infty}$, \begin{linenomath*}
			\begin{equation*}
				\lim_{t\to\infty}\sup_{\substack{|x_i|\geq ct \\ x_N\leq 0, \; \tilde x_N\geq 0}}\left(u(x_1,\dots,x_N,t),v(x_1,\dots,\tilde x_N,t)\right)=(0,0),
				\end{equation*}
				\end{linenomath*}
		uniformly in $x_j$, for all $j\in\{1,\ldots,N-1\}\setminus\{i\}$,
		\item \label{0th:2.4.ii} for all $c<c^*_{m,\infty}$ and $a>0$, 
		\begin{linenomath*}
			\begin{equation*}
		\lim_{t\to\infty}\sup_{\substack{|x_i|\leq ct \\ -a\leq x_N\leq 0, \; 0\leq \tilde x_N\leq a}}\left(|u(x_1,\dots,x_N,t)-U(x_N)|+|v(x_1,\dots,\tilde x_N,t)-V(\tilde x_N) |\right)=0,
						\end{equation*}
	\end{linenomath*}
		locally uniformly in $x_j$, for all $j\in\{1,\ldots,N-1\}\setminus\{i\}$, where $(U,V)$ is the unique positive, bounded steady state of problem \eqref{0eq:2.1}.
	\end{enumerate}
\end{theorem}

\begin{proof}
	The proof relies on the construction of super- and subsolutions moving with appropriate speeds, along much the same lines of Section~\ref{0section1}. We recall that the arguments there are developed for the linearized problem around $(0,0)$, and thus they hold also in the present case.
	
	For the supersolutions, we now look for intersections between the intervals
	\begin{equation}
	\label{2eq:1.11bis}
	I_D(c):=\left\{r_D^-(c)\leq \a\leq r_D^+(c)\right\} \qquad \text{ and } \qquad I_{d,m}(c):=\left\{0< \a\leq \tilde r_d^+(c)\right\},
	\end{equation}
	where $r_D^\pm(c)$ are as in \eqref{0eq:1.12bis} and $\tilde r_d^+(c):=\frac{c+\sqrt{c^2+4d\r}}{2d}$. The smallest value of $c$ for which intersection occurs gives the quantity $c^*_{m,\infty}$. Observe that $I_{d,m}$ is strictly larger than $I_d$ from \eqref{0eq:1.11bis}, which shows that $c^*_{\infty}\geq c^*_{m,\infty}$, with strict inequality if $c^*_{\infty}>c_g$.

	The construction of compactly supported subsolutions only requires some minor modifications when $\frac{d}{D}>2+\frac{\r}{g'(0)}$. The main difference with respect to the proof of Proposition~\ref{0pr:1.5} is that $f'(0)$ has to be replaced by $-\r$ everywhere, which entails that the analogue of the quantity $\un c$ defined in \eqref{0eq:1.19bis} is now equal to $0$ and $\breve\b(c)>0$ for all $c>0$. As a consequence, the curve $\Si_d(c,L)$ converges, as $L\to+\infty$, to
	\begin{linenomath*}
	\begin{equation*}
	\Si_d(c,+\infty):=\left\{\left(\b,\tilde r_d^+(c)\right):\b<0\right\}\bigcup\left\{(0,\a):0<\a\leq \tilde r_d^+(c)\right\},
	\end{equation*}
\end{linenomath*}
	and one can repeat the same arguments as in Proposition~\ref{0pr:1.5}.
\end{proof}

By reasoning as in Section~\ref{0section1}, it is possible to show that the analogue of the qualitative properties of Proposition~\ref{0pr:1.6}\eqref{0pr:1.6.i}--\eqref{0pr:1.6.iii} hold for $c^*_{m,\infty}$, seen as a function of $d$. These proofs follow, \emph{mutatis mutandis}, as those of Proposition~\ref{0pr:1.6}, thus we do not give the details.

\subsection{Cylinder with external mortality}
\label{2section4}
We conclude by studying problem \eqref{2eq:4.0}; thus, as in Section~\ref{section3}, we use the notation $x:=x_1$ and $y:=(x_2,\ldots,x_N)$. As we did in the previous sections, we start by analyzing the long-time behavior. 
The first result provides the counterpart of the upper bound given in Propositions~\ref{0pr:1.1} and~\ref{1pr:2.1}, and its proof follows exactly along the same lines.
\begin{proposition}
	\label{2pr:4.1}
	Assume \eqref{eq:gKPP}, and let $(u,v)$ be the solution of \eqref{2eq:4.0} starting with an initial datum $(u_0,v_0)\not\equiv(0,0)$. Then, there exists a non-negative, bounded steady state $(U_2,V_2)$ which is $x$-independent, $y$-radial, and such that
	\[\limsup_{t\to+\infty}(u,v)\leq(U_2,V_2)\]
	locally uniformly in space, in the closure of the corresponding domains.
\end{proposition}

Regarding the lower bound, things go differently here. Indeed, it may happen that the unique steady state of \eqref{2eq:4.0} with the symmetries specified in Proposition~\ref{2pr:4.1} is $(0,0)$. In Proposition~\ref{2pr:new} below, we establish a necessary and sufficient condition for the existence of \emph{positive} steady states with the considered symmetries, which will allow us to obtain, in Theorem~\ref{2th:4.6}, also a lower bound for the solutions of the parabolic problem \eqref{2eq:4.0}.

Before that, we begin with a preliminary discussion, that will lead us to an equivalent problem: the $x$-independent, $y$-radial, positive, bounded, stationary solutions of \eqref{2eq:4.0} are of the form $(U(y),V(\tilde y))=(\Phi(|y|),\Psi(|\tilde y|))$, with such profile functions satisfying
\begin{equation}
\label{2eq:stationaryD}
\left\{
\begin{array}{ll}
-D \Phi''-D\frac{N-2}{r}\Phi'=g(\Phi) & r\in(0,R) \\
-d \Psi''-d\frac{N-2}{r}\Psi'=-\r\Psi & r>R \\
D\, \Phi'(R) = \nu \Psi(R)-\mu \Phi(R) & \\
-d\, \Psi'(R) = \mu \Phi(R)-\nu \Psi(R) & \\
\Phi'(0)=0.
\end{array}\right.
\end{equation}
Observe that the function $\Psi$ is uniquely determined by $\Phi'(R)$. Namely, since positive, bounded solutions of the second equation in \eqref{2eq:stationaryD} have the form 
\begin{equation}
\label{2eq:4.1bis}
\Psi(r)=\g r^{-\tau} K_{\tau}\left(\sqrt{\frac{\r}{d}}\,r\right),
\end{equation} with $\g>0$ and $\tau=\frac{N-3}{2}$ (see \eqref{2eq:A.1}), the exchange conditions (third and fourth equations in~\eqref{2eq:stationaryD}), together with \eqref{2eq:A.4},  then give
\begin{equation}
\label{2eq:4.2}
\left\{
\begin{array}{l}
D\Phi'(R)=d\Psi'(R)=-\g R^{-\tau}\sqrt{d\r}K_{\tau+1}\left(\sqrt{\frac{\r}{d}}\,R\right) \\
\mu \Phi(R)=\nu \Psi(R)-d \Psi'(R)=\g R^{-\tau} \left(\nu K_{\tau}\left(\sqrt{\frac{\r}{d}}\,R\right)+\sqrt{d\r}K_{\tau+1}\left(\sqrt{\frac{\r}{d}}\,R\right)\right),
\end{array}\right.
\end{equation}
whence, from the first relation,
\begin{equation}
\label{2eq:4.3}
\g=\frac{-DR^\tau\Phi'(R)}{\sqrt{d\r}K_{\tau+1}\left(\sqrt{\frac{\r}{d}}R\right)}.
\end{equation}
Moreover, taking the ratio between the equations in \eqref{2eq:4.2} yields to
\begin{equation}
\label{2eq:4.2bis}
D \Phi'(R)+\k \Phi(R)=0, \qquad  \k:=\frac{\mu\sqrt{d\r}K_{\tau+1}\left(\sqrt{\frac{\r}{d}}\,R\right)}{\nu K_{\tau}\left(\sqrt{\frac{\r}{d}}\,R\right)+\sqrt{d\r}K_{\tau+1}\left(\sqrt{\frac{\r}{d}}\,R\right)}>0.
\end{equation} 
Thus, the problem of looking for the above-mentioned steady states reduces to finding a positive solution $U(y)=\Phi(|y|)$ of the Robin problem 
\begin{equation}
\label{2eq:stationaryUD}
\left\{
\begin{array}{ll}
-D\, \D_y U=g(U), & y\in B_{N-1}(R) \\
D\p_n U+\k U=0  & y\in\p B_{N-1}(R),
\end{array}\right.
\end{equation}
where $B_{N-1}(R)$ denotes the ball of radius $R$ centered at the origin of $\R^{N-1}$, and then taking $V(\tilde y)=\Psi(|\tilde y|)$ given by \eqref{2eq:4.1bis} and \eqref{2eq:4.3}. Observe that the solutions of \eqref{2eq:stationaryUD} satisfy $\Phi'(R)<0$ by \eqref{2eq:4.2bis} and, therefore, $\g>0$, that is, $V(\tilde y)=\Psi(|\tilde y|)$ is positive too.

It is a well-known fact that the existence of positive solutions for \eqref{2eq:stationaryUD} depends on the linear stability of $0$. Namely, calling $\b_0^2$ the principal eigenvalue of the Robin problem
	\begin{equation}
\label{2eq:linearD}
\left\{
\begin{array}{ll}
-\D_y\phi=\b_0^2\,\phi, & y\in B_{N-1}(R) \\
D\p_n \phi+\k \phi=0  & y\in\p B_{N-1}(R),
\end{array}\right.
\end{equation}
the following result holds.

\begin{proposition}
	\label{2pr:new}
	Assume \eqref{eq:gKPP} and \eqref{eq:gstrongKPP}. Then, there exists a positive solution of \eqref{2eq:stationaryUD} if and only if \eqref{2eq:condnecsuff} holds. Moreover, if this is the case, such a solution is unique.
\end{proposition}
\begin{proof}
	We observe that, thanks to the strong maximum principle and the Hopf boundary lemma, any solution $U$ of \eqref{2eq:stationaryUD} which is positive inside $B_{N-1}(R)$ satisfies $\inf U>0$ and~$\sup U<1$.
	
	As for the existence, the constant $1$ is a supersolution of \eqref{2eq:stationaryUD}. In addition, if condition \eqref{2eq:condnecsuff} holds, by taking a positive eigenfunction $\phi$ of \eqref{2eq:linearD} and a number $\t>0$, $\t\sim 0$ so that~$D\b_0^2<g'(0)-\t$, we have that $\e\phi$ is a positive subsolution of \eqref{2eq:stationaryUD} for $\e>0$, $\e\sim 0$, and it lies below the supersolution. Thus, there exists a solution of \eqref{2eq:stationaryUD} lying between the supersolution and the subsolution that we have constructed, i.e., it is positive.
	
	Uniqueness follows with the same arguments as in the the proof of Proposition~\ref{1pr:2.2}.
	
	Finally, assume that \eqref{2eq:condnecsuff} fails and that there exists a nonnegative solution $U$ of \eqref{2eq:stationaryUD}. Thanks to this assumption and \eqref{eq:gKPP}, $\e\phi$, where $\phi$ is a positive eigenfunction of \eqref{2eq:linearD}, is a supersolution of \eqref{2eq:stationaryUD} for every $\e>0$. The maximum principle then gives $U\leq \e\phi$, entailing that $U\equiv 0$.
\end{proof}

As an immediate consequence of Propositions~\ref{2pr:4.1} and~\ref{2pr:new}, we obtain the extinction result of Theorem~\ref{2th:section4.1}. Conversely, we now prove the invasion result when \eqref{2eq:condnecsuff} holds.

\begin{theorem}
	\label{2th:4.6}
	Assume \eqref{eq:gKPP}, \eqref{eq:gstrongKPP} and \eqref{2eq:condnecsuff}. Then, the solution of \eqref{2eq:4.0} starting from a non-trivial initial datum satisfies
	\begin{linenomath*}
	\begin{equation*}
	\lim_{t\to+\infty}(u,v)=(U,V)  
	\end{equation*}
\end{linenomath*}
	locally uniformly in space, in the closure of the corresponding domains, where $(U,V)$ is the unique positive, bounded steady state of \eqref{2eq:4.0}.
\end{theorem}
\begin{proof}
	Thanks to Proposition~\ref{2pr:new}, condition \eqref{2eq:condnecsuff} is equivalent to the well-posedness for \eqref{2eq:stationaryUD} and, thus, as we have seen, to the existence and uniqueness of $x$-independent, $y$-radial, positive, bounded steady states of \eqref{2eq:4.0}. We want to modify such a steady state in order to get a generalized subsolution for \eqref{2eq:4.0} with compact support. The proof then will follow as the one of Proposition~\ref{0pr:1.1}, together with the result of Proposition~\ref{2pr:4.1}.
	
	The first step is to perturb problem \eqref{2eq:linearD}, complemented with a normalization condition, keeping in mind that the solution depends smoothly on the coefficients. We replace $\r$ with $\r+2d\a^2$, where $\a$ is a small positive number that we fix later and such that condition \eqref{2eq:condnecsuff} still holds true after this change.	Define
	\begin{linenomath*}
	\begin{align*}
	u_*(x,y)&:=\left\{\begin{array}{ll}
	\cos(\a x)\Phi_1\left(|y|\right)&\text{if $x\in\left(-\frac{\pi}{2\a},\frac{\pi}{2\a}\right)$, $|y|\in\left(0,R\right)$} \notag \\
	0 &\text{otherwise,} 
	\end{array}\right. \\ 	
	v_*(x,\tilde y)&:=\left\{\begin{array}{ll}
	\cos(\a x)\Psi_1(|\tilde y|)&\text{if $x\in\left(-\frac{\pi}{2\a},\frac{\pi}{2\a}\right)$, $|\tilde y|>R$} \\
	0 &\text{otherwise,} 
	\end{array}\right.
	\end{align*}
\end{linenomath*}
	where $\Phi_1(|y|)$ is the unique normalized
	positive eigenfunction of problem \eqref{2eq:linearD} with $\k$ as in \eqref{2eq:4.2bis} and $\Psi_1$ is given by \eqref{2eq:4.1bis} and \eqref{2eq:4.3} with $\Phi$ replaced by $\Phi_1$ and, everywhere, $\r$ replaced by $\r+2d\a^2$. As seen in the discussion at the beginning of this subsection, $(u_*,v_*)$ is a subsolution of
	the linearization of \eqref{2eq:4.0} around $(u,v)=(0,0)$, with $\r$ replaced by~$\r+2d\a^2$, for $\a$ small enough.	Therefore, by usual computations,
	$\e(u_*,v_*)$ is a generalized subsolution of the semilinear problem
	provided $\a,\e>0$ are small enough.
	
	However, $v_*$ is not compactly supported. For this reason, we define
	\begin{linenomath*}
	\begin{equation*}	
	\tilde v_*(x,\tilde y):=\left\{\begin{array}{ll}
	\cos(\a x)\tilde\Psi_1(|\tilde y|)&\text{if $x\in\left(-\frac{\pi}{2\a},\frac{\pi}{2\a}\right)$, $|\tilde y|\in(R,L)$} \\
	0 &\text{otherwise,} 
	\end{array}\right.
	\end{equation*}
\end{linenomath*}
where $L$ will be defined later and $\tilde\Psi_1(r)$ is the solution of the initial value problem
\begin{linenomath*}
\begin{equation*}	
\left\{\begin{array}{l}
-d\tilde\Psi_1''-d\frac{N-2}{r}\tilde\Psi_1'=-\left(\r+d\a^2\right)\tilde\Psi_1 \qquad r>R \\
\tilde\Psi_1(R)=\Psi_1(R)  \\
\tilde\Psi_1'(R)=\Psi_1'(R).
\end{array}\right.
\end{equation*}
\end{linenomath*}
The pair $\e(u_*,\tilde v_*)$ is still a stationary, generalized subsolution of \eqref{2eq:4.0} for small $\a,\e$. Moreover, the function $w:=\tilde\Psi_1-\Psi_1$ satisfies $w(0)=w'(0)=0$ and
\[
-dw''-d\frac{N-2}{r}w'+\left(\r+d\a^2\right)w>0 \qquad \text{ $r\in(R,+\infty)$},
\]
and thus is negative in a right neighborhood of $R$, and it cannot have a global (negative) minimum on $(R,+\infty)$, owing to the elliptic weak maximum principle. Since $\lim_{r\to+\infty}\Psi_1(r)=0$, this implies that $\liminf_{r\to+\infty}\tilde\Psi_1(r)<0$, and we can take $L>R$ as the first point where~$\tilde\Psi_1$ vanishes.
\end{proof}

Before passing to the study of the asymptotic speed of propagation for problem \eqref{2eq:4.0}, we give an useful characterization of the principal eigenvalue of problem \eqref{2eq:linearD} in terms of the notation introduced in Section~\ref{1subsection:c^*_2}.

\begin{proposition}
	\label{2pr:4.1bis}
	The principal eigenvalue of problem \eqref{2eq:linearD} satisfies
	\begin{equation}
		\label{2eq:4.8}
		0<\b_0<\ov\b<\frac{j_\tau}{R},
		\end{equation}
	where $\ov\b$ is the quantity defined in \eqref{1eq:3.ovb}. Moreover, $\b_0$ is the unique solution satisfying \eqref{2eq:4.8} of the following relation:
	\begin{equation}
		\label{2eq:4.8bis}
		\chi_u\left(\b_0\right)=\chi_v\left(\sqrt{\frac{\r}{d}}\right),
		\end{equation}
	where $\chi_u$ and $\chi_v$ are the functions defined in \eqref{1eq:3.7.u} and \eqref{1eq:3.7.v}. As a consequence,
	\begin{equation}
		\label{2eq:4.8bisbis}
		\b_0>\un\b,
		\end{equation}
	where $\un\b$ is the quantity introduced in \eqref{1eq:3.9bis} and \eqref{1eq:3.9bisbis}. 
	\begin{proof}
		The unique solutions of the differential equation in \eqref{2eq:linearD} which are defined in the whole $B_{N-1}(R)$ are $\phi(y)=\g|y|^{-\tau}J_{\tau}\left(\b_0|y|\right)$, with $\g>0$. In order $\phi$ to be positive, $\b_0 R$ has to lie in $(0,j_\tau)$. In addition, by \eqref{2eq:4.2bis} the boundary relation in \eqref{2eq:linearD} is equivalent to
		\begin{equation}
			\label{2eq:4.8ter}
			h_u\left(\b_0 R\right)=\frac{\mu R}{D}\frac{\sqrt{d\r}K_{\tau+1}\left(\sqrt{\frac{\r}{d}}\,R\right)}{\n K_{\tau}\left(\sqrt{\frac{\r}{d}}\,R\right)+\sqrt{d\r}K_{\tau+1}\left(\sqrt{\frac{\r}{d}}\,R\right)}=\frac{\k R}{D},
			\end{equation}
		while, from \eqref{1eq:3.ovb}, $h_u\left(\ov\b R\right)$ equals $\frac{\mu R}{D}$, which is greater than the right-hand side of \eqref{2eq:4.8ter}. As a consequence, since $h_u(\b R)$ is increasing in $\b$, we have that $\b_0<\ov\b$. Finally, \eqref{2eq:4.8bis} follows from \eqref{2eq:4.8ter} with easy computations. 
		\end{proof}
	\end{proposition}

The previous results allow us to determine the long-time behavior of the solution of \eqref{2eq:4.0} for large $N$, the other parameters of the problem being fixed, providing us with the proof of Theorem~\ref{2th:section4.2}\eqref{2th:section4.2.i}.

	\begin{proof}[Proof of Theorem~\ref{2th:section4.2}\eqref{2th:section4.2.i}]
	From \eqref{1eq:4.4quater} and \eqref{2eq:4.8bisbis}, we have that $\b_0(N)\to+\infty$ as $N\to+\infty$. Thus, the conclusion follows by applying the extinction result in Theorem~\ref{2th:section4.1}.
\end{proof}

In addition, the characterization of $\b_0$ given in Proposition~\ref{2pr:4.1bis} allows us to establish the following properties of $\b_0$ as a function of $R$.

\begin{proposition}
	\label{2pr:4.1ter}
The function $R\mapsto\b_0(R)$ is continuous and decreasing.

\begin{proof}
	The continuity (actually differentiability) of this function follows directly from \eqref{2eq:4.8bis}, since all the quantities appearing there are differentiable in $R$. Moreover, by explicitly indicating the dependence on $R$ in \eqref{2eq:4.8bis}, we obtain $\chi_u\left(\b_0(R),R\right)=\chi_v\left(\sqrt{\frac{\r}{d}},R\right)$, and  differentiating this relation with respect to $R$ gives
	\begin{linenomath*}
		\begin{equation*}
	\p_\b\chi_u\left(\b_0(R),R\right)\b_0'(R)+\p_R\chi_u\left(\b_0(R),R\right)=\p_R\chi_v\left(\sqrt{\frac{\r}{d}},R\right).
	\end{equation*}
\end{linenomath*}
	Thus, we obtain $\b_0'(R)<0$, since, as shown in Sections~\ref{1subsection:c^*_2} and~\ref{1section4}, $\p_\b\chi_u>0$, $\p_R\chi_u>0$ and $\p_R\chi_v\leq 0$.
	\end{proof}	
	\end{proposition}

When  invasion occurs, we want to study the speed of propagation $c^*_m$ at which it takes place. For this reason, we assume without further reference  \eqref{2eq:condnecsuff}, which, thanks to \eqref{2eq:4.8}, \eqref{2eq:4.8bis} and the monotonicity of $\chi_u(\b)$, can be equivalently written as
\begin{equation}
	\label{2eq:5.1}
		\sqrt{\frac{g'(0)}{D}}\geq\bar{\b} \quad \text{ or } \quad \begin{cases}
		\sqrt{\frac{g'(0)}{D}}<\bar{\b} \\
		\chi_v\left(\sqrt{\frac{\r}{d}}\right)<\chi_u\left(\sqrt{\frac{g'(0)}{D}}\right).
	\end{cases}
\end{equation}

The construction of $c^*_m$ follows most of the lines as the one of the speed of propagation for problem \eqref{1eq:1.1}, thus, we will repeatedly use the notation of Section~\ref{1subsection:c^*_2} and we will not give all the details of the construction, but will simply point out only the main differences. To construct supersolutions of \eqref{2eq:4.0}, exactly as in Section~\ref{1subsection:c^*_2}, we look for (super)solutions of the linearized problem around $(0,0)$ of the form \eqref{1eq:3.2},
where $\d$ is implicitly given in terms of $\b$ though relation \eqref{1eq:3.8}.

This leads to look for intersections in the first quadrant of the $(\b,\a)$ plane between the curves $\Si_D$ defined in \eqref{1eq:3.13} and $\tilde{\Si}_d(c)=\{(\b,\tilde{\a}^\pm_d(c,\b))\}$, where 
\begin{linenomath*}
\begin{equation*}
\tilde{\a}^\pm_d(c,\b):=\frac{c\pm\sqrt{c^2+4d\r-4d^2\d^2(\b)}}{2d}
\end{equation*}
\end{linenomath*}
and $\b$ ranges in the domain of definition of $\tilde{\a}^\pm_d(c,\b)$.
These functions are similar to those in \eqref{1eq:3.12}, except for the fact that their graphs are non-empty for all $c>0$, independently of the other parameters. Indeed, recalling \eqref{1eq:3.8} and \eqref{2eq:4.8bis}, we have $\d(\b_0)=\sqrt{\frac{\r}{d}}$; thus, for every $c>0$, there exists $\breve{\b}(c)\in(\b_0,\ov\b)$ such that $\tilde{\a}^\pm_d(c,\b)$ is defined for $\b\in[\un\b,\breve{\b}(c)]$. 
The function $c\mapsto\breve{\b}(c)$ is continuous, increasing and satisfies
\begin{linenomath*}
\begin{equation*}
\lim_{c\da 0}\breve{\b}(c)=\b_0, \qquad \lim_{c\ua \infty}\breve{\b}(c)=\ov\b.
\end{equation*}
\end{linenomath*}
In analogy with the notation introduced in Section~\ref{1subsection:c^*_2}, we the denote by $\tilde{\mc{S}}_d(c)$ the region in the first quadrant of the $(\b,\a)$ plane lying between the curves $\tilde\a_d^\pm$.

Another major difference is that the construction performed in Proposition~\ref{1pr:3.1} of subsolutions for $c<c_f$ supported in $\R^N\setminus\ov\O$ cannot be repeated here, due to the fact that we are no longer considering a Fisher-KPP nonlinearity in such a domain.

Nonetheless, we will now show that, analogously to what occurs in \cite{T16} and \cite{BRR4} for the road-field problems considered there, the regions $\mathcal{S}_D(c)$ and $\tilde{\mathcal{S}}_d(c)$ lie at positive distance for $c\sim 0$. Indeed,
$\Si_D(0)$ is defined for $\b\geq\sqrt{\frac{g'(0)}{D}}$, while $\tilde\Si_d(0)$ is defined for for $\b\leq\b_0$, and $\b_0<\sqrt{\frac{g'(0)}{D}}$, owing to \eqref{2eq:condnecsuff}.

Thus, thanks to the monotonicity of $\Si_D$ and $\tilde\Si_d$ with respect to $c$, $c^*_m$ can be constructed, as in Section~\ref{1subsection:c^*_2}, as the smallest value of $c>0$ for which $\Si_D(c)$ and $\tilde{\Si}_d(c)$ intersect, being tangent, for $\b>\un\b$. Observe that, if $D\geq d$, the curves that touch, in the first quadrant, for the first time are $\tilde{\a}^-_d$ and $\a^+_D$, thus, since $\tilde\a_d^-(c,\b_0)=0$, the first intersection necessarily occurs for $\b>\b_0>\un\b$, and the tangency is therefore guaranteed. When $D<d$, instead, in order to guarantee the tangency for $\b>\un\b$, we have to restrict to $N\leq 5$ as in Section~\ref{1subsection:c^*_2} (see \eqref{eq:claim} and its proof). 

After these preliminaries, we have that the desired supersolutions exist for all $c\geq c^*_m$, while stationary, generalized subsolutions with compact support of the problem with additional transport term in the $x$ direction can be obtained, for $c<c^*_m$, $c\sim c^*_m$, by considering a truncated problem with $v=0$ for $|y|=L$, $L\gg R$, and adapting the construction of Proposition~\ref{1pr:3.2}. In this way we have proved the last part of Theorem~\ref{2th:section4.1}, which asserts that $c^*_m$ is the asymptotic speed of propagation of problem \eqref{2eq:4.0} in the $x$ direction.

\vspace{0.2cm}
Before concluding with the proof of the remaining qualitative properties of $c^*_m$ stated in Theorem~\ref{2th:section4.2}, we present an auxiliary result in the spirit of Theorem~\ref{1th:3.3}\eqref{1th:3.3.iii}, which gives a sufficient condition for $c^*_m$ to be smaller than the Fisher-KPP speed $c_g$.

\begin{proposition}
	\label{2pr:4.11}
	If
	\begin{equation}
	\label{2eq:6.8bisbis}
	(d-D)c_g-2dD\un\b\leq 0, \quad \text{or} \quad \begin{cases}
	(d-D)c_g-2dD\un\b> 0 \\
	\frac{d}{D}\left(1-DC\right)^2\leq 2+\frac{\r}{g'(0)}-2DC,
	\end{cases}
	\end{equation}
	where $C:=\frac{\un\beta}{\sqrt{Dg'(0)}}$,
	then $c^*_m<c_g$.
	
	\begin{proof}
		Since $\tilde\a_d^-(c,\b_0)=0$ for every $c>0$, and $\a_D^-(c_g,\b)$ is a straight line passing through the points $\left(0,\frac{c_g}{2D}\right)$ and $\left(\sqrt{\frac{g'(0)}{D}},0\right)$, we have that, if the uppermost point of $\tilde\Si_d(c_g)$, that lies at $\b=\un\b$, lies not below $\a_D^-(c_g,\cdot)$, which is equivalent to
		\begin{equation}
		\label{2eq:6.11}
		\tilde\a_d^+(c_g,\un\b)=\frac{c_g+\sqrt{c_g^2+4d\r}}{2d}\geq\frac{c_g}{2D}-\un\b=\a_D^-(c_g,\un\b),
		\end{equation}
		then the tangency occurs before $c_g$, giving that $c^*_m<c_g=2\sqrt{Dg'(0)}$. Simple computations show that \eqref{2eq:6.11} is equivalent to \eqref{2eq:6.8bisbis}.
	\end{proof}
\end{proposition}

\begin{proof}[Proof of Theorem~\ref{2th:section4.2}\eqref{2th:section4.2.ii}]
	We start with the existence of $R_0$. By using Proposition~\ref{2pr:4.1ter}, it will suffice to show that
	\begin{equation}
	\label{2eq:6.3}
	\lim_{R\to 0}\b_0(R)=+\infty, \qquad \lim_{R\to+\infty}\b_0(R)=0
	\end{equation}
	and to take $R_0$ as the unique value of $R$ for which
	\begin{equation}
	\label{2eq:6.3bis}
	\b_0^2(R_0)=\frac{g'(0)}{D}.
	\end{equation}
	The second limit in \eqref{2eq:6.3} follows directly from \eqref{2eq:4.8}. Regarding the first one, we obtain from \eqref{2eq:4.8ter} that
	\begin{equation}
	\label{2eq:6.4}
	\b_0^2(R)=\frac{1}{R^2}k_u^2\left(\frac{\k R}{D}\right),
	\end{equation}
	where $k_u:(0,+\infty)\to(0,j_\tau)$ is the inverse function of $h_u$.
	Since $\k\in(0,\mu)$ is non-increasing in $R$, thanks to \eqref{2eq:A.5}, the argument of $k_u$ in \eqref{2eq:6.4} converges linearly to $0$ as $R\to 0$. As a consequence, \eqref{1eq:A.7} implies that $\lim_{R\to 0}\b_0^2(R)R\in(0,+\infty)$, and the desired result follows.
	
	The monotonicity of $c^*_m(R)$ follows as in the proof of Theorem~\ref{1th:3.3}\eqref{1th:3.3.vi}. Moreover, from \eqref{2eq:6.3bis}, we have that $\b_0^2(R)\ua\frac{g'(0)}{D}$ as $R\da R_0$, thus, the points where $\tilde\a_d^-$ and $\a_D^-$ vanish become arbitrarily close when $R\da R_0$. This, together with the monotonicities of these functions with respect to $\b$, implies that, if $\lim_{R\da R_0} c^*_m(R)$ was positive, the curves $\Si_D(c^*_m(R))$ and $\tilde\Si_d(c^*_m(R))$ would be secant for $R\sim R_0$, against the construction of $c^*_m$.
	
	The proof of the limit for $R\to+\infty$ also follows along the same lines of the one of Theorem~\ref{1th:3.3}\eqref{1th:3.3.vi}, with the only difference, here, that $\tilde\Si_d(c)$ converges, in the first quadrant of the $(\b,\a)$ plane, to the vertical segment $\{0\}\times I_{d,m}(c)$, with $I_{d,m}(c)$ as in \eqref{2eq:1.11bis}.
\end{proof}

\begin{proof}[Proof of Theorem~\ref{2th:section4.2}\eqref{2th:section4.2.iii}]
	To prove the first part of the result, we use the equivalent formulation of \eqref{2eq:condnecsuff} given by \eqref{2eq:5.1}. For this reason, we study the function $D\ov\b(D)^2$ in order to compare it with $g'(0)$. As a first step, we prove that such a function is increasing. Namely, from \eqref{1eq:3.ovb}, we have
	\begin{equation}
		\label{2eq:6.6}
		\ov\b(D)=\frac{1}{R} k_u\left(\frac{\mu R}{D}\right),
	\end{equation}
	thus
	\[\left(D\ov\b^2(D)\right)'=\frac{1}{R^2}k_u\left(\frac{\mu R}{D}\right)\left(k_u\left(\frac{\mu R}{D}\right)-2\frac{\mu R}{D}k_u'\left(\frac{\mu R}{D}\right)\right),\]
	and $D\ov\b^2(D)$ is increasing if and only if $2sk_u'(s)<k_u(s)$ for every $s>0$, which, by setting $s=h_u(r)$, is equivalent to show that, for all $r\in(0,j_\tau)$,
	\begin{equation}
		\label{2eq:6.7}
		2<r\frac{h_u'(r)}{h_u(r)}=2+r\left(\frac{J_{\tau+1}(r)}{J_{\tau}(r)}-\frac{J_{\tau+2}(r)}{J_{\tau+1}(r)}\right),
	\end{equation}
	where, to get the last relation, we have used the definition of $h_u$ given in \eqref{1eq:A.3}, and \eqref{1eq:A.2}. To conclude, we observe that \eqref{2eq:6.7} holds true since the last addend in the right-hand side is positive by \eqref{1eq:A.2bis}. Then, by using \eqref{2eq:6.6} and that $k_u$ is bounded and satisfies \eqref{1eq:A.7}, we have
	\[
	\lim_{D\to 0}D\ov\b(D)^2=0, \qquad \lim_{D\to +\infty}D\ov\b(D)^2=\frac{\mu (N-1)}{R}.
	\]
	We now assume \eqref{2eq:6.2} and distinguish two cases: 
	\[
	(a) \quad
	\frac{\mu(N-1)}{R\, g'(0)}\leq 1 \qquad \text{ or } \qquad (b) \quad 1<\frac{\mu(N-1)}{R\, g'(0)}\leq 1+\frac{\nu}{\sqrt{d\r}}\frac{K_{\tau}\left(\sqrt{\frac{\rho}{d}}R\right)}{K_{\tau+1}\left(\sqrt{\frac{\rho}{d}}R\right)}.
	\]
	In case $(a)$, the properties proved above ensure that $D\ov\b^2(D)<g'(0)$ for all $D>0$, and, thus,  \eqref{2eq:5.1} also holds true for all $D>0$, as we wanted to prove.
	
	In case $(b)$, instead, there exists $D_1>0$ such that
	\begin{equation}
	\label{2eq:D1}
	D\ov\b^2(D)< g'(0) \quad \text{ if $D< D_1$}, \qquad D_1\ov\b^2(D_1)= g'(0), \qquad D\ov\b^2(D)> g'(0) \quad \text{ if $D> D_1$}.
	\end{equation}
	In this case, for $D>D_1$ we have to study the behavior of $\chi_u\left(\sqrt{\frac{g'(0)}{D}}\right)$ in order to compare this quantity with $\chi_v\left(\sqrt{\frac{\r}{d}}\right)$. We have that
	\begin{equation}
	\label{2eq:chiuD}
	D\mapsto\chi_u\left(\sqrt{\frac{g'(0)}{D}}\right)=\frac{\nu g'(0)R \frac{J_{\tau+1}(s(D))}{s(D)J_{\tau}(s(D))}}{\mu-g'(0)R \frac{J_{\tau+1}(s(D))}{s(D)J_{\tau}(s(D))}},
	\end{equation}
	where we have set $s(D):=\sqrt{\frac{g'(0)}{D}}R$, is decreasing for $D>D_1$. Indeed, the right-hand side in the last expression is decreasing in $D$, since so is $s(D)$ and thanks to Lemma~\ref{le:A.1} applied with $\e=-\a=1$. Moreover, \eqref{2eq:D1} and \eqref{1eq:A.6} imply that
	\[
	\lim_{D\da D_1} \chi_u\left(\sqrt{\frac{g'(0)}{D}}\right)=+\infty \qquad \text{ and } \qquad \lim_{D\to +\infty} \chi_u\left(\sqrt{\frac{g'(0)}{D}}\right)=\frac{\nu g'(0)R}{\mu(N-1)-g'(0)R}.
	\]
	Thanks to assumption \eqref{2eq:6.2} and the monotonicity of the function \eqref{2eq:chiuD}, the right limit in the previous relation implies $\chi_u\left(\sqrt{\frac{g'(0)}{D}}\right)>\chi_v\left(\sqrt{\frac{\r}{d}}\right)$ for all $D>D_1$, as we wanted.
	
	Once that we know that, under assumption \eqref{2eq:6.2},   \eqref{2eq:condnecsuff} is satisfied for every $D>0$, we study the behavior of $c^*_m(D)$ for large $D$. First of all, we observe that, for $D\geq d$, condition \eqref{2eq:6.8bisbis} is satisfied, since $(d-D)c_g-2dD\un\b\leq(d-D)c_g\leq 0$, thus Proposition \ref{2pr:4.11} ensures that $c^*_m(D)<c_g$.
	
	We now consider the case in which strict inequality holds in \eqref{2eq:6.2}, and
	we reason as in the proof of Theorem~\ref{1th:3.3}\eqref{1th:3.3.iv}: if we assume by contradiction that $c^*_m(D)$ is bounded, we obtain that the tangency would occur between $\tilde\a_d^-$ and $\a_D^+$. Moreover, by combining \eqref{2eq:6.4} and \eqref{1eq:A.7}, we obtain, for $D\sim+\infty$,
	\begin{equation}
	\label{2eq:6.14}
	\b_0^2(D)=\frac{\k(N-1)}{R}\frac{1}{D}-\frac{N-1}{N+1}\frac{\k^2}{D^2}+o\left(D^{-2}\right),
	\end{equation}
	thus, since we are assuming \eqref{2eq:6.2} with strict inequality, we have $\lim_{D\to\infty}D\b_0^2(D)<g'(0)$, and the curves $\tilde\a_d^-\left(c^*_m(D),\frac{\b'}{\sqrt{D}}\right)$ and $\a_D^+\left(c^*_m(D),\frac{\b'}{\sqrt{D}}\right)$ would not intersect for large $D$ (see the details in the proof of Theorem~\ref{1th:3.3}\eqref{1th:3.3.iv} and observe that we have performed the change of variable $\b'=\b\sqrt{D}$ as in that proof). This contradicts the construction of $c^*_m(D)$ and shows that the speed is unbounded as $D\to+\infty$.	To conclude this case, we observe that the rate of convergence to $+\infty$ again follows as in the proof of Theorem~\ref{1th:3.3}\eqref{1th:3.3.iv}.
	
	If equality holds in \eqref{2eq:6.2}, for large $D$ and every fixed $c>0$, we have $\frac{c}{2D}<\frac{c}{2d}$. Thus, for $c=c^*_m$, the tangency occurs between $\a_D^+$ and $\tilde\a_d^-$ and $\hat\b(c^*_m(D),D)>\b_0(D)$, because otherwise the curves would be secant. This last relation and the definition of $\hat\b$ given in \eqref{1eq:3.13bis} yield to
	\begin{linenomath*}
	\begin{equation*}
	c^{*2}_m(D)=4D(g'(0)-D\hat\b^2(c^*_m(D),D))<4D\left(g'(0)-D\b_0^2(D)\right).
	\end{equation*}
\end{linenomath*}
	The desired result now follows by taking the $\limsup$ for $D\to+\infty$ in the last relation, using \eqref{2eq:6.14} and the fact that now, since we are assuming equality in \eqref{2eq:6.2}, we have $\frac{\k(N-1)}{R}=g'(0)$.
	
	Finally, if \eqref{2eq:6.2} fails, by reasoning as in the first part of the proof, we have that there exists $D_0>D_1$ such that
	\begin{equation}
	\label{2eq:6.8bis}
	\chi_u\left(\sqrt{\frac{g'(0)}{D_0}}\right)=\chi_v\left(\sqrt{\frac{\r}{d}}\right)
	\end{equation}
	and the set of $D$'s satisfying $D>D_1$ for which relation $\chi_u\left(\sqrt{\frac{g'(0)}{D}}\right)>\chi_v\left(\sqrt{\frac{\r}{d}}\right)$ holds true is $\{D_1<D<D_0\}$. As a consequence, \eqref{2eq:5.1} and, thus, \eqref{2eq:condnecsuff} hold true if and only if~$D<D_0$. To conclude the proof, we observe that the limit of $c^*_m(D)$ as~$D\ua D_0$ follows by reasoning  as for the proof of the case $R\da R_0$ in Theorem~\ref{2th:section4.2}\eqref{2th:section4.2.ii}, since \eqref{2eq:5.1} and \eqref{2eq:6.8bis} give that $\b_0^2(D)\ua\frac{g'(0)}{D_0}$ as $D\ua D_0$.
	\end{proof}

\begin{remark}
Theorem~\ref{2th:section4.2}\eqref{2th:section4.2.iii} implies that, when \eqref{2eq:6.2} does not hold, $c^*_m(D)$ is decreasing for $D<D_0$, $D\sim D_0$, giving another example, in addition to the one commented in Remark~\ref{1re:4.2}\eqref{1re:4.2ii} for problem \eqref{1eq:1.1}, in which the asymptotic speed of propagation is not increasing with respect to $D$.
\end{remark}

\setcounter{equation}{0}
\begin{appendices}
\renewcommand{\theequation}{A.\arabic{equation}}

\section{Some properties of Bessel functions}
\label{sectionappA}
In this Appendix we gather several facts on Bessel functions of the first kind and modified Bessel functions of the first and second kind that we use in this work. When no reference is explicitly given, we send the interested reader to the complete monograph \cite{W}. The proofs of some specific properties we were not able to find in the literature
are presented below, for the sake of completeness and also because we think they can be of independent interest.

The Bessel function of the first kind $J_\tau(r)$, $\tau>-1$, is a solution which is defined for $r\geq 0$ of the differential equation
\begin{equation}
	\label{1eq:A.-1}
	r^2 w''(r)+rw'(r)+(r^2-\tau^2)w(r)=0.
\end{equation}
It has a first positive zero which will be denoted by $j_\tau$, it is positive for $r\in(0,j_\tau)$, and we have (see \cite[p.~508, (3)]{W}) that $j_\tau$ is increasing with respect to $\tau$.

The following asymptotic expansion  holds (see \cite[pp.~40, 41]{W})
\begin{equation}
	\label{1eq:A.1}
	J_{\tau}(r)=\left(\frac{r}{2}\right)^\tau\left(\frac{1}{\G(\tau+1)}-\frac{r^2}{4\G(\tau+2)}+O(r^4)\right) \quad \text{ for $r\sim 0$}, 
\end{equation}
where $\G$ is the Gamma function, as well as the relation for the derivative (see \cite[p.~45]{W})
\begin{equation}
	\label{1eq:A.2}
	(r^{-\tau}J_\tau(r))'=-r^{-\tau}J_{\tau+1}(r). 
\end{equation}
Moreover, we have the following monotonicity property for quantities involving the quotient of Bessel functions of the first kind.
\begin{lemma}
	\label{le:A.1}
	For every $\e>0$, $\a\geq -\e$ and $\tau>-1$, the function $r \mapsto r^\a\frac{J_{\tau+\e}(r)}{J_{\tau}(r)}$	is increasing for $r\in(0,j_{\tau})$.
	\begin{proof}
		By using \eqref{1eq:A.2}, we obtain
		\begin{linenomath*}
		\begin{multline*}\left(r^\a\frac{J_{\tau+\e}(r)}{J_{\tau}(r)}\right)'=\left(r^{\a+\e}\frac{r^{-(\tau+\e)}J_{\tau+\e}(r)}{r^{-\tau}J_{\tau}(r)}\right)'=(\a+\e)r^{\a-1}\frac{J_{\tau+\e}(r)}{J_{\tau}(r)}+ \\
			+r^{\a}\frac{J_{\tau+\e}(r)J_{\tau+1}(r)-J_{\tau+\e+1}(r)J_{\tau}(r)}{J_{\tau}^2(r)},
		\end{multline*}
	\end{linenomath*}
		which, in the considered range of $r$, is the sum of non-negative terms, since $J_\tau$ and $J_{\tau+\e}$ are positive by the monotonicity of $j_\tau$ with respect to $\tau$,   and since, as shown in 
		\cite[Remark 2, p.~289]{B}, the function 		
		\begin{equation}
		\label{1eq:A.2bis}
		\tau\mapsto\frac{J_{\tau+1}}{J_{\tau}} \text{ is decreasing for $r\in(0,j_\tau)$ and $\tau>-1$,}
		\end{equation}
		thus $\frac{J_{\tau+1}(r)}{J_{\tau}(r)}>\frac{J_{\tau+\e+1}(r)}{J_{\tau+\e}(r)}$.
	\end{proof}
	\end{lemma}
Thanks to the previous lemma, the function
\begin{equation}
		\label{1eq:A.3}
	\begin{array}{rl}
	h_u(r):(0,j_\tau)&\to(0,\infty) \\
	r&\mapsto \displaystyle{r\frac{J_{\tau+1}(r)}{J_{\tau}(r)}}
	\end{array}
	\end{equation}
is increasing, invertible and, by \eqref{1eq:A.1}, it satisfies
\begin{equation}
	\label{1eq:A.6}
		h_u(r)=\frac{r^2}{2(\tau+1)}\left(1+\frac{r^2}{4(\tau+1)(\tau+2)}+o\left(r^2\right)\right) \qquad \text{for $r\sim 0$},
\end{equation}
while its inverse, which will be denoted by $k_u(r)$ and is also increasing, thanks to the previous relation satisfies
\begin{equation}
	\label{1eq:A.7}
	k_u^2(r)=2(\tau+1)r-\frac{\tau+1}{\tau+2}r^2+o\left(r^2\right) \qquad \text{ for $r\sim 0$}.
\end{equation}

Another monotonicity property related to quotients of Bessel functions of the first kind is given in the following proposition.

\begin{proposition}
	\label{pr:A.3}
	For $\tau>-1$ and $r\in(0,j_{\tau})$, the function $\frac{h_u'(r)}{r}$ is increasing.
	\begin{proof}
		Using \eqref{1eq:A.2}, direct computations give
		\begin{linenomath*}
		\begin{gather*}
			\frac{h_u'(r)}{r}=2\frac{J_{\tau+1}}{rJ_{\tau}}+\frac{J_{\tau+1}^2-J_{\tau+2}J_{\tau}}{J_{\tau}^2}=2\frac{J_{\tau+1}}{rJ_{\tau}}+\sum_{n\in\N}4(\tau+2+2n)\left(\frac{J_{\tau+2+2n}}{rJ_{\tau}}\right)^2,
			\end{gather*}
		\end{linenomath*}
			where the last equality comes from Lommel's formula (see \cite[p.~152]{W}). We then conclude by observing that the first summand in the right-hand side is increasing thanks to Lemma~\ref{le:A.1} applied with $\e=-\a=1$, and the same occurs for each term of the series thanks to Lemma~\ref{le:A.1} applied with $\e=2+2n$ and $\a=-1$.
		\end{proof}
	\end{proposition}

The last property that we recall is the following asymptotic expansion for large orders 	(see \cite[$\S$ 9.3.1]{AS}):
\begin{equation}
	\label{1eq:A.8}
		J_{\tau}(r)\sim\frac{1}{\sqrt{2\pi\tau}}\left(\frac{e\, r}{2\tau}\right)^{\tau} \quad \text{as $\tau\to+\infty$.}
	\end{equation}

Passing to the modified Bessel function of the second kind $K_{\tau}(r)$, it is a  solution of the differential equation 
\begin{equation}
	\label{2eq:A.1}
	r^2 w''(r)+rw'(r)-(r^2+\tau^2)w(r)=0
\end{equation}
which is positive and defined for all $r>0$, and decays to $0$ at $+\infty$, since (see \cite[p.~202]{W})
\begin{equation}
	\label{2eq:A.2}
	\lim_{r\to+\infty}K_{\tau}(r)\left(\frac{\pi}{2r} \right)^{-\frac{1}{2}}e^{r}=1.
	\end{equation}
On the other hand, the following asymptotic expansions, where we take $N\in\N$, $N\geq 2$ and $\tau=\frac{N-3}{2}$, hold for $r>0$, $r\sim 0$  (see \cite[p.~80]{W}):
\begin{equation}
	\label{2eq:A.3}
	K_{\tau}(r)\sim\begin{cases}
		r^{-\frac{1}{2}}\sqrt{\frac{\pi}{2}}\left(1-r+\frac{r^2}{2}+o(r^2) \right) & \text{if $N=2,4$} \\
		C-\log(r)+\frac{r^2}{4}\left(1+C-\log(r) \right)+o(r^3) & \text{if $N=3$} \\
		r^{-1}\left(1-\frac{r^2}{4}\left(1+2C-2\log(r)\right)  +o(r^3)\right)& \text{if $N=5$} \\
		r^{-\tau}\G(\tau)2^{\tau-1}\left(1-\frac{r^2}{4(\tau-1)}+o(r^2) \right) & \text{if $N\geq 6$},
	\end{cases}
\end{equation}
where $C$ is a positive constant. In addition, we have the following relation (see \cite[p.~79]{W})
\begin{equation}
	\label{2eq:A.4}
	(r^{-\tau}K_\tau(r))'=-r^{-\tau}K_{\tau+1}(r),
\end{equation}
  and, thanks to \cite[p.~79]{W} and \cite[Lemma 2.4]{IM} respectively, we have, for all $r>0$,
  \begin{equation}
  \label{2eq:A.5}
  \frac{K_{1/2}(r)}{K_{-1/2}(r)}=1,  \qquad r\mapsto\frac{K_{\tau+1}(r)}{K_\tau(r)} \text{ is decreasing for $\tau>-1/2$}.
  \end{equation}
Moreover, 
 \begin{equation}
\label{2eq:A.6}
h_v(r):= \displaystyle{r\frac{K_{\tau+1}(r)}{K_{\tau}(r)}} \text{ is increasing for $r>0$},
\end{equation}
 since $K'_{\tau}(r)=-\frac{rK_{\tau-1}(r)+\tau K_{\tau}(r)}{r}$ (see \cite[p.~79]{W}), thus
 \begin{linenomath*}
\begin{equation*}
\left(r\frac{K_{\tau+1}(r)}{K_{\tau}(r)}\right)'=r\frac{K_{\tau-1}(r)K_{\tau+1}(r)-K_{\tau}(r)^2}{K_{\tau}(r)^2},
\end{equation*}
\end{linenomath*}
and the numerator in the last expression is positive (see \cite[(1.7)]{IM}).

By using \eqref{2eq:A.3}, and recalling that $\G\left(\frac{3}{2}\right)=\frac{\sqrt{\pi}}{2}$, we have that
\begin{equation}
\label{2eq:A.7}
\lim_{r\da0}h_v(r)=\max\{0,N-3\}.
\end{equation}

Finally, we briefly recall some properties of the modified Bessel function of the first kind $I_{\tau}(r)$, which is a solution of \eqref{2eq:A.1} that is linearly independent with respect to $K_{\tau}(r)$, is positive for all $r>0$, increasing, since
\begin{equation}
	\label{3eq:A.2}
	(r^{-\tau}I_\tau(r))'=r^{-\tau}I_{\tau+1}(r),
\end{equation}
(see \cite[p.~79]{W}), and it satisfies (see \cite[p.~203]{W})
\begin{equation}
	\label{3eq:A.3}
	\lim_{r\to+\infty}I_{\tau}(r)=+\infty.
	\end{equation}

\setcounter{equation}{0}
\renewcommand{\theequation}{B.\arabic{equation}}

\section{Well-posedness of Cauchy problems}
\label{sectionappB}
In our proofs and constructions throughout this whole work, we use the fact that the parabolic equations
involved, complemented with non-negative, bounded, uniformly continuous initial data, 
have a unique solution globally defined in time, which is regular up to the boundary for positive times. 
Due to the presence of the transmission condition of Robin type through the common boundary of the adjacent domains, 
the well-posedness for such problems is non-standard and we give here some details on this point.

The uniqueness of the solution relies on some comparison principles that we repeatedly use also to study the asymptotic behavior of our systems as $t\to+\infty$.
As mentioned in  Section \ref{section2.1}, such comparison principles can be obtained by adapting to the
current geometric setting 
the proofs of those given in \cite[Section 3]{BRR1}. Moreover, the same comparison principles allow us to obtain the bounds that guarantee the global definition.

The existence and regularity up to the boundary, instead, can be obtained using the 
same arguments as in \cite[Appendix A]{BRR1} and \cite[Section 2]{BRR2}: starting from an initial datum $(u_0,v_0)$ 
which satisfies the transmission condition at the interface, one iteratively constructs two monotone sequences $(u_n)_{n\in\N}$,
$(v_n)_{n\in\N}$, 
where $u_n$ solves the reaction-diffusion equation for the $u$ component, together with the Robin boundary condition 
with datum $v=v_{n-1}$, then $v_n$ solves the reaction-diffusion problem with boundary datum $u=u_n$, 
and so on. 
We show below how to handle the situation in which the initial datum does not satisfy the transmission condition at the common boundary of the adjacent domains.

First of all, we take our initial datum $(u_0,v_0)$, which is bounded and uniformly continuous up to the boundary of the respective domains of definition, extend it to the whole $\R^N$ maintaining the uniform continuity, and then approximate it with pairs of smooth functions $(u_0^\e,v_0^\e)$, $\e>0$,  satisfying
\begin{linenomath*}
	\begin{equation*}
\|u_0^\e-u_0\|_{L^\infty(\O)}<\e,\qquad
\|v_0^\e-v_0\|_{L^\infty(\R^N\setminus\O)}<\e.
\end{equation*}
\end{linenomath*}

Next, in order to get the compatibility condition at the boundary, we consider a smooth function
$\chi:\R\to\R$ satisfying 
\begin{linenomath*}
\begin{equation*}
\chi(0)=0,\qquad \chi'(0)=1,\qquad \supp\chi\subset(-R,R),
\end{equation*}
\end{linenomath*}
and we define, using cylindrical coordinates (here $x$ denotes the axis variable and $\rho$ the radial one and $\s$ the angular ones), the functions
\begin{linenomath*}
	\begin{align*}
\tilde u_0^\e(x,\rho,\sigma)&:=u_0^\e(x,\rho,\sigma)+
\frac{h(\e)} D\chi\left(\frac{\rho-R}{h(\e)}\right)
\left(\nu v_0^\e-\mu u_0^\e-D\, \p_\rho u_0^\e\right)(x, R,\sigma),\\
\tilde v_0^\e(x,\rho,\sigma)&:=v_0^\e(x,\rho,\sigma)+
\frac{h(\e)}d\chi\left(\frac{\rho-R}{h(\e)}\right)
\left(-\mu u_0^\e+\nu v_0^\e-d\, \p_\rho v_0^\e\right)(x,R,\sigma),
\end{align*}
\end{linenomath*}
where $h$ is a positive, smooth function such that
\begin{equation}
\label{eq:B.1}
h(\e)\leq\e \qquad \text{ and } \qquad h(\e)\max\left\{\sup_{\overline{\O}}|\p_\r u_0^\e|,\sup_{\R^N\setminus\O}|\p_\r v_0^\e|\right\}\to 0, \qquad \text{as $\e\to0$.}
\end{equation}
These functions match $(u_0^\e,v_0^\e)$ on $\partial\O$, since $\rho=R$ there,
and direct computation shows that they fulfill the compatibility conditions
\begin{linenomath*}
	\begin{equation*}
		D\, \p_n \tilde u_0^\e=\nu \tilde v_0^\e-\mu \tilde u_0^\e,\qquad
-d\, \p_n \tilde v_0^\e=\mu \tilde u_0^\e-\nu \tilde v_0^\e,
\qquad\text{on $\partial\O$}.
\end{equation*}
\end{linenomath*}
Finally, there holds that
\begin{linenomath*}
	\begin{equation*}
		\|\tilde u_0^\e-u_0\|_{L^\infty(\O)}<\e+\frac{h_1(\e)}{D}\max|\chi|,\qquad
\|\tilde v_0^\e-v_0\|_{L^\infty(\R^N\setminus\O)}<\e+\frac{h_1(\e)}{d}\max|\chi|,
\end{equation*}
\end{linenomath*}
where
\begin{linenomath*}
	\begin{equation*}
h_1(\e)\!:=\!h(\e)\Big[\nu  \big(\e \!+\! \big\|v_0\big\|_{L^\infty(\R^N\setminus\O)}\big) \!+\! \mu \big(\e \!+\! \big\|u_0\big\|_{L^\infty(\O)}\big) \!+\! D\big\|\p_{\rho}u_0^\e\big\|_{L^\infty(\O)} \!+\! 
d\big\|\p_{\rho}v_0^\e\big\|_{L^\infty(\R^N\setminus\O)}\big], 
\end{equation*}
\end{linenomath*}
and $h_1(\e)$ is Lipschitz continuous and vanishes at $\e=0$, thanks to \eqref{eq:B.1}.

Then, we can apply the arguments of \cite[Appendix A]{BRR1} and \cite[Section 2]{BRR2} to get, for all $\e>0$, the existence of a classical solution $(\tilde u^\e,\tilde v^\e)$ of system \eqref{1eq:1.1}, with initial datum $(\tilde u^\e_0,\tilde v^\e_0)$, which 
additionally satisfies the standard interior and boundary parabolic estimates.

Let us now consider any subsequence of the family of solutions $(\tilde u^\e,\tilde v^\e)_{\e>0}$ as $\e\to 0$. 
For $0<\e'<\e$, the function $w=\tilde u^{\e}-\tilde u^{\e'}$ (resp.~$z=\tilde v^{\e}-\tilde v^{\e'}$)
satisfies a linear equation with bounded coefficients. Then, comparing with  $C \e e^{\beta t}(\nu,\mu)$, which is a supersolution of such linear equation for sufficiently large $\b$ and lies above the 
initial datum $\tilde u^\e_0-\tilde u^{\e'}_0$ (resp.~$\tilde v^{\e}_0-\tilde v^{\e'}_0$) 
for sufficiently large $C$, one deduces that any subsequence of the family $(\tilde u^\e)_{\e>0}$ (resp.~$(\tilde v^\e)_{\e>0}$)
 is a Cauchy sequence in $L^\infty(\O\times[0,T])$ (resp.~$L^\infty((\R^N\setminus\O)\times[0,T])$) for any $T>0$.
Taking the limit as $\e\to 0$, we obtain the existence of a pair $(u,v)$ which is a solution of \eqref{1eq:1.1} 
which fulfills the initial datum $(u_0,v_0)$ in a classical way 
and satisfies analogous estimates up to the boundary and the desired regularity properties in any time interval $[t_0,T]$, with $T>t_0>0$ (see, e.g., \cite[Appendix A]{BRR1} for further details).

\end{appendices}

\section*{Acknowledgments}
This work has been supported by:
\begin{itemize}
	\item the European Research
	Council under the European Union's Seventh Framework Programme (FP/2007-2013) /
	ERC Grant Agreement number 321186 - ReaDi \lq\lq Reaction-Diffusion E\-qua\-tions, Propagation and
	Modelling" led by Henri Berestycki,
	\item the Institute of Complex Systems of Paris \^Ile-de-France under DIM 2014 \lq\lq Diffusion heterogeneities in different spatial dimensions and applications to ecology and medicine",
	\item the Ministry of Science, Innovation and Universities of the Government of Spain through projects MTM2015-65899-P, PGC2018-097104-B-100 and contract Juan de la Cierva Incorporaci\'on IJCI-2015-25084, 
	\item the \lq\lq Departamento de Matem\'atica Aplicada a la Ingenier\'ia Industrial\rq\rq of the Universidad Polit\'ecnica de Madrid.
\end{itemize}


\begin{thebibliography}{10}
	%
	\bibitem{AS}
	M. Abramowitz, I.A. Stegun, Handbook of Mathematical Functions with Formulas, Graphs, and Mathematical Tables, National Bureau of Standards Applied Mathematics Series, vol. 55, Tenth Printing, Washington, 1972. 
	
	\bibitem{AW}
	D.G. Aronson, H.F. Weinberger, \emph{Multidimensional nonlinear diffusion arising in population
	              genetics}, Adv. in Math. \textbf{30} (1978), 33--76.
	              

	
		\bibitem{B}  \'A. Baricz, \emph{Functional inequalities involving Bessel and modified Bessel functions of the first kind}, Expo. Math. \textbf{26} (2008), 279--293.
	

	
	\bibitem{BHR} H. Berestycki, F. Hamel, L. Roques, \emph{Analysis of the periodically fragmented environment model.
		{I}. {S}pecies persistence},  J. Math. Biol. \textbf{51} (2005), 75--113.
	
	\bibitem{BN} H. Berestycki, G. Nadin, \emph{Asymptotic  spreading  for  general  heterogeneous  Fisher-KPP type equations}, Memoirs of the American Mathematical Society, American Mathematical Society, To appear.
	\href{https://hal.ird.fr/LJLL/hal-01171334}{hal-01171334v4}
	
	\bibitem{BRR1}
	H. Berestycki, J.-M. Roquejoffre, L. Rossi, \emph{The influence of a line with fast diffusion in Fisher-KPP propagation}, J. Math. Biol. \textbf{66} (2013), 743--766.
	%
	\bibitem{BRR2}
	H. Berestycki, J.-M. Roquejoffre, L. Rossi, \emph{Fisher-KPP propagation in the presence of a line: further effects}, Nonlinearity \textbf{29} (2013), 2623--2640.
	%
	\bibitem{BRR3}
	H. Berestycki, J.-M. Roquejoffre, L. Rossi, \emph{The shape of expansion induced by a line with fast diffusion in Fisher-KPP equations}, \textbf{343}  Comm. Math. Phys. (2016), 207--232.
	
		\bibitem{BRR4}
	H. Berestycki, J.-M. Roquejoffre, L. Rossi, \emph{Travelling waves, spreading and extinction for {F}isher-{KPP}
		propagation driven by a line with fast diffusion}, Nonlinear Anal. \textbf{137} (2016), 171--189.
	
	

	
	\bibitem{dCL}
	T. de Camino-Beck, M.A. Lewis, \emph{Invasion with stage-structured coupled map lattices: Application to the spread of scentless
	chamomile}, Ecol. Modell. \textbf{220} (2009), 3394--3403.
	
	\bibitem{Faria}
	N.R. Faria et al., \emph{The early spread and epidemic ignition of HIV-1 in human populations}, Science \textbf{346} (2014), 56--61.
	
	\bibitem{FLT} K. Fellner, E. Latos, B.Q. Tang \emph{Well-posedness and exponential equilibration of a volume-surface reaction-diffusion system with nonlinear boundary coupling}, Ann. Inst. H. Poincar\'{e} Anal. Non Lin\'{e}aire \textbf{35} (2018), 643--673. 
	%
	\bibitem{F}
	R.A. Fisher, \emph{The wave of advantage of advantageous genes}, Annals of Eugenics \textbf{7} (1937), 355--369.
	%
	
	\bibitem{G18}
	L. Girardin, \emph{Non-cooperative Fisher-KPP systems: traveling waves and long-time behavior}, Nonlinearity \textbf{31} (2018), 108--164.
	
	\bibitem{H}
	M. Holzer, \emph{Anomalous spreading in a system of coupled {F}isher-{KPP}
		equations}, Phys. D \textbf{270} (2014), 1--10.
	
	
	\bibitem{IM}
	M.E.H. Ismail, M.E. Muldoon, \emph{Monotonicity of the zeros of a cross-product of Bessel functions}, SIAM J. Math. Anal. \textbf{9} (1978),  759--767.
	
	\bibitem{KPP}
	A.N. Kolmogorov, I.G. Petrovskii, N.S. Piskunov, \emph{\'Etude de l'\'equation de la diffusion avec croissance de la quantit\'e de matière et son application à  un problème biologique}, Bull. Univ. \'Etat Moscou, S\'er. Intern. A \textbf{1} (1937), 1--25.
	
	\bibitem{LW}
	H. Li, X. Wang, \emph{Using effective boundary conditions to model fast diffusion on a road in a large field}, Nonlinearity \textbf{30} (2017), 3853--3894.
	
	\bibitem{L}
	G.M. Lieberman, Second order parabolic differential equations, World Scientific Publishing Co., 1996.
	
	\bibitem{MCV}
	A. Madzvamuse, A.H.W. Chung, C. Venkataraman, \emph{Stability analysis and simulations of coupled bulk-surface
	reaction-diffusion systems}, Proc. A. \textbf{471} (2015), 18 pp. 
	%
	\bibitem{McK}
	H.W. McKenzie, E.H. Merrill, R.J. Spiteri, M.A. Lewis, \emph{How linear features alter predator movement and the functional response}, Interface focus, \textbf{2} (2012), 205--216.
	
	\bibitem{MBC}
	A. Morris, L. B\"orger, E. Crooks, \emph{Individual variability in dispersal and invasion speed}, Mathematics \textbf{7} (2019).
	
	
	\bibitem{R}
	L. Roques, J.-P. Rossi, H. Berestycki, J. Rousselet, J. Garnier, J.-M. Roquejoffre, L. Rossi, S. Soubeyrand, C. Robinet,
	Modeling the spatio-temporal dynamics of the pine processionary moth, in \emph{Processionary Moths and Climate Change: An
	Update}, Springer, 2015.

\bibitem{RTV}
L. Rossi, A. Tellini, E. Valdinoci, \emph{The effect on Fisher-KPP propagation in a cylinder with fast diffusion on the boundary}, SIAM J. Math. Anal. \textbf{49} (2017), 4595--4624.

\bibitem{T16}
A. Tellini, \emph{Propagation speed in a strip bounded by a line with different diffusion}, J. Differential Equations \textbf{260} (2016), 5956--5986.


	\bibitem{W}
	G.N. Watson, A Treatise on the Theory of Bessel Functions, Cambridge University Press, Cambridge, 1944.
	
	\bibitem{WLL}
	H.F. Weinberger, M.A. Lewis, B. Li, \emph{Anomalous spreading speeds of cooperative recursion systems.}, J. Math. Biol. \textbf{55} (2007), 207--222.
\end{thebibliography}
\end{document}